\documentclass[11pt]{elsarticle}

\usepackage{amsmath, amsthm, amssymb}
\usepackage{graphicx}
\usepackage{psfrag}
\usepackage{epsfig}
\usepackage{color}

\numberwithin{equation}{section}

\begin{document}

\begin{frontmatter}

\title{Explicit partitioning strategies for the interaction between a fluid and a multilayered poroelastic structure: an operator-splitting approach}

 \author[MA]{M. Buka\v{c}\corref{cor1}}
 \ead{martinab@pitt.edu}
 \author[ME]{P. Zunino}
 \ead{paz13@pitt.edu}
 \author[MA] {I. Yotov}
 \ead{yotov@math.pitt.edu}
  \cortext[cor1]{Corresponding author (+1 412 624 0177)}

 \address[MA]{Department of Mathematics, University of Pittsburgh, Pittsburgh, PA 15260,  USA}
 \address[ME]{Department of Mechanical Engineering \& Materials Science, University of Pittsburgh, Pittsburgh, PA 15261,  USA}

\begin{abstract}
We study the interaction between an incompressible, viscous, Newtonian fluid and a multilayered structure, which consists of a thin elastic layer and a thick poroelastic material.
The thin layer is modeled using the linearly elastic Koiter membrane model, while the thick poroelastic layer is modeled as a Biot system. The objective of this work is to investigate how 
the poroelastic phenomena affect the characteristic features of blood flow in arteries, such as propagation of pressure waves. We develop a loosely coupled fluid-structure interaction finite element solver 
based on the Lie operator splitting scheme. We prove a conditional stability of the scheme and derive error estimates. Theoretical results are supported with numerical examples.
\end{abstract}


\end{frontmatter}

\def\nn{\boldsymbol{n}}
\def\hnn{\hat{\boldsymbol{n}}}
\def\ttau{\boldsymbol{t}}
\def\httau{\hat{\boldsymbol{t}}}
\def\hGamma{\hat{\Gamma}}
\def\bilinAA{{\mathcal A}}
\def\dt{\partial_\tau}
\def\dtt{d_\tau}
\def\interf{\Gamma^n}

\def\domf{\Omega_f}
\def\hdomf{\hat{\Omega}_f}
\def\velf{\boldsymbol{v}_h}
\def\hvelf{\hat{\boldsymbol{v}}_h}
\def\bff{\boldsymbol{\varphi}_{f,h}}
\def\hbff{\hat{\boldsymbol{\varphi}}_{f,h}}
\def\strf{\boldsymbol{\sigma}_{f,h}}
\def\hstrf{\hat{\boldsymbol{\sigma}}_{f,h}}
\def\strfv{\boldsymbol{D}_{f,h}}
\def\strfp{\pf}
\def\pf{p_{f,h}}
\def\hpf{\hat{p}_{f,h}}
\def\bpf{\psi_f}
\def\hbpf{\hat{\psi}_{f,h}}

\def\domp{\Omega_p}
\def\hdomp{\hat{\Omega}_s}
\def\velp{\boldsymbol{q}_h}
\def\hvelp{\hat{\boldsymbol{q}}_h}
\def\bfp{{\boldsymbol{r}}_h}
\def\hbfp{\hat{\boldsymbol{r}}}
\def\pp{p_p}
\def\hpp{\hat{p}_{p,h}}
\def\bpp{\psi_{p,h}}
\def\hbpp{\hat{\psi}_{p,h}}
\def\velr{\boldsymbol{w}_h}
\def\hvelr{\hat{\boldsymbol{w}}_h}

\def\strp{\boldsymbol{\sigma}_{s,h}}
\def\hstrp{\hat{\boldsymbol{\sigma}}_{s,h}}
\def\dispp{\boldsymbol{U}_h}
\def\hdispp{\hat{\boldsymbol{U}}_h}
\def\bsp{\boldsymbol{\varphi}_{s,h}}
\def\hbsp{\hat{\boldsymbol{\varphi}}_{s,h}}
\def\strep{\boldsymbol{\sigma}_{E,h}}
\def\hstrep{\hat{\boldsymbol{\sigma}}_{E,h}}

\def\dispm{\boldsymbol{\eta}_h}
\def\hdispm{\hat{\boldsymbol{\eta}}_h}
\def\bm{\boldsymbol{\xi}_h}
\def\hbm{\hat{\boldsymbol{\xi}}_h}

\def\femvf{\mathbf{V}_h^f}
\def\femvp{\mathbf{V}_h^p}
\def\fempf{Q_h^f}
\def\fempp{Q_h^p}
\def\hfems{\hat{\mathbf{X}}_h^s}
\def\hfemm{\hat{\mathbf{X}}_h^m}

\def\bilinasp{\hat{a}_{s}}
\def\bilinbsp{\hat{b}_{s}}
\def\bilinbspnoh{b_{s}}
\def\bilinm{\hat{a}_m}
\def\bilinaff{{a}_{f}^n}
\def\bilinbff{{b}_{f}^n}
\def\bilinafp{{a}_{p}^n}
\def\bilinbfp{{b}_{p}^n}
\def\bilincfp{{c}_{p}^n}
\def\stabff{{s}_{f}^n}

\def\energyf{E_{f,h}}
\def\energys{\hat{E}_{s,h}}
\def\energym{\hat{E}_{m,h}}

\newtheorem{theorem}{Theorem}
\newtheorem{lemma}{Lemma}
\newtheorem{corollary}{Corollary}
\newtheorem{assumption}{Assumption}

\section{Introduction}
We study the interaction between an incompressible viscous, Newtonian fluid and a multilayered poroelastic structure.
The need of a multilayered description of the wall is suggested by results of \textit{in vivo} measurements of 
arterial wall motion~\cite{cinthio2006longitudinal,cinthio2005evaluation}, which indicate that the inner parts of the vessel wall (intima-media complex) exhibit a larger longitudinal displacement
than the outer part of the vessel wall (adventitia), introducing the presence of substantial shear strain and shear stress within the wall. 
Keeping in mind the general and ambitious aim of improving the understanding of arterial mechanical properties, we also introduce poroelastic effects in FSI simulations. 
The material properties of arteries have been widely studied~\cite{armentano1995arterial,bauer1979separate,fung1972bio,humphrey1995mechanics,vito2003blood,robertson2012structurally,canic2006modeling}. Pseudo-elastic~\cite{zhou1997degree,fung1972bio}, viscoelastic~\cite{armentano1995arterial,canic2006modeling,bauer1979separate} and nonlinear material models represent well known examples. To our knowledge, only a few of them have been deeply analyzed in the time dependent domain, namely when coupled with the pulsation induced by heartbeat. These considerations also apply to poroelasticity, which is addressed here. 
Poroelasticity becomes particularly interesting when looking at the coupling of flow with mass transport. This is a significant potential application of our model, since mass transport provides nourishment, remove wastes, affects pathologies and allows to deliver drugs to arteries~\cite{prosi2005mathematical}.
Poroelastic phenomena are interesting in different applications where soft biological tissues are involved. We mention for example cerebro-spinal flow~\cite{lesinigo2013lumped}, which also involves FSI, the study of hysteresis effects observed in the myocardial tissue~\cite{holzapfel2004anisotropic,holzapfel2009constitutive}, as well as the modeling of lungs as a continuum material~\cite{rausch2011material}. 
Besides biological applications, this model can also be used in numerous other applications: geomechanics, ground-surface water flow, reservoir compaction and surface subsidence, seabed-wave interaction problem, etc.

While there exist many complex and detailed models for mutilayered structures in different applications, the interaction between the fluid and a multilayered structure remains an area of active research. To our knowledge, the only theoretical result was presented in~\cite{BorSunMulti}, where the authors proved existence of a solution to a fluid-two-layered-structure interaction problem, in which one layer is modeled as a thin (visco)elastic shell and the other layer as a linearly elastic structure.
Several studies focused on numerical simulations. An interaction between the fluid and a two-layer anisotropic elastic structure was used in~\cite{tang2011multi} to model the human right and left ventricles.
Slightly different models were used in~\cite{kim2002fluid} to model  fully coupled fluid-structure-soil interaction for cylindrical liquid-contained structures subjected to horizontal ground excitation.
The work in~\cite{botkin2007dispersion} focused on studying velocity of acoustic waves excited in multilayered structures contacting with fluids. 
A fluid-multilayered structure interaction problem coupled with transport was studied in~\cite{chung2012effect}, with the purpose of investigating 
 low-density lipoprotein transport within a multilayered arterial wall.
However, none of these studies present a numerical scheme supported with numerical analysis.

In this work, we propose a model that captures interaction between a fluid and a multilayered structure, which consists of a thin elastic layer and a thick poroelastic layer. We assume that the thin layer represents a homogenized combination of the endothelium, tunica intima, and internal elastic lamina, and that the thick layer represents tunica media.
The thin elastic layer is modeled using the linearly elastic Koiter membrane model, while the poroelastic medium is modeled using the Biot equations. 
The Biot system consists of an elastic skeleton and connecting pores filled with fluid. We assume that the elastic skeleton is homogeneous and isotropic, while the fluid in the pores is modeled using the Darcy equations.
The Biot system is coupled to the fluid and the elastic membrane via the kinematic (no-slip and conservation of mass) and dynamic (conservation of momentum) boundary conditions. 
More precisely, we assume that the elastic membrane can not store fluid, but allows the flow through it in the normal direction. In the tangential direction, we prescribe the no-slip boundary condition. 
This assumption is reasonable in blood flow modeling, since it has been shown in~\cite{lee2008permeability} that predominant direction of intimal transport is the radial direction normal to the endothelial surface, for all ranges of relative
intimal thickness.

The coupling between a fluid and a single layer poroelastic structure has been previously studied in~\cite{badia2009coupling,showalter2010poroelastic,murad2001micromechanical,tully2009coupling}. In particular, the work in~\cite{badia2009coupling} is based on the modeling and a numerical solution of the interaction between an incompressible,
Newtonian fluid, described using the Navier Stokes equations, and a poroelastic structure modeled as a Biot system. The problem was solved using both a monolithic and a partitioned approach. 
The partitioned approach was based on the domain decomposition procedure, with the purpose of solving the Navier-Stokes equations separately from the Biot system. However, sub-iterations were needed between the two problems due to the instabilities associated with the ``added mass effect''. Namely, in fluid-structure interaction problems, the ``classical'' loosely-coupled 
methods have been shown to be unconditionally unstable if the density of the structure is comparable to the density of the fluid~\cite{causin2005added}, which is the case in hemodynamics applications. To resolve this problem, as an alternative to sub-iterations, several different splitting strategies have been proposed~\cite{bukavc2012fluid,badia2009robin,fernandez2011incremental,fernandez2013fully,nobile2008effective,hansbo2005nitsche}. 
In particular, the kinematically coupled $\beta$-scheme proposed in~\cite{bukavc2012fluid} is based on embedding the no-slip kinematic condition into the thin structure equations. Using
the Lie operator splitting approach~\cite{glowinski2003finite}, the structure equations is split so that the structure inertia
is treated together with the fluid as a Robin boundary condition, while the structure elastodynamics is treated separately.
This method has been shown to be unconditionally stable, and therefore independent of the fluid and structure densities~\cite{SunBorMar}.

Motivated by the kinematically coupled $\beta$-scheme, in this manuscript we propose a loosely-coupled finite element scheme based on the Lie operator splitting method. We use the operator splitting to separate the fluid problem (Navier-Stokes equations) from the Biot problem. 
The no-slip kinematic condition in the tangential direction is embedded into the membrane equations. Operator splitting is preformed so that the tangential component of the structure inertia is treated together with the fluid as a Robin boundary condition. 
Assuming the primal formulation for the Darcy equations, the continuity of the normal flux and the balance of normal components of stress between the Navier Stokes fluid and the fluid in the pores is treated in a similar way as in the partitioned algorithms for the Stokes-Darcy coupled problems~\cite{layton2012long,shan2012decoupling}.
The membrane elastodynamics is embedded into the Biot system as a Robin boundary condition. 
In the contrast with domain decomposition methods proposed in~\cite{badia2009coupling}, the operator splitting approach does not require sub-iterations between the fluid and the Biot problem, making our scheme more computationally efficient. 

We prove a conditional stability of the proposed scheme, where the stability condition does not depend on the fluid and structure densities, but it is related to the decoupling of the Stokes-Darcy interaction problem. Furthermore, we derive the error estimates and prove the convergence of the scheme. 
The rates of convergence and the stability condition were validated numerically on a classical benchmark problem typically used to test the results of fluid-structure interaction algorithms. In the second numerical example, we investigate the effects of porosity to the structure displacement. Namely, we distinguish a high storativity and
a high permeability case in the Darcy equations, and compare them to the results obtained using a purely elastic model. Depending on the regime, we observe a significantly different behavior of the coupled system.

The rest of the paper is organized as follows. In the following section we introduce the model equations and the coupling conditions. In Section 3 we propose a loosely-coupled scheme based on the operator-splitting approach. The weak formulation and stability of the scheme is presented in Section 4.
In Section 5 we derive the error analysis of the scheme. Finally, the numerical results are presented in Section 6.

\section{Description of the problem}

Consider a bounded, deformable, two-dimensional domain $\Omega(t) = \Omega^f(t) \cup \Omega^p(t)$ of reference length $L$, which consists of two regions, $\Omega^f(t)$ and $\Omega^p(t)$, see Figure~\ref{fig:domain}. 
We assume that the region $\Omega^f(t)$ has reference width $2R$, and is filled by an incompressible, viscous fluid.
We denote the width of the second region $\Omega^p(t)$ by $r_p$, and assume that $\Omega^p(t)$ is occupied by a fully-saturated poroelastic matrix.
\begin{figure}[ht]
\centering{
\includegraphics[scale=0.6]{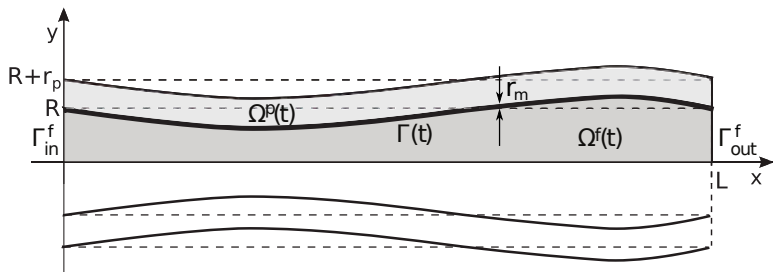}
}
\caption{Deformed domains $\Omega^f(t) \cup \Omega^p(t).$}
\label{fig:domain}
\end{figure}
The two regions are separated by a common interface $\Gamma(t)$. We assume that $\Gamma(t)$ has a mass, and represent a thin, elastic structure. 
Namely, we assume that the thickness of the  interface $r_m$ is ``small'' with respect to the radius of the fluid domain, $r_m << R$. Thus, the volume of the interface is negligible, so it acts as a membrane that can not store fluid, but allows the flow through it in the normal direction. 

We are interested in simulating a pressure-driven flow through the deformable channel with a two-way coupling 
between the fluid, thin elastic interface, and poroleastic structure.
Without loss of generality, we restrict the model to a two-dimensional (2D) geometrical model representing a deformable channel. We consider only the upper half of the fluid domain supplemented by a symmetry condition at the axis of symmetry.
Thus, the reference fluid and structure domains in our problem (showed by dashed lines in Figure~\ref{fig:domain}) are given, respectively, by
\begin{eqnarray*}
\hat{\Omega}^f &:=& \{(x,y) | \ 0<x<L, 0<y<R \}, \\
\hat{\Omega}^p &:=& \{(x,y) | \ 0<x<L, R<y<R+r_p \},
\end{eqnarray*}
and the reference lateral boundary by $\hat{\Gamma} = \{(x,R) | \ 0<x<L \}.$
The inlet and outlet fluid boundaries are defined, respectively, as $\Gamma_{in}^f = \{(0,y) | \ 0<y<R \}$ and
$\Gamma_{out}^f = \{(L,y) | \ 0<y<R \}.$

We model the flow using  the Navier-Stokes equations for a viscous, incompressible, Newtonian fluid:
\begin{align}\label{NS1}
& \rho_f \bigg( \frac{\partial \boldsymbol{v}}{\partial t}+ \boldsymbol{v} \cdot \nabla \boldsymbol{v} \bigg) = \nabla \cdot \boldsymbol\sigma^f + \boldsymbol g & \textrm{in}\; \Omega^f(t) \times (0,T), \\
 \label{NS2}
&\nabla \cdot \boldsymbol{v} = 0   & \textrm{in}\; \Omega^f(t) \times (0,T),
\end{align}
where $\boldsymbol{v}=(v_x,v_y)$ is the fluid velocity, $\boldsymbol\sigma_f= -p_f \boldsymbol{I} + 2 \mu_f \boldsymbol{D}(\boldsymbol{v})$  is the fluid stress tensor, $p_f$ is the fluid pressure, $\rho_f$ is the fluid density, 
$\mu_f$ is the fluid viscosity and  $\boldsymbol{D}(\boldsymbol{v}) = (\nabla \boldsymbol{v}+(\nabla \boldsymbol{v})^{\tau})/2$ is the rate-of-strain tensor.
Denote the inlet and outlet fluid boundaries by $\Gamma^f_{in}$ and $\Gamma^f_{out},$ respectively. At the inlet and outlet boundary we prescribe the normal stress:
\begin{align}
& \boldsymbol {\sigma}^f \boldsymbol {n}_{in} = -p_{in}(t) \boldsymbol{n}_{in} & \textrm{on} \; \Gamma^f_{in} \times (0,T), \label{inlet} \\
& \boldsymbol {\sigma}^f \boldsymbol {n}_{out} = 0  & \textrm{on} \; \Gamma^f_{out} \times (0,T),\label{outlet}
\end{align}
where $\boldsymbol{n}_{in}/\boldsymbol{n}_{out}$ are the outward normals to the inlet/outlet fluid boundaries, respectively.
These boundary conditions are common in blood flow modeling~\cite{badia2008fluid, miller2005computational, nobile2001numerical} even though they are not physiologically optimal since the flow distribution and pressure field in the modeled domain are
often unknown~\cite{Vignon-Clementel20063776}.
Along the middle line of the channel $\Gamma^f_0 = \{(x,0)| \; 0 < x < L\}$ we impose the  symmetry conditions:
\begin{equation}\label{symmetry_condition}
 \frac{\partial v_x}{\partial y} = 0, \quad v_y = 0 \quad \textrm{on} \; \Gamma^f_0 \times (0,T).
\end{equation}


The lateral boundary represents a deformable, thin elastic wall, whose dynamics is modeled by the linearly elastic Koiter membrane model, given in the Lagrangian formulation by:
\begin{align}
& \rho_{m} r_m \frac{\partial^2 \hat\eta_x}{\partial t^2}-C_2 \frac{\partial \hat\eta_y}{\partial \hat{x}}-C_1 \frac{\partial^2 \hat\eta_x}{\partial\hat {x}^2}= \hat {f}_x   & \textrm{on} \; \hat{\Gamma} \times (0,T), \label{structure1}  \\
& \rho_{m} r_m \frac{\partial^2 \hat\eta_y}{\partial t^2}+C_0 \hat\eta_y +C_2 \frac{\partial \hat\eta_x}{\partial \hat{x}}   = \hat {f}_y & \textrm{on} \; \hat{\Gamma}\times (0,T), \label{structure2}
\end{align}
where $\hat{\boldsymbol \eta} (\hat{x},t)= (\hat\eta_x(\hat{x},t), \hat\eta_y(\hat{x},t))$ denotes the axial and radial displacement, $\hat{\boldsymbol f} = (\hat {f}_x, \hat {f}_y)$ is a vector of surface density of the force applied to the membrane, $\rho_{m}$ denotes the membrane density and 
\begin{equation}\label{coeff}
 \begin{array}{rlrlrl}
 C_0 &= \frac{r_m}{R^2} \big(\frac{2 \mu_m \lambda_m}{\lambda_m+2\mu_m}+2\mu_m\big), \; & C_1 &= r_m \big(\frac{2 \mu_m \lambda_m}{\lambda_m+2\mu_m}+2\mu_m\big), & C_2 &=\frac{r_m}{R} \frac{2 \mu_m \lambda_m}{\lambda_m+2\mu_m}.
\end{array}
\end{equation} 
The coefficients $\mu_m$ and $\lambda_m$ are the Lam\'e coefficients for the membrane. Note that we can write the system~\eqref{structure1}-\eqref{structure2} more compactly as
\begin{equation}\label{KoiterCompact}
 \rho_{m}r_m \frac{\partial^2 \hat{\boldsymbol \eta}}{\partial t^2} + \hat{\mathcal{L}} \hat{\boldsymbol \eta} = \hat{\boldsymbol f}, 
\quad
 \hat{\mathcal{L}} := \left( \begin{array}{cc}
 -C_1 \partial_{\hat{x}\hat{x}} & -C_2 \partial_{\hat{x}}   \\
 C_2 \partial_{\hat{x}} & C_0 
\end{array} \right) .
\end{equation}

The fluid domain is bounded by a deformable porous matrix consisting of a skeleton and connecting pores filled with fluid, whose dynamics is described by the Biot model, which in the primal, Eulerian formulation reads as follows:
\begin{align}
& \rho_{p} \frac{D^2 \boldsymbol U}{D t^2} - \nabla \cdot \boldsymbol \sigma^p = \boldsymbol h & \textrm{in} \; \Omega^p(t)\times  (0,T),  \label{B1}\\
& \frac{D}{D t}(s_0 p_p + \alpha \nabla \cdot \boldsymbol U) - \nabla \cdot (\kappa \nabla p_p) = s & \textrm{in} \; \Omega^p(t) \times  (0,T),\label{B3}
\end{align}
where $\frac{D}{Dt}$ denotes the classical concept of material derivative.
The stress tensor of the poroelastic medium is given by
$
\boldsymbol \sigma^p = \boldsymbol \sigma^E - \alpha p_p \boldsymbol I,
$
where $\boldsymbol \sigma^E$ denotes the elasticity stress tensor. With the assumption that the displacement $\boldsymbol U = (U_x, U_y)$ of the skeleton is connected to stress tensor $\boldsymbol \sigma^E$ via the Saint-Venant Kirchhoff elastic model, we have
$
\boldsymbol \sigma^E (\boldsymbol U) = 2 \mu_p  \boldsymbol D (\boldsymbol U) + \lambda_p \textrm{tr}(\boldsymbol D(\boldsymbol U)) \boldsymbol I,
$
where $\lambda_p$ and $\mu_p$ denote the Lam\'e coefficients for the skeleton, and, with the hypothesis of ``small'' deformations, $\boldsymbol D(\boldsymbol U) = (\nabla \boldsymbol U+(\nabla \boldsymbol U)^{T})/2.$ 

System~\eqref{B1}-\eqref{B3} consists of the momentum equation for the balance of total forces~\eqref{B1}, and the storage equation~\eqref{B3} for the fluid mass conservation in the pores of the matrix, where $p_p$ is the fluid pressure. The density of saturated porous medium is denoted by $\rho_{p}$, and $\kappa$ denotes the uniformly positive definite hydraulic
conductivity tensor. For simplicity of the presentation we assume that 
$\kappa$ is a scalar function. The coefficient $c_0>0$ is the storage coefficient, and the Biot-Willis constant $\alpha$ is the pressure-storage coupling coefficient. 
The relative velocity of the fluid within the porous structure $\boldsymbol q$ can be reconstructed via Darcy's law
$$ \boldsymbol q = -\kappa \nabla p_p \quad  \textrm{in} \; \Omega^p(t) \times (0,T).$$

Denote the inlet and outlet poroelastic structure boundaries, respectively, by $\Gamma_{in}^p = \{(0,y)| \; R<y<R+r_p\}$ and $\Gamma_{out}^p = \{(L,y)| \; R<y<R+r_p\}$, and the reference exterior boundary by $\hat{\Gamma}^p_{ext}=\{(x,R+r_p)| \; 0<x<L\}$.
We assume that the poroelastic structure is fixed at the inlet and outlet boundaries:
\begin{equation}\label{homostructure1}
 \boldsymbol U  = 0 \quad \textrm{on} \; \Gamma^p_{in} \cup \Gamma^p_{out} \times (0,T),
\end{equation}
that the external structure boundary $\Gamma^p_{ext}(t)$ is exposed to external ambient pressure 
\begin{eqnarray}
 \boldsymbol {n}_{ext} \cdot \boldsymbol \sigma^E \boldsymbol {n}_{ext} &=&  -p_e \quad \textrm{on} \; \Gamma^p_{ext}(t) \times (0,T),
\end{eqnarray}
where  $\boldsymbol n_{ext}$ is the outward unit normal vector on $\Gamma^p_{ext}(t)$,
and that the tangential displacement of the exterior boundary is zero:
\begin{equation}
U_x =  0 \quad \textrm{on} \; \Gamma^p_{ext}(t) \times (0,T).
\end{equation}
On the fluid pressure in the porous medium, we impose drained boundary conditions~\cite{detournay1993fundamentals}:
\begin{equation}
p_p = 0 \quad \textrm{on} \; \Gamma^p_{ext}(t) \cup \Gamma^p_{in} \cup \Gamma^p_{out}  \times (0,T).
\end{equation}
Initially, the fluid, elastic membrane and the poroelastic structure are assumed to be at rest, with zero displacement from the reference configuration
\begin{equation}\label{initial}
 \boldsymbol{v}=0, \quad \boldsymbol U = 0, \quad \frac{D \boldsymbol U}{D t}=0, \quad \hat{\boldsymbol \eta} = 0, \quad \frac{\partial \hat{\boldsymbol \eta}}{\partial t}=0, \quad \boldsymbol q =0, \quad p_p=0.
 \end{equation}
 
 To deal with the motion of the fluid domain we adopt the Arbitrary Lagrangian-Eulerian (ALE) approach~\cite{hughes1981lagrangian,donea1983arbitrary,nobile2001numerical}.
In the context of finite element method approximation of moving-boundary problems,
ALE method deals efficiently with the deformation of the mesh, especially at the boundary and near the interface
between the fluid and the structure,
and with the issues related to the approximation of the time-derivatives
$\partial \boldsymbol v/\partial t \approx (\boldsymbol {v}(t^{n+1})-\boldsymbol {v}(t^{n}))/\Delta t$ which,
due to the fact that $\Omega^f(t)$ depends on time, is not well defined since
the values $\boldsymbol {v}(t^{n+1})$ and $\boldsymbol {v}(t^{n})$ 
correspond to the values of $\boldsymbol {v}$ defined at two different domains.
Following the ALE approach, we introduce two 
families of (arbitrary, invertible, smooth) mappings ${\cal{A}}_t$ and ${\cal{S}}_t$, defined on 
 reference domains $\hat{\Omega}^f$ and $\hat{\Omega}^p$, respectively, which track the domain in time:
\begin{eqnarray}
& & \mathcal{A}_t : \hat{\Omega}^f  \rightarrow \Omega^f(t) \subset \mathbb{R}^2,  \quad \boldsymbol{x}=\mathcal{A}_t(\hat{\boldsymbol{x}}) \in \Omega^f(t), \; \; \textrm{for} \; \hat{\boldsymbol{x}} \in \hat{\Omega}^f, \label{At}\\
& & \mathcal{S}_t : \hat{\Omega}^p  \rightarrow \Omega^p(t) \subset \mathbb{R}^2,  \quad \boldsymbol{x}=\mathcal{S}_t(\hat{\boldsymbol{x}}) \in \Omega^p(t), \quad \textrm{for} \; \hat{\boldsymbol{x}} \in \hat{\Omega}^p. \label{lt}
\end{eqnarray}
Note that the fluid domain is determined by the displacement of the membrane $\hat{\boldsymbol \eta}$, while the porous medium domain is determined by its displacement $\hat{\boldsymbol U}$, where $\hat{\boldsymbol U}$ is the displacement of the porous medium evaluated at the reference configuration..  
However, because of condition~\eqref{CBJS}, we can define a homeomorphism over $\Omega^f(t)\cup\Omega^p(t)$ by setting mappings $\mathcal{A}_t$ and $\mathcal{S}_t$ equal on $\Gamma(t)$. 
For the structure, we adopt the material mapping
\begin{equation}
\mathcal{S}_t(\hat{\boldsymbol x}) = \hat{\boldsymbol x} + \hat{\boldsymbol U} (\hat{\boldsymbol x}, t), \quad \forall \hat{\boldsymbol x} \in \hat{\Omega}^p.
\end{equation}
Since the mapping $\mathcal{A}_t$ is arbitrary, with the only requirement that it matches $\mathcal{S}_t$ on $\Gamma(t)$, we can define $\mathcal{A}_t$ as
\begin{equation}
\mathcal{A}_t(\hat{\boldsymbol x}) = \hat{\boldsymbol x} + \textrm{Ext}(\hat{\boldsymbol \eta} (\hat{x}, t) )= \hat{\boldsymbol x} + \textrm{Ext}(\hat{\boldsymbol U} (\hat{\boldsymbol x}, t)|_{\hat{\Gamma}}), \quad \forall \hat{\boldsymbol x} \in \hat{\Omega}^f.
\end{equation}

\subsection{The coupling conditions}
In order to prescribe the coupling conditions on the physical fluid-structure
interface $\Gamma(t)$, denote by  $\boldsymbol \eta := \hat{\boldsymbol \eta} \circ (\mathcal{A}_t^{-1}|_{\Gamma(t)})$, where $A_t$ is defined in~\eqref{At}.
While the lumen and the poroelastic medium contain fluid, we assume that the elastic membrane does not contain fluid, but allows the flow 
through it in the normal direction. This is a reasonable assumption because the elastic membrane represents tunica intima. It has been shown by experimental studies that the normal
transport in tunica intima is significantly greater than tangential transport~\cite{lee2008permeability}.
Denote by $\boldsymbol n$ the outward normal to the fluid domain and by $\boldsymbol \tau$ the tangential unit vector.
Thus, the fluid, elastic membrane and poroelastic structure are coupled via the following boundary conditions:
\begin{itemize}
 \item Mass conservation: since the thin lamina allows the flow through it, the continuity of normal flux is
\begin{equation}
  \boldsymbol{v} \cdot \boldsymbol n  =  \bigg(\alpha \frac{D \boldsymbol U}{D t} -\kappa \nabla p_p \bigg)\cdot \boldsymbol n  \quad \textrm{on} \; \Gamma(t).\label{CNF}
\end{equation} 
\item Since we do not allow filtration in the tangential direction, we prescribe no-slip boundary conditions between the fluid in the lumen and the elastic membrane, and between the elastic membrane and poroelastic medium: 
\begin{equation}
\boldsymbol v \cdot \boldsymbol \tau = \frac{\partial \boldsymbol \eta}{\partial t} \cdot \boldsymbol \tau, \quad 
\boldsymbol \eta  = \boldsymbol U \quad \textrm{on} \; \Gamma(t). \label{CBJS}
\end{equation}
\item Balance of normal components of the stress in the fluid phase:
\begin{equation}
 \boldsymbol n \cdot \boldsymbol \sigma^f \boldsymbol n = -p_p  \quad \textrm{on} \; \Gamma(t). \label{CBNSF}
\end{equation}
\item The conservation of momentum describes balance of contact forces. Precisely, it says that the sum of contact forces at the fluid-porous medium interface is equal to zero:
\begin{eqnarray}
\alpha \boldsymbol n \cdot \boldsymbol{\sigma}^f \boldsymbol{ n} - \boldsymbol n \cdot \boldsymbol {\sigma}^p \boldsymbol n + J^{-1}\boldsymbol f \cdot \boldsymbol n= 0 \quad \textrm{on} \; \Gamma(t),   \label{CBFN} \\
\boldsymbol \tau \cdot \boldsymbol{\sigma}^f \boldsymbol{ n} - \boldsymbol \tau \cdot \boldsymbol {\sigma}^p \boldsymbol n + J^{-1} \boldsymbol f \cdot \boldsymbol \tau= 0 \quad \textrm{on} \; \Gamma(t),   \label{CBFT}
\end{eqnarray}
where $\boldsymbol f := \hat{\boldsymbol f} \circ (\mathcal{A}_t^{-1}|_{\Gamma(t)})$, and $J$ denotes the Jacobian of the transformation from $\Gamma(t)$ to $\hat{\Gamma}$ given by
\begin{equation}\label{Jacobian}
J = \sqrt{\bigg(1+ \frac{\partial \eta_x}{\partial x} \bigg)^2+\bigg(\frac{\partial \eta_y}{\partial x}\bigg)^2}.
\end{equation}
\end{itemize}

\subsection{Weak formulation of the monolithic problem}
For a domain $\Omega$, we denote by $|| \cdot ||_{H^k(\Omega)}$ the norm in the Sobolev space $H^k(\Omega)$. The norm in $L^2(\Omega)$ is denoted by $||\cdot||_{L^2(\Omega)}$, and the $L^2(\Omega)-$ inner product by $(\cdot,\cdot)_{\Omega}$.
To find a weak form of the Navier-Stokes equation, introduce the following test function spaces: 
\begin{eqnarray}
V^f(t) &=& \{\boldsymbol {\varphi}: \Omega^f(t) \rightarrow \mathbb{R}^2| \; \boldsymbol \varphi = \hat{\boldsymbol \varphi} \circ (\mathcal{A}_t)^{-1}, \hat{\boldsymbol \varphi} \in (H^1(\hat{\Omega}^f))^2, \quad \varphi_y=0  \; \textrm{on} \; \Gamma_0^f\}, \label{Vf(t)} \\
Q^f(t) &=& \{\psi: \Omega^f(t) \rightarrow \mathbb{R}| \; \psi = \hat{\psi} \circ (\mathcal{A}_t)^{-1}, \hat{\psi} \in L^2(\hat{\Omega}^f) \}, \label{Q(t)} 
\end{eqnarray}
for all $t \in [0,T)$.
The variational formulation of the Navier-Stokes equations now reads: given $t \in (0, T)$ find $(\boldsymbol v, p_f) \in  V^f(t) \times Q^f(t)$ such that for all $(\boldsymbol \varphi^f, \psi^f) \in V^f(t) \times Q^f(t)$
\begin{equation*}
\rho_f \int_{\Omega^f(t)} \frac{\partial \boldsymbol v}{\partial t} \cdot \boldsymbol \varphi^f d\boldsymbol x+ \rho_f \int_{\Omega^f(t)} (\boldsymbol v \cdot \nabla) \boldsymbol v \cdot \boldsymbol \varphi^f d \boldsymbol x  +2 \mu_f \int_{\Omega^f(t)} \boldsymbol D(\boldsymbol v) : \boldsymbol D(\boldsymbol \varphi^f) d \boldsymbol x 
\end{equation*}
\begin{equation*}
- \int_{\Omega^f(t)} p_f \nabla \cdot \boldsymbol \varphi^f d \boldsymbol x + \int_{\Omega^f(t)} \psi^f \nabla \cdot \boldsymbol v d \boldsymbol x = \int_{\Gamma(t)} \boldsymbol \sigma^f \boldsymbol n \cdot \boldsymbol \varphi^f ds
\end{equation*}
\begin{equation}
 + \int_{\Omega^f(t)} \boldsymbol g \cdot \boldsymbol \varphi^f d \boldsymbol x+ \int_{\Gamma_{in}} p_{in}(t) \varphi^f_x dy. \label{NSWeak}
\end{equation}

In order to write the weak form of the linearly elastic Koiter membrane, let $\hat{V}^m = (H_0^1(0,L))^2$.
Then the weak formulation reads as follows: given $t \in (0,T)$ find $\hat{\boldsymbol \eta} \in \hat{V}^m$ such that for all $\hat{\boldsymbol \zeta} \in \hat{V}^m$
\begin{equation*}
\rho_{m} r_m \int_0^L \frac{\partial^2 \hat\eta_x}{\partial t^2} \hat\zeta_x d\hat{x} + \rho_{m} r_m \int_0^L \frac{\partial^2 \hat\eta_y}{\partial t^2} \hat\zeta_y d\hat{x} -C_2 \int_0^L \frac{\partial \hat\eta_y}{\partial \hat{x}} \hat\zeta_x d\hat{x} +C_1 \int_0^L \frac{\partial \hat\eta_x}{\partial \hat{x}} \frac{\partial \hat\zeta_x}{\partial \hat{x}} d\hat{x}
\end{equation*}
\begin{equation}
+C_0 \int_0^L \hat\eta_y \hat\zeta_y d\hat{x}  + C_2 \int_0^L  \frac{\partial \hat\eta_x}{\partial \hat{x}} \hat\zeta_y d\hat{x}  = \int_0^L \hat{\boldsymbol f} \cdot \hat{\boldsymbol \zeta}  d\hat{x} . \label{KoiterWeak}
\end{equation}

Finally, let us introduce
\begin{eqnarray*}
V^p(t) &=& \{\boldsymbol {\varphi}: \Omega^p(t) \rightarrow \mathbb{R}^2| \; \boldsymbol \varphi = \hat{\boldsymbol \varphi} \circ (\mathcal{S}_t)^{-1}, \hat{\boldsymbol \varphi} \in (H^1(\hat{\Omega}^p))^2,  
\ \boldsymbol \varphi =0 \; \textrm{on} \; \Gamma^p_{in}\cup \Gamma^p_{out}, \varphi_x=0 \;\textrm{on} \;\Gamma^p_{ext}(t)\}, \label{Vp}\\
Q^p(t) &=& \{\psi: \Omega^p(t) \rightarrow \mathbb{R}| \; \psi = \hat{\psi} \circ (\mathcal{S}_t)^{-1}, \hat{\psi} \in H^1(\hat{\Omega}^p),  \psi|_{\partial \Omega^p(t) \backslash \Gamma(t)} = 0 \}. \label{Qp}
\end{eqnarray*}
Now the weak form of the Biot system reads as follows:  given $t \in (0, T)$ find $(\boldsymbol U, p_p) \in  V^p(t) \times Q^p(t)$ such that for all $(\boldsymbol \varphi^p, \psi^p) \in V^p(t)\times Q^p(t)$
\begin{equation*}
\rho_{p} \int_{\Omega^p(t)} \frac{D^2 \boldsymbol U}{D t^2} \boldsymbol \varphi^p d \boldsymbol x + \int_{\Omega^p(t)} \boldsymbol \sigma^E : \nabla \boldsymbol \varphi^p d \boldsymbol x - \alpha \int_{\Omega^p(t)} p_p \nabla \cdot \boldsymbol \varphi^p d \boldsymbol x +  \int_{\Omega^p(t)} s_0 \frac{D p_p}{D t} \psi^p d \boldsymbol x
\end{equation*}
\begin{equation*}
+\alpha  \int_{\Omega^p(t)} \nabla \cdot \frac{D \boldsymbol U}{D t} \psi^p d \boldsymbol x + \int_{\Omega^p(t)} \kappa \nabla p_p \cdot \nabla \psi^p d \boldsymbol =  -\int_{\Gamma(t)} \boldsymbol \sigma^p \boldsymbol n \cdot \boldsymbol \varphi^p d x - \int_{\Gamma(t)} \kappa \nabla p_p \cdot \boldsymbol n \psi^p d x 
\end{equation*}
\begin{equation}
-\int_{\Gamma^p_{ext}} p_e \boldsymbol \varphi^p_y dx+  \int_{\Omega^p(t)} \boldsymbol h \cdot \boldsymbol \varphi^p d \boldsymbol x+  \int_{\Omega^p(t)} s \boldsymbol \varphi^s d \boldsymbol x. \label{BiotWeak}
\end{equation}

To write a weak formulation of the coupled Navier-Stokes/Koiter/Biot system, define a space of admissible solutions
\begin{eqnarray}
W(t) &=&  \{ (\boldsymbol {\varphi}^f, \hat{\boldsymbol \zeta}, \boldsymbol \varphi^p) \in V^f(t) \times \hat{V}^m \times V^p(t) | \ \boldsymbol \zeta = \boldsymbol \varphi^p|_{\Gamma(t)}, \boldsymbol \varphi^f|_{\Gamma(t)} \cdot \boldsymbol \tau = \boldsymbol \zeta \cdot \boldsymbol \tau\}, 
\end{eqnarray}
where $\boldsymbol \zeta := \hat{\boldsymbol \zeta} \circ (\mathcal{A}_t^{-1}|_{\Gamma(t)})$, and add together equations~\eqref{NSWeak},~\eqref{KoiterWeak} and~\eqref{BiotWeak}:
\begin{equation*}
\rho_f \int_{\Omega^f(t)} \frac{\partial \boldsymbol v}{\partial t} \cdot \boldsymbol \varphi^f d\boldsymbol x+ \rho_f \int_{\Omega^f(t)} (\boldsymbol v \cdot \nabla) \boldsymbol v \cdot \boldsymbol \varphi^f d \boldsymbol x  +2 \mu_f \int_{\Omega^f(t)} \boldsymbol D(\boldsymbol v) : \boldsymbol D(\boldsymbol \varphi^f) d \boldsymbol x 
\end{equation*}
\begin{equation*}
- \int_{\Omega^f(t)} p_f \nabla \cdot \boldsymbol \varphi^f d \boldsymbol x + \int_{\Omega^f(t)} \psi^f \nabla \cdot \boldsymbol v d \boldsymbol x +\rho_{m} r_m \int_0^L \frac{\partial^2 \hat\eta_x}{\partial t^2} \hat\zeta_x dx + \rho_{m} r_m \int_0^L \frac{\partial^2 \hat\eta_y}{\partial t^2} \hat\zeta_y dx
\end{equation*}
\begin{equation*}
 -C_2 \int_0^L \frac{\partial \hat\eta_y}{\partial \hat{x}} \hat\zeta_x dx +C_1 \int_0^L \frac{\partial \hat\eta_x}{\partial \hat{x}} \frac{\partial \hat\zeta_x}{\partial \hat{x}} dx+C_0 \int_0^L \hat\eta_y \hat\zeta_y dx + C_2 \int_0^L  \frac{\partial \hat\eta_x}{\partial \hat{x}} \hat\zeta_y dx
\end{equation*}
\begin{equation*}
+\rho_{p} \int_{\Omega^p(t)} \frac{D^2 \boldsymbol U}{D t^2} \boldsymbol \varphi^p d \boldsymbol x + \int_{\Omega^p(t)}  \boldsymbol \sigma^E : \nabla \boldsymbol \varphi^p d \boldsymbol x - \alpha \int_{\Omega^p(t)} p_p \nabla \cdot \boldsymbol \varphi^p d \boldsymbol x  +  \int_{\Omega^p(t)} s_0 \frac{D p_p}{D t} \psi^p d \boldsymbol x
\end{equation*}
\begin{equation*}
+\alpha  \int_{\Omega^p(t)} \nabla \cdot \frac{D \boldsymbol U}{D t} \psi^p d \boldsymbol x  + \int_{\Omega^p(t)} \kappa \nabla p_p \cdot \nabla \psi^p d \boldsymbol = \int_{\Gamma(t)} \boldsymbol \sigma^f \boldsymbol n \cdot \boldsymbol \varphi^f d s  -\int_{\Gamma(t)} \boldsymbol \sigma^p \boldsymbol n \cdot \boldsymbol \varphi^p d x 
\end{equation*}
\begin{equation*}
- \int_{\Gamma(t)} \kappa \nabla p_p \cdot \boldsymbol n \psi^p d x  +  \int_0^L \hat{\boldsymbol f} \cdot \hat{\boldsymbol \zeta}  dx+ \int_{\Omega^f(t)} \boldsymbol g \cdot \boldsymbol \varphi^f d \boldsymbol x+ \int_{\Gamma_{in}} p_{in}(t) \varphi^f_x dy 
\end{equation*}
\begin{equation}
-\int_{\Gamma^p_{ext}} p_e \boldsymbol \varphi^p_y dx
  +\int_{\Omega^p(t)} \boldsymbol h \cdot \boldsymbol \varphi^p d \boldsymbol x+ \int_{\Omega^p(t)} s \psi^p d \boldsymbol x. 
\end{equation}
Denote by $I_{\Gamma(t)}$ the interface integral
\begin{equation*}
I_{\Gamma(t)} =  \int_{\Gamma(t)} (\boldsymbol \sigma^f \boldsymbol n \cdot \boldsymbol \varphi^f - \boldsymbol \sigma^p \boldsymbol n \cdot \boldsymbol \varphi^p- \kappa \nabla p_p \cdot \boldsymbol n \psi^p  + J^{-1} \boldsymbol f \cdot \boldsymbol \zeta ) dx.
\end{equation*}

Decomposing the stress terms and thin shell forcing term into their normal and tangential components and employing conditions~\eqref{CNF} and~\eqref{CBNSF} we have
\begin{equation*}
I_{\Gamma(t)} = \int_{\Gamma(t)} \bigg( -p_p  \boldsymbol  \varphi^f \cdot \boldsymbol n - (\boldsymbol n \cdot  \boldsymbol \sigma^p \boldsymbol n) (\boldsymbol  \varphi^p \cdot \boldsymbol n)  + J^{-1}(\boldsymbol f \cdot \boldsymbol n) (\boldsymbol \zeta \cdot \boldsymbol n) +\boldsymbol v \cdot \boldsymbol n \psi^p -\alpha \frac{D \boldsymbol U}{D t} \cdot \boldsymbol n\psi^p
\end{equation*}
\begin{equation*}
+ (\boldsymbol \tau \cdot \boldsymbol \sigma^f \boldsymbol n ) (\boldsymbol \varphi^f \cdot \boldsymbol \tau) - (\boldsymbol \tau \cdot \boldsymbol \sigma^p \boldsymbol n) (\boldsymbol \varphi^p \cdot \boldsymbol \tau)+ J^{-1}(\boldsymbol f \cdot \boldsymbol \tau) (\boldsymbol \zeta \cdot \boldsymbol \tau) \bigg) dx.
\end{equation*}
For each triple of test functions $(\boldsymbol \varphi^f , \hat{\boldsymbol \zeta}, \boldsymbol \varphi^p) \in W(t)$, and due to the condition~\eqref{CBFT} we have
\begin{equation*}
I_{\Gamma(t)} = \int_{\Gamma(t)} \bigg( -p_p  \boldsymbol  \varphi^f \cdot \boldsymbol n - (\boldsymbol n \cdot  \boldsymbol \sigma^p \boldsymbol n) (\boldsymbol  \varphi^p \cdot \boldsymbol n)  + J^{-1}(\boldsymbol f \cdot \boldsymbol n) (\boldsymbol \varphi^p \cdot \boldsymbol n) +\boldsymbol v \cdot \boldsymbol n \psi^p -\alpha \frac{D \boldsymbol U}{D t} \cdot \boldsymbol n\psi^p \bigg) dx.
\end{equation*}
Finally, decomposing $\boldsymbol \sigma^p \boldsymbol n$ into $\boldsymbol \sigma^E \boldsymbol n - \alpha p_p \boldsymbol n$, and employing conditions~\eqref{CBNSF} and~\eqref{CBFN} we have
\begin{equation*}
I_{\Gamma(t)} = \int_{\Gamma(t)} \bigg( -p_p  \boldsymbol  \varphi^f \cdot \boldsymbol n +\alpha p_p  \boldsymbol  \varphi^p \cdot \boldsymbol n +\boldsymbol v \cdot \boldsymbol n \psi^p -\alpha \frac{D \boldsymbol U}{D t} \cdot \boldsymbol n\psi^p \bigg) dx.
\end{equation*}

Thus, the weak formulation of the coupled Navier-Stokes/Koiter/Biot system reads as follows:  given $t \in (0, T)$ find $(\boldsymbol v, \hat{\boldsymbol \eta}, \boldsymbol U, p_f, p_p) \in V^f(t) \times \hat{V}^m\times V^p(t) \times Q^f(t)\times Q^p(t)$, with $(\boldsymbol v, \frac{\partial \hat{\boldsymbol \eta}}{\partial t}, \frac{ D\boldsymbol U}{Dt}) \in W(t),$ such that for all $(\boldsymbol \varphi^f,  \hat{\boldsymbol \zeta}, \boldsymbol \varphi^p, \psi^f, \psi^p) \in  W(t) \times Q^f(t) \times Q^p(t)$
\begin{equation*}
\rho_f \int_{\Omega^f(t)} \frac{\partial \boldsymbol v}{\partial t} \cdot \boldsymbol \varphi^f d\boldsymbol x+ \int_{\Omega^f(t)} (\boldsymbol v \cdot \nabla) \boldsymbol v \cdot \boldsymbol \varphi^f d \boldsymbol x  +2 \mu_f \int_{\Omega^f(t)} \boldsymbol D(\boldsymbol v) : \boldsymbol D(\boldsymbol \varphi^f) d \boldsymbol x 
\end{equation*}
\begin{equation*}
- \int_{\Omega^f(t)} p_f \nabla \cdot \boldsymbol \varphi^f d \boldsymbol x + \int_{\Omega^f(t)} \psi^f \nabla \cdot \boldsymbol v d \boldsymbol x +\rho_{m} r_m \int_0^L \frac{\partial^2 \hat\eta_x}{\partial t^2} \hat\zeta_x dx + \rho_{m} r_m \int_0^L \frac{\partial^2 \hat\eta_y}{\partial t^2} \hat\zeta_y dx
\end{equation*}
\begin{equation*}
 -C_2 \int_0^L \frac{\partial \hat\eta_y}{\partial \hat{x}} \hat\zeta_x dx +C_1 \int_0^L \frac{\partial \hat\eta_x}{\partial \hat{x}} \frac{\partial \hat\zeta_x}{\partial \hat{x}} dx+C_0 \int_0^L \hat\eta_y \hat\zeta_y dx  + C_2 \int_0^L  \frac{\partial \hat\eta_x}{\partial \hat{x}} \hat\zeta_y dx
\end{equation*}
\begin{equation*}
+\rho_{p} \int_{\Omega^p(t)} \frac{D^2 \boldsymbol U}{D t^2}\boldsymbol \varphi^p d \boldsymbol x + \int_{\Omega^p(t)} \boldsymbol \sigma^E : \nabla \boldsymbol \varphi^p d \boldsymbol x  - \alpha \int_{\Omega^p(t)} p_p \nabla \cdot \boldsymbol \varphi^p d \boldsymbol x+\int_{\Omega^p(t)} s_0 \frac{D p_p}{D t} \psi^p d \boldsymbol x
\end{equation*}
\begin{equation*}
+\alpha  \int_{\Omega^p(t)} \nabla \cdot \frac{D \boldsymbol U}{D t} \psi^p d \boldsymbol x +  \int_{\Omega^p(t)} \kappa \nabla p_p \cdot \nabla \psi^p d \boldsymbol x +\int_{\Gamma(t)} p_p  \boldsymbol  \varphi^f \cdot \boldsymbol n dx -\alpha \int_{\Gamma(t)} p_p  \boldsymbol  \varphi^p\cdot \boldsymbol n dx
\end{equation*}
\begin{equation*}
 -\int_{\Gamma(t)}\boldsymbol v \cdot \boldsymbol n \psi^p dx +\alpha \int_{\Gamma(t)}\frac{D \boldsymbol U}{D t} \cdot \boldsymbol n\psi^p dx =  \int_{\Omega^f(t)} \boldsymbol g \cdot \boldsymbol \varphi^f d \boldsymbol x+ \int_{\Gamma_{in}} p_{in}(t) \varphi^f_x dy
 -\int_{\Gamma^p_{ext}} p_e \boldsymbol \varphi^p_y dx
 \end{equation*}
\begin{equation}
  +\int_{\Omega^p(t)} \boldsymbol h \cdot \boldsymbol \varphi^p d \boldsymbol x 
 +  \int_{\Omega^p(t)} s \psi^p d \boldsymbol x.  \label{NSMBweak}
\end{equation}

\subsection{Energy equality}
In this section we will use an equivalent variational formulation to~\eqref{KoiterWeak} (see~\cite{canic2006modeling} for details), given by:
\begin{equation*}
 \rho_{m} r_m \int_0^L \frac{\partial^2 \hat\eta_x}{\partial t^2} \hat\zeta_x d\hat{x} + \rho_{m} r_m \int_0^L \frac{\partial^2 \hat\eta_y}{\partial t^2} \hat\zeta_y d\hat{x} + \frac{r_m}{2} \int_0^L  4 \mu_m (\frac{\partial \hat\eta_x}{\partial \hat{x}}\frac{\partial \hat\zeta_x}{\partial \hat{x}} +\frac{1}{R^2} \hat\eta_y \hat\zeta_y) d\hat{x}
\end{equation*}
\begin{equation}
 +\frac{r_m}{2} \int_0^L \frac{4 \mu_m \lambda_m}{\lambda_m+2\mu_m} (\frac{\partial \hat\eta_x}{\partial \hat{x}}+\frac{1}{R} \hat\eta_y)(\frac{\partial \hat\zeta_x}{\partial \hat{x}}+\frac{1}{R} \hat\zeta_y) d\hat{x} = \int_0^L \hat{\boldsymbol f} \cdot \hat{\boldsymbol \zeta} d\hat{x}. \label{KoiterWeakEnergy}
\end{equation}
To formally derive the energy of the coupled problem, we add together variational formulations for Navier-Stokes equations~\eqref{NSWeak}, Koiter membrane model~\eqref{KoiterWeakEnergy}, and Biot system~\eqref{BiotWeak}. 
The coupling conditions are then used to couple the fluid, thin structure, and porous medium sub-problems, using manipulations as in the previous section, resulting in an equation similar to~\eqref{NSMBweak}.
Let $$(\boldsymbol \varphi^f,  \hat{\boldsymbol \zeta},  \boldsymbol \varphi^p, \psi^f, \psi^p)=(\boldsymbol v,  \frac{\partial \hat{\boldsymbol \eta}}{\partial t}, \frac{D \boldsymbol U}{D t}, p_f, p_p).$$
Then, the energy equality for coupled system is given as follows:
\begin{equation*}
\frac{1}{2}\frac{d}{dt} \bigg\{\rho_f ||\boldsymbol v||^2_{L^2(\Omega^f(t))} + \rho_{m} r_m \bigg|\bigg|\frac{\partial \hat\eta_x}{\partial t} \bigg|\bigg|^2_{L^2(0,L)} +\rho_{m} r_m \bigg|\bigg|\frac{\partial \hat\eta_y}{\partial t} \bigg|\bigg|^2_{L^2(0,L)}\bigg\} 
\end{equation*}
 \begin{equation*}
 +\displaystyle r_m \left[
  4 \mu_m \bigg|\bigg|\frac{\hat\eta_y}{R}\bigg|\bigg|^2_{L^2(0,L)} +
  4 \mu_m \bigg|\bigg|\frac{\partial \hat\eta_x}{\partial \hat{x}} \bigg|\bigg|^2_{L^2(0,L)}
  +\frac{4 \mu_m \lambda_m}{\lambda_m+2\mu_m} \bigg|\bigg|\frac{\partial \hat\eta_x}{\partial \hat{x}}+ \frac{\hat\eta_y}{R} \bigg|\bigg|^2_{L^2(0,L)}
 \right]
\end{equation*}
\begin{equation*}
+\rho_{p}\bigg|\bigg|\frac{D \boldsymbol U}{D t} \bigg| \bigg|^2_{L^2(\Omega^p(t))}   +2\mu_p ||\boldsymbol D(\boldsymbol U)||^2_{L^2(\Omega^p(t))} +  \lambda_p ||\nabla \cdot \boldsymbol U||^2_{L^2(\Omega^p(t))}   +s_0||p_p||^2_{L^2(\Omega^p(t))} \bigg\} 
\end{equation*}
\begin{equation*}
 + 2 \mu_f ||\boldsymbol D(\boldsymbol v)||^2_{L^2(\Omega(t))}  
 \nonumber\\+   ||\kappa \nabla p_p||^2_{L^2(\Omega^p(t))}
=   \int_{\Omega^f(t)} \boldsymbol g \cdot \boldsymbol v d \boldsymbol x
+\int_{\Gamma_{in}} p_{in}(t) v_x dy 
\end{equation*}
\begin{equation*}
   -\int_{\Gamma^p_{ext}} p_e \frac{D U_y}{D t} dx +\int_{\Omega^p(t)} s p_p d \boldsymbol x  +\int_{\Omega^p(t)} \boldsymbol h \cdot \frac{D \boldsymbol U}{D t} d \boldsymbol x.
\end{equation*}

\section{A loosely-coupled scheme based on the operator-splitting approach}

To solve the fluid-multilayer structure interaction problem described in Section~1 numerically,
we propose a loosely coupled scheme based on a time-splitting approach known as the Lie splitting~\cite{glowinski2003finite}. Details of the Lie splitting are described below.

\subsection{The Lie scheme}\label{sec:Lie}
Let $A$ be an operator from a Hilbert space $H$ into itself, and suppose $\phi_0 \in$ $H$.
Consider the following initial value problem:
\begin{equation}\label{LieProblem}
\frac{\partial \phi}{\partial t} + A(\phi) = 0, \quad \textrm{in} \ (0,T), \quad {\rm where}\quad  A = \sum\limits_{i=1}^P A_i,
\quad
\phi(0) = \phi_0.
\end{equation}
The Lie scheme consists of splitting the full problem into $P$ sub-problems, each defined by the operator $A_i, i = 1,...,P$.
The original problem is discretized in time with the time step $\Delta t>0$, so that $t^n=n\Delta t$. The Lie splitting scheme consist of solving a series of problems $ \frac{\partial \phi_i}{\partial t} + A_i(\phi_i) = 0$, for $i = 1,...,P$, each defined
over the entire time interval $(t^n, t^{n+1})$, but with the initial data for the $i^{th}$ problem given by the solution of the $(i-1)^{st}$ problem at $t^{n+1}$.
More precisely, set $\phi^0=\phi_0.$ Then, for $n \geq 0$ compute $\phi^{n+1}$ by solving
\begin{equation}
\frac{\partial \phi_i}{\partial t} + A_i(\phi_i) = 0 \quad \textrm{in} \; (t^n, t^{n+1}),
\quad
\phi_i(t^n) = \phi^{n+(i-1)/P}, 
\end{equation}
and then set $\phi^{n+i/P} = \phi_i(t^{n+1}),$ for $i=1, \dots. P.$
This method is first-order accurate in time.
To increase the accuracy in time to second-order,
a symmetrization of the scheme can be performed.

\subsection{The first-order system in the ALE framework}
To apply the Lie operator splitting scheme, we have to rewrite our system in a first-order form. 
Therefore, we express the second-order time derivatives of both elastic and poroelastic structure displacements in terms of the 
first-order derivative of structure velocities. Furthermore, we consider the Navier-Stokes equation in the ALE formulation.
To write the Navier-Stokes equations in the ALE form, we notice that for a function $f=f(\boldsymbol{x},t)$ defined on $\Omega^f(t) \times (0,T)$ 
the corresponding function $\hat{f} := f \circ \mathcal{A}_t$ defined on $\hat{\Omega} \times (0,T)$
is given by
$$\hat{f}(\hat{\boldsymbol{x}},t) = f(\mathcal{A}_t(\hat{\boldsymbol{x}}),t).$$
Differentiation with respect to time, after using the chain rule, gives
\begin{equation}
 \frac{\partial f}{\partial t}\bigg|_{\hat{\boldsymbol x}} =  \frac{\partial {f}}{\partial t}+\boldsymbol{w} \cdot \nabla f,
 \quad
 \boldsymbol{w}(\boldsymbol{x},t) = \frac{\partial \mathcal{A}_t(\hat{\boldsymbol x})}{\partial t}.\bigg. \label{w}
\end{equation}
where $\boldsymbol{w}$ denotes the domain velocity.
We will apply this rule to write the time-derivative of the velocity in Navier-Stokes equations on the reference domain. Note that we do not have to apply the same rule
to the time-derivatives in the Biot system and in Koiter membrane equations since the material time-derivative is suitable for the time discretization, due to $\frac{D q}{Dt} = \frac{\partial q}{\partial t} \big|_{\hat x}$, 
and the membrane equations are given on the reference configuration.
With these assumptions, our problem now reads:
Given $t \in (0, T)$, find $\boldsymbol {v}= (v_x,v_y)$, $p_f, \boldsymbol{\hat{\eta}}=(\hat{\eta}_x,\hat{\eta}_y), \boldsymbol {\hat{\xi}}=(\hat{\xi}_x,\hat{\xi}_y), \boldsymbol {U}=(U_x,U_y), \boldsymbol{V}=(V_x,V_y)$ and $p_p$,
with $\boldsymbol \eta (\boldsymbol {x},t) = \hat{\boldsymbol \eta}(\mathcal{A}_t^{-1}(\boldsymbol x),t),$ for $\boldsymbol x \in \Gamma(t)$,
 such that
 \begin{subequations}
\begin{align}
& \rho_f \bigg( \displaystyle\frac{\partial \boldsymbol {v}}{\partial t}\bigg|_{\hat{\boldsymbol x}}+ (\boldsymbol {{v}}-\boldsymbol {w}) \cdot \nabla \boldsymbol {v} \bigg)  = \nabla \cdot \boldsymbol\sigma^f +\boldsymbol g & \textrm{in}\; \Omega^f(t) \times (0,T), \label{NSfo} \\
& \nabla \cdot \boldsymbol {v} = 0  &\textrm{in}\; \Omega^f(t) \times (0,T), \\
& \rho_{m} r_m \displaystyle\frac{\partial \hat{\boldsymbol \xi}}{\partial t}+\hat{\mathcal{L}} \hat{\boldsymbol \eta}= \hat {\boldsymbol f}^k   & \textrm{on} \; \hat{\Gamma} \times (0,T), \label{Koiterfo}  \\
& \rho_{m} r_m \bigg(\hat{\boldsymbol \xi}-\displaystyle\frac{\partial \hat{\boldsymbol \eta}}{\partial t} \bigg)=0 & \textrm{on} \; \hat{\Gamma}\times (0,T), \label{etafo} \\
& \rho_{p} \displaystyle\frac{D \boldsymbol V}{D t} = \nabla \cdot \boldsymbol \sigma^p + \boldsymbol h & \textrm{in} \; \Omega^p(t) \times (0,T), \label{structurefo} \\
& s_0 \displaystyle\frac{D}{D t} p_p + \alpha \nabla \cdot \boldsymbol V - \nabla \cdot (\kappa \nabla p_p )= s & \textrm{in} \; \Omega^p(t) \times (0,T), \label{pressurefo}\\
& \rho_{p} \bigg (\boldsymbol V-\displaystyle\frac{D \boldsymbol U}{D t} \bigg)= 0 & \textrm{in} \; \Omega^p(t) \times (0,T), \label{vfo}
\end{align} \label{sys1}
\end{subequations}
with the kinematic coupling conditions on $\Gamma(t)$:
\begin{equation}
\boldsymbol \xi \cdot \boldsymbol \tau = \boldsymbol {v} \cdot \boldsymbol \tau,    \quad \boldsymbol \eta = \boldsymbol {U},   \label{kinematic}
\end{equation}
dynamic coupling conditions on $\Gamma(t)$:
\begin{eqnarray}
\boldsymbol \tau \cdot \boldsymbol{\sigma}^f \boldsymbol{ n} - \boldsymbol \tau \cdot \boldsymbol {\sigma}^p \boldsymbol n + J^{-1} \boldsymbol \tau \cdot \boldsymbol f^k = 0, \label{dyn1}\\
\alpha \boldsymbol n \cdot \boldsymbol{\sigma}^f \boldsymbol{ n} - \boldsymbol n \cdot \boldsymbol {\sigma}^p \boldsymbol n + J^{-1}\boldsymbol n \cdot \boldsymbol f^k= 0, \label{dyn2}\\
 \boldsymbol n \cdot \boldsymbol \sigma^f \boldsymbol n = -p_p, \label{dyn3}
\end{eqnarray}
and the continuity of normal flux on $\Gamma(t)$:
\begin{equation}
\boldsymbol{v} \cdot \boldsymbol n  =  \bigg(\alpha \frac{\partial \boldsymbol \eta}{\partial t} - \kappa \nabla p_p \bigg)\cdot \boldsymbol n, \label{sys2}
\end{equation}
with the boundary and initial conditions given in Section 1.
\if 1=0
The following boundary  conditions on $(\partial \Omega^f(t) \cup \partial \Omega^p(t)) \backslash \Gamma(t)$ are enforced:
\begin{subequations}
\begin{eqnarray}
& & \frac{\partial v_x}{\partial y}(x,0,t) =  v_y(x,0,t) = 0 \quad \textrm{on} \; (0,L)\times (0,T), \\
& & \boldsymbol {v}(0,R,t) = \boldsymbol {v}(L,R,t) = 0, \quad \hat{\boldsymbol\eta}|_{\hat{x}=0,L}=\frac{\partial \hat{\eta}_y}{\partial x}\bigg|_{\hat{x}=0,L}=0 , \\
& &  \boldsymbol U(0, y, t) = \boldsymbol U(L, y, t) = 0, \\
& & \boldsymbol {\sigma}^f \boldsymbol {n}|_{in}(0, y,t) = -p_{in}(t) \boldsymbol {n}|_{in} \; \; \textrm{on} \; \Gamma^f_{in} \times (0,T), \\
& & \boldsymbol {\sigma}^f \boldsymbol {n}|_{out}(L, y,t) = 0 \;\; \textrm{on} \;\; \Gamma^f_{out} \times (0,T), \\
& & \boldsymbol {n}_{ext} \cdot \boldsymbol \sigma^E \boldsymbol {n}_{ext} =  -p_e \; \; \textrm{on} \; \Gamma_{ext}\times (0,T), \\
& & \boldsymbol U_x(x,R+r_p,t) =  0 \; \;\textrm{on} \; (0,L)\times (0,T), \\
& & p_p = 0 \quad \textrm{on} \; \Gamma^p_{ext}(t) \cup \Gamma^p_{in} \cup \Gamma^p_{out}  \times (0,T).
\end{eqnarray}
\end{subequations}
At time $t = 0$ the following initial conditions are prescribed:
\begin{equation}
\boldsymbol {v}|_{t=0} = \boldsymbol {0}, \; \hat{\boldsymbol\eta}|_{t=0} = \boldsymbol {0}, \; \hat{\boldsymbol\xi}|_{t=0} = \boldsymbol {0},\; \boldsymbol U|_{t=0} = \boldsymbol {0},\; \boldsymbol V|_{t=0} = \boldsymbol {0}, \; p_p|_{t=0} = \boldsymbol {0}. \label{sys2}
\end{equation}
\fi

\vskip 0.1in
\noindent
{\bf Remark 1.}
Denote by $\mathcal{L}$ the inverse Piola transformation of $\hat{\mathcal{L}}$, namely
$\mathcal{L} = J^{-1}\hat{\mathcal{L}} F^{-T},$ where $F = \nabla_{\boldsymbol x} \mathcal{A}_t$.
Then, composing the Koiter shell equations~\eqref{KoiterCompact} with $\mathcal{A}_t^{-1}$, and employing the first condition in~\eqref{kinematic}, condition~\eqref{dyn3}, and relation
$$\frac{\partial}{\partial t} (\boldsymbol \xi \cdot \boldsymbol \tau) = \frac{\partial \boldsymbol \xi}{\partial t}\cdot \boldsymbol \tau+ \frac{\partial \boldsymbol \tau}{\partial t}\cdot \boldsymbol \xi,$$ 
we can write conditions~\eqref{dyn1} and~\eqref{dyn2} as follows:
\begin{subequations}
\begin{align}
&\rho_{m} r_m \frac{\partial (\boldsymbol v \cdot  \boldsymbol \tau)}{\partial t} + \boldsymbol \tau \cdot \mathcal{L} \boldsymbol \eta - \rho_{m} r_m  \frac{\partial \boldsymbol \tau}{\partial t}\cdot \boldsymbol \xi+ J \boldsymbol \tau \cdot \boldsymbol{\sigma}^f \boldsymbol{ n} - J \boldsymbol \tau \cdot \boldsymbol {\sigma}^p \boldsymbol n  = 0, \quad  &\textrm{on} \; \Gamma(t) \label{dyn_new1}\\
&\rho_{m} r_m \displaystyle\frac{\partial \boldsymbol \xi}{\partial t}\cdot \boldsymbol n  +\boldsymbol n \cdot \mathcal{L} \boldsymbol \eta- J \alpha p_p - J\boldsymbol n \cdot \boldsymbol {\sigma}^p \boldsymbol n  =0, \quad & \textrm{on} \; \Gamma(t). \label{dyn_new2}
\end{align} \label{dyn_new}
\end{subequations}
We will use conditions~\eqref{dyn1}-\eqref{dyn2} written the form~\eqref{dyn_new1}-\eqref{dyn_new2} when performing the operator splitting.

\subsection{Details of the loosely-coupled scheme}
In this section we will apply the Lie splitting scheme to problem~\eqref{sys1}, where the discretization in time will be done using the backward Euler scheme.
We will denote the discrete time derivatives by 
\begin{equation*}
 d_t \boldsymbol \varphi^{n+1} = \frac{\boldsymbol \varphi^{n+1}-\boldsymbol \varphi^{n}}{\Delta t}, \quad \textrm{and} \quad d_{tt} \boldsymbol \varphi^{n+1} = \frac{d_t \boldsymbol \varphi^{n+1}- d_t \boldsymbol \varphi^{n}}{\Delta t},
\end{equation*}
and the discrete time average by
\begin{equation*}
  \boldsymbol \varphi^{n+1/2} = \frac{\boldsymbol \varphi^{n+1}+\boldsymbol \varphi^{n}}{2},
\end{equation*}
where all quantities are evaluated on the reference domain.
In our case, using the notation from Section~\ref{sec:Lie}, $\phi$ that appears in equation~\eqref{LieProblem} is a vector $\phi=(\boldsymbol v, \boldsymbol v|_{\Gamma(t)} \cdot \boldsymbol \tau, \boldsymbol {\xi} \cdot \boldsymbol n, \boldsymbol \eta, \boldsymbol V, p_p, \boldsymbol U)^T$.
We will split the first-order system~\eqref{sys1}-\eqref{sys2} into two main sub-problems, separating the problem defined on the fluid domain $\Omega^f(t)$ from the problem defined on the poroelastic medium domain $\Omega^p(t)$. 
In that case, we will split the sum of all operators that appear in the system~\eqref{sys1} into two parts, as $A_1 + A_2$, where $A_i = (A_i^f, A_i^{k_1}, A_i^{k_2}, A_i^{\eta}, A_i^V, A_i^{p_p}, A_i^U)^T$, for $i=1,2$.
For each of the equations, this will be done in the following way: 
\begin{itemize}
 \item Equation~\eqref{NSfo} will be split so that $A_1^f = \rho_f (\boldsymbol {{v}}-\boldsymbol {w}) \cdot \nabla \boldsymbol {v}   - \nabla \cdot \boldsymbol\sigma^f$ and $A_2^f = 0$,
 \item Equation~\eqref{Koiterfo} will be used in form~\eqref{dyn_new}, where equation~\eqref{dyn_new1} will be split so 
 $$A_1^{k_1} = - \rho_{m} r_m  \frac{\partial \boldsymbol \tau}{\partial t}\cdot \boldsymbol \xi+ J \boldsymbol \tau \cdot \boldsymbol{\sigma}^f \boldsymbol{ n} \quad \textrm{and} \quad A_2^{k_1} = \boldsymbol \tau \cdot \mathcal{L} \boldsymbol \eta - J \boldsymbol \tau \cdot \boldsymbol {\sigma}^p \boldsymbol n, $$ and equation~\eqref{dyn_new2} will be split so 
 $$A_1^{k_2} = 0 \quad  \textrm{and} \quad A_2^{k_2} =\boldsymbol n \cdot \mathcal{L} \boldsymbol \eta- J \alpha p_p - J\boldsymbol n \cdot \boldsymbol {\sigma}^p \boldsymbol n.$$ 
 \item Equation~\eqref{etafo} will be split so that $A_1^{\eta} = 0$, and $A_1^{\eta} = \boldsymbol \xi,$
 \item Equation~\eqref{structurefo} will be split so that $A_1^{V} = 0$ and $A_2^{V} = -\nabla \cdot \boldsymbol \sigma^p,$
 \item Equation~\eqref{pressurefo} will be split so that $A_1^{p_p} = 0$ and $A_2^{p_p}=\alpha \nabla \cdot \boldsymbol V - \nabla \cdot (\kappa \nabla p_p ),$ and finally,
 \item Equation~\eqref{vfo} will be split so that $A_1^U = 0$ and $A_2^U = \boldsymbol V.$
 \end{itemize}

Using this approach, our system is decoupled into a fluid problem and the Biot problem. Furthermore, we not only split the coupled problem into two different domains, but we also treat different physical phenomena separately. 
Details of the loosely coupled scheme are given as follows.

\begin{itemize}
 \item \textbf{Step 1.}  Step 1 is a geometry problem which involves computation of a fluid domain and ALE velocity $\boldsymbol w:$
\begin{equation}
  \mathcal{A}_{t^n}(\hat{\boldsymbol x}) = \hat{\boldsymbol x} + \textrm{Ext} (\hat{\boldsymbol \eta}^n), \quad \Omega^f(t^{n}) = \mathcal{A}_{t^n}(\hat{\Omega}^f), \quad \boldsymbol w^n = d_{t} \boldsymbol x^n,
\end{equation}
where $\hat{\boldsymbol x} \in \hat{\Omega}, \boldsymbol x^{n} \in \Omega(t^n),$ and $\boldsymbol x^{n-1} \in \Omega(t^{n-1}).$
\item \textbf{Step 2.}  Step 2 involves solving the Navier-Stokes equations, and equation
 \begin{equation}
\rho_{m} r_m \frac{\partial (\boldsymbol v|_{\Gamma(t^n)} \cdot  \boldsymbol \tau)}{\partial t}  - \rho_{m} r_m  \frac{\partial \boldsymbol \tau}{\partial t}\cdot \boldsymbol \xi+ J \boldsymbol \tau \cdot \boldsymbol{\sigma}^f \boldsymbol{ n} = 0 \quad \textrm{on} \; \Gamma(t^n) \times (t^n, t^{n+1}). \label{step1cond}
 \end{equation}
 while time-derivatives of all the other functions are equal to zero. In particular, since the time derivative of the displacement of the elastic shell is equal to zero, the fluid domain is not changing in this step,
which implies
$$\frac{\partial \boldsymbol \tau}{\partial t} = 0 \quad \textrm{and} \; J=1 \quad \textrm{for} \; t \in (t^n, t^{n+1}).$$
Therefore, equation~\eqref{step1cond} can be seen as a Robin-type boundary condition for fluid velocity. Now, in the time-discrete framework, Step 2 reads as follows:
Find $\boldsymbol{v}^{n+1}$ and $p_f^{n+1},$ with $\boldsymbol v^n, p_f^n$ and $p_p^n$ obtained at the previous time step, such that
{\small{
\begin{subequations}
\begin{align}
& \rho_f d_t \boldsymbol{v}^{n+1} + (\boldsymbol{v}^{n+1}-\boldsymbol{w}^{n}) \cdot \nabla \boldsymbol{v}^{n+1}=\nabla \cdot \boldsymbol{\sigma}^f(\boldsymbol v^{n+1}, p_f^{n+1})+\boldsymbol g  & \textrm{in} \; \Omega^f(t^n), \\ 
& \nabla \cdot \boldsymbol{v}^{n+1}=0 & \textrm{in} \; \Omega^f(t^n), \\ 
& \boldsymbol \tau \cdot \boldsymbol{\sigma}^f(\boldsymbol v^{n+1}, p_f^{n+1}) \boldsymbol{ n} + \rho_{m} r_m d_t \boldsymbol v^{n+1} \cdot  \boldsymbol \tau = 0 & \textrm{on} \; \Gamma(t^n), \\
& \boldsymbol n \cdot \boldsymbol \sigma^f \boldsymbol {n}  = p^n_p  & \textrm{on} \; \Gamma(t^n). 
\end{align}\label{step1a}
\end{subequations}}}
with the following boundary conditions on $\Gamma^f_{\rm in}\cup\Gamma^f_{\rm out}\cup\Gamma^f_0$: 
\begin{equation*}
  \frac{\partial v^{n+1}_x}{\partial y} = \quad v^{n+1}_y = 0 \quad \textrm{on} \; \Gamma^f_0,
\quad 
  \boldsymbol{v}^{n+1}(0,R,t) = \boldsymbol{v}^{n+1}(L,R,t) = 0, 
\end{equation*}
\begin{equation*}
 \boldsymbol\sigma^f(\boldsymbol v^{n+1}, p_f^{n+1}) \boldsymbol{n} = -p_{in}(t)\boldsymbol{n}\  {\rm on}\ \Gamma^f_{\rm in}, \; \; \boldsymbol\sigma^f(\boldsymbol v^{n+1}, p_f^{n+1}) \boldsymbol {n} = 0   \ {\rm on}\ 
 \Gamma^f_{\rm out}.
\end{equation*}


\item \textbf{Step 3:}  Step 3 involves solving Biot problem together with equations
\begin{subequations}
\begin{align}
 & \rho_{m} r_m \frac{\partial (\boldsymbol v \cdot  \boldsymbol \tau)}{\partial t} + \boldsymbol \tau \cdot \mathcal{L} \boldsymbol \eta  - J \boldsymbol \tau \cdot \boldsymbol {\sigma}^p \boldsymbol n  = 0 & \textrm{on} \; \Gamma(t^n)\times(t^n, t^{n+1}), \label{dyn1step2} \\
& \rho_{m} r_m \displaystyle\frac{\partial \boldsymbol \xi}{\partial t}\cdot \boldsymbol n  +\boldsymbol n \cdot \mathcal{L} \boldsymbol \eta- J \alpha p_p - J\boldsymbol n \cdot \boldsymbol {\sigma}^p \boldsymbol n  =0 & \textrm{on} \; \Gamma(t^n)\times(t^n, t^{n+1}), \label{dyn2step2} 
\end{align}
\end{subequations}
and conditions
\begin{align}
& \alpha \frac{\partial \boldsymbol{\eta}}{\partial t} \cdot \boldsymbol n  =  \boldsymbol{v} \cdot \boldsymbol{n} + \kappa \nabla p_p  \cdot \boldsymbol n  & \textrm{on} \; \Gamma(t^n) \times(t^n, t^{n+1}), \label{nfluxstep2}\\ 
& \boldsymbol{\eta} \cdot  \boldsymbol{n}= \boldsymbol{U}\cdot  \boldsymbol{n},  \quad \boldsymbol{\eta} \cdot  \boldsymbol{\tau}= \boldsymbol{U}\cdot  \boldsymbol{\tau}, \quad  \boldsymbol{\xi} \cdot  \boldsymbol{\tau}= \boldsymbol{v} \cdot \boldsymbol{\tau} & \textrm{on} \; \Gamma(t^n) \times(t^n, t^{n+1}), 
\end{align}
while the fluid velocity in $\Omega^f(t^n)$ does not change in this step.
Since $\boldsymbol{\eta} \cdot  \boldsymbol{n}= \boldsymbol{U}|_{\Gamma(t^n)} \cdot  \boldsymbol{n}$ and $\quad \boldsymbol{\eta} \cdot  \boldsymbol{\tau}= \boldsymbol{U}|_{\Gamma(t^n)} \cdot  \boldsymbol{\tau}$, we have $\boldsymbol \eta = \boldsymbol U|_{\Gamma(t^n)}$ and $\boldsymbol \xi = \boldsymbol V|_{\Gamma(t^n)}$. Thus, we can rewrite conditions~\eqref{dyn1step2},~\eqref{dyn2step2}, and~\eqref{nfluxstep2} in the following way:
\begin{eqnarray}
& & \rho_{m} r_m \frac{\partial (\boldsymbol V \cdot  \boldsymbol \tau)}{\partial t} + \boldsymbol \tau \cdot \mathcal{L} \boldsymbol U  - J \boldsymbol \tau \cdot \boldsymbol {\sigma}^p \boldsymbol n  = 0 \quad \textrm{on} \; \Gamma(t^n)\times(t^n, t^{n+1}), \label{dyn1step2F} \\
& & \rho_{m} r_m \displaystyle\frac{\partial \boldsymbol V}{\partial t}\cdot \boldsymbol n  +\boldsymbol n \cdot \mathcal{L} \boldsymbol U- J \alpha p_p - J\boldsymbol n \cdot \boldsymbol {\sigma}^p \boldsymbol n  =0 \quad \textrm{on} \; \Gamma(t^n)\times(t^n, t^{n+1}), \label{dyn2step2F} \\
& & \alpha \frac{\partial \boldsymbol{U}}{\partial t} \cdot \boldsymbol n  =  \boldsymbol{v} \cdot \boldsymbol{n} + \kappa \nabla p_p  \cdot \boldsymbol n  \quad  \textrm{on} \; \Gamma(t^n) \times(t^n, t^{n+1}),
\end{eqnarray}
in which case they become Robin-type boundary conditions for the Biot system.

Finally, using the Newmark scheme for the elasticity equations, Step 3 reads as follows:
Find $\boldsymbol U^{n+1}, \boldsymbol V^{n+1}, \hat{\boldsymbol\eta}^{n+1}, \hat{\boldsymbol\xi}^{n+1}$ and $p_p^{n+1},$ with $\boldsymbol v^{n+1}$ computed in Step 1 and $\boldsymbol U^{n}, \boldsymbol V^{n}$ and $p_p^{n}$ computed in the previous time-step, such that $\boldsymbol \eta^{n+1} = \boldsymbol U|_{\Gamma(t^n)}^{n+1}, \boldsymbol \xi^{n+1} = \boldsymbol V|_{\Gamma(t^n)}^{n+1}$ and 
\begin{subequations}
{\small{
\begin{align}\label{step2}
& \rho_{p} d_t \boldsymbol V^{n+1} = \nabla \cdot \boldsymbol \sigma^p(\boldsymbol U^{n+1/2}, p_p^{n+1}) +\boldsymbol h & \textrm{in} \; \Omega^p(t^n),  \\ 
& s_0 d_t p_p^{n+1} + \alpha \nabla \cdot d_t \boldsymbol U^{n+1} - \nabla \cdot (\kappa \nabla p_p^{n+1})= s  & \textrm{in} \; \Omega^p(t^n),\\
& \rho_{p} (\boldsymbol V^{n+1/2}-d_t \boldsymbol U^{n+1})= 0 & \textrm{in} \; \Omega^p(t^n), \\ 
& J \boldsymbol \tau \cdot \boldsymbol {\sigma}^p(\boldsymbol U^{n+1}, p_p^{n+1}) \boldsymbol n = \rho_{m} r_m d_t \boldsymbol V^{n+1} \cdot  \boldsymbol \tau  + \boldsymbol \tau \cdot \mathcal{L} \boldsymbol U^{n+1/2} & \textrm{on} \; \Gamma(t^n), \label{S3bc1} \\
& J\boldsymbol n \cdot \boldsymbol {\sigma}^p(\boldsymbol U^{n+1}, p_p^{n+1}) \boldsymbol n = \rho_{m} r_m d_t \boldsymbol V^{n+1}\cdot \boldsymbol n +\boldsymbol n \cdot \mathcal{L} \boldsymbol U^{n+1/2} - J \alpha p_p^{n+1} & \textrm{on} \; \Gamma(t^n),  \label{S3bc2}\\
& \kappa \nabla p_p^{n+1} \cdot \boldsymbol n= \alpha d_t\boldsymbol{U}^{n+1} \cdot \boldsymbol n  -  \boldsymbol{v}^{n+1} \cdot \boldsymbol{n}    & \textrm{on} \; \Gamma(t^n) , 
 \end{align}}}
 \end{subequations}
with boundary conditions: 
\begin{equation*}
  \hat{\boldsymbol\eta}^{n+1}|_{\hat{x}=0,L}= 0, \quad p_p^{n+1} = 0 \quad \textrm{on} \; \Omega^p(t^n)\backslash \Gamma(t^n), \quad \boldsymbol U^{n+1}  = 0 \quad \textrm{on} \; \Gamma^p_{in}\cup\Gamma^p_{out}, 
\end{equation*}
$$ \boldsymbol {n}_{ext} \cdot \boldsymbol \sigma^E(\boldsymbol U^{n+1}) \boldsymbol {n}_{ext} =  -p_e \quad \textrm{on} \; \Gamma^p_{ext}, \quad  U_x^{n+1} = 0 \quad \textrm{on} \; \Gamma^p_{ext}.$$
Do $t^n=t^{n+1}$ and return to Step 1.

\vskip 0.1in
\noindent
{\bf Remark 2.}
In practice, the structure is usually handled in Lagrangian framework. Together with the hypothesis of ``small'' deformations we can assume the structure is linearly elastic, in which case 
we can easily recast Step 3 in the reference domain, where the boundary conditions~\eqref{S3bc1}-\eqref{S3bc2} simplify as follows:
\begin{align*}\label{step2linear}
& \boldsymbol e_1 \cdot \hat{\boldsymbol {\sigma}}^p(\hat{\boldsymbol U}^{n+1}, \hat{p}_p^{n+1}) \hat{\boldsymbol n} = \rho_{m} r_m d_t\hat{V}_x^{n+1}-C_2 \frac{\partial \hat U_y^{n+1/2}}{\partial \hat{x}} -C_1 \frac{\partial^2 \hat U_x^{n+1/2}}{\partial\hat {x}^2} & \textrm{on} \; \hat{\Gamma} \times(t^n, t^{n+1}), \\
& \boldsymbol e_2 \cdot \hat{\boldsymbol {\sigma}}^p(\hat{\boldsymbol U}^{n+1}, \hat{p}_p^{n+1}) \hat{\boldsymbol n} = \rho_{m} r_m d_t\hat{V}_y^{n+1} +C_0 \hat U_y^{n+1/2} +C_2 \frac{\partial \hat U_x^{n+1/2}}{\partial \hat{x}}- \alpha \hat{p}^{n+1}_p & \textrm{on} \; \hat{\Gamma} \times(t^n, t^{n+1}), 
 \end{align*}
where $\boldsymbol e_1$ and $\boldsymbol e_2$ are the Cartesian unit vectors.
\end{itemize}

The proposed scheme is an explicit loosely-coupled scheme where the first step consists of a fluid (Navier-Stokes) problem, and the second step consists of a poroelastic problem. 
Both sub-problems are solved with a Robin-type boundary conditions, which take into account thin-shell inertia and kinematic conditions implicitly. We note that the original monolithic problem becomes fully decoupled, and there are no sub-iterations needed between the two sub-problems. 

\vskip 0.1in
\noindent
{\bf Remark 3.}
One can apply additional splitting to Step 1 and Step 2 of the algorithm described above. Namely, 
the fluid problem described in Step 1 can be split into its viscous part (the Stokes equations for an incompressible fluid) and the pure advection part (incorporating the fluid and ALE advection simultaneously). The Biot system described in Step 2 can be split so the elastodynamics is treated separately from the pressure. 
For the details of possible Biot splitting strategies see~\cite{mikelic2012convergence} and the references therein.

\section{Weak formulation and stability}
In this section we write the variational formulation and prove the conditional stability of the loosely coupled scheme proposed in Section~2. For simplicity, we work out the analysis assuming that the displacement of the boundary is small enough and can be neglected. Under these
assumptions, domains $\Omega^f(t)$ and $\Omega^p(t)$ are fixed:
$$\Omega^f(t)=\hat{\Omega}^f, \quad \Omega^p(t)=\hat{\Omega}^p, \quad \forall t \in (0,T).$$
Although simplified, this problem still retains the main difficulties associated with the ``added-mass'' effect and the difficulties that partitioned schemes encounter when modeling fluid-porous medium coupling. Since from now on all the variable are defined on the fixed domain, we will drop the ``hat'' notation to avoid cumbersome expressions.

Let $t^n:=n \Delta t$ for $n = 1, \ldots, N,$ where $T=N \Delta t$ is the final time. Let the test function spaces $V^f, Q^f, V^p$ and $Q^p$ be defined as in~\eqref{Vf(t)},~\eqref{Q(t)},~\eqref{Vp}, and~\eqref{Qp} respectively. 

We introduce the following bilinear forms
\begin{eqnarray*}
 a_f(\boldsymbol v, \boldsymbol \varphi^f) &=& 2 \mu_f \int_{\Omega^f} \boldsymbol D(\boldsymbol v) : \boldsymbol D(\boldsymbol \varphi^f) d \boldsymbol x,  \label{af}\\
 b_f(p_f, \boldsymbol \varphi^f) &=&  \int_{\Omega^f} p_f \nabla \cdot \boldsymbol \varphi^f d \boldsymbol x, \label{bf} \\
 a_e(\boldsymbol U, \boldsymbol \varphi^p) &=& 2 \mu_p \int_{\Omega^p} \boldsymbol D(\boldsymbol U) : \boldsymbol D(\boldsymbol \varphi^p) d \boldsymbol x + \lambda_p \int_{\Omega^p} (\nabla \cdot \boldsymbol U)(\nabla \cdot \boldsymbol \varphi^p)  d \boldsymbol x,    \label{ae}\\
 a_p(p_p, \psi^p) &=&   \int_{\Omega^p} \kappa \nabla p_p \cdot \nabla \psi^p d \boldsymbol x, \label{ap} \\
 b_{ep}(p_p, \boldsymbol \varphi^p) &=& \alpha \int_{\Omega^p} p_p \nabla \cdot \boldsymbol \varphi^p d \boldsymbol x, \label{bp} \\
 a_m (\boldsymbol \eta, \boldsymbol \zeta) &=& -C_2 \int_0^L \frac{\partial \eta_y}{\partial x} \zeta_x dx +C_1 \int_0^L \frac{\partial \eta_x}{\partial x} \frac{\partial \zeta_x}{\partial x} dx +C_0 \int_0^L \eta_y \zeta_y dx + C_2 \int_0^L  \frac{\partial \eta_x}{\partial x} \zeta_y dx,   \label{ak}\\
 c_{fp}(p_p, \boldsymbol \varphi^f) &=& \int_{\Gamma} p_p \boldsymbol \varphi^f \cdot \boldsymbol n dx, \label{cfp} \\
 c_{ep}(p_p, \boldsymbol \varphi^p) &=& \alpha \int_{\Gamma} p_p \boldsymbol \varphi^p \cdot \boldsymbol n dx, \label{cep}
 \end{eqnarray*}
 and the trilinear form
 \begin{equation*}
   d_f(\boldsymbol v, \boldsymbol u, \boldsymbol \varphi) = \rho_f \int_{\Omega^f} (\boldsymbol v \cdot \nabla) \boldsymbol u \cdot \boldsymbol \varphi d \boldsymbol x. \label{df}
 \end{equation*}
\if 1=0
Now, the weak form of the loosely-coupled scheme presented in Section 2 is given as follows:
\begin{itemize}
 \item \textbf{Step 1.} Given $t^{n+1} \in (0,T], n = 0, \ldots, N-1,$ find $\boldsymbol{v} \in V^f$ and $p_f \in Q^f$ such that for all $(\boldsymbol \varphi^f,\psi^f) \in V^f \times Q^f$ and  $t\in (t^n, t^{n+1})$, with $p_p^n$ obtained at the previous time step:
\begin{equation}
\rho_f \int_{\Omega^f} \frac{\partial \boldsymbol v}{\partial t} \cdot \boldsymbol \varphi^f d\boldsymbol x+ a_f(\boldsymbol v, \boldsymbol \varphi^f)+\rho_{m}r_m \int_{\Gamma} \frac{\partial (\boldsymbol v \cdot \boldsymbol \tau)}{\partial t} \boldsymbol \varphi^f \cdot \boldsymbol \tau dx- b_f(p_f, \boldsymbol \varphi^f)+b_f(\psi^f, \boldsymbol v)+ c_{fp}(p^n_p, \boldsymbol \varphi^f) \nonumber
 \end{equation}
\begin{equation}
 = \int_{\Omega^f} \boldsymbol f^f\cdot \boldsymbol \varphi^f d\boldsymbol x
 + \int_0^R p_{in}(t) \varphi^f_x|_{x=0} dy. 
 \end{equation}
\item \textbf{Step 2.} Given $\boldsymbol v^{n+1}$ computed in Step 1, find $\boldsymbol U \in V^p, \boldsymbol V \in V^p, \boldsymbol\eta \in \hat{V}^m, \boldsymbol\xi \in \hat{V}^m,$ and $p_p \in Q^p,$ with $\boldsymbol\eta = \boldsymbol U|_{\Gamma}$ and $\boldsymbol\xi = \boldsymbol V|_{\Gamma}$, such that for all for all $(\boldsymbol \varphi^p,\boldsymbol \phi^p, \psi^p) \in V^p \times V^p \times Q^p$ and $t\in(t^n,t^{n+1}):$ 
 \begin{equation*}
\rho_{p} \int_{\Omega^p} \big(\boldsymbol V - \frac{\partial \boldsymbol U}{\partial t} \big) \cdot \boldsymbol \phi^p d \boldsymbol x +\rho_{p} \int_{\Omega^p} \frac{\partial \boldsymbol V}{\partial t}\cdot \boldsymbol \varphi^p d \boldsymbol x + a_e(\boldsymbol U, \boldsymbol \varphi^p)+ \int_{\Omega^p} s_0 \frac{\partial p_p}{\partial t} \psi^p d \boldsymbol x
\end{equation*}
\begin{equation*}
  +a_p(p_p, \psi^p)- b_{ep}(p_p, \boldsymbol \varphi^p) + b_{ep}(\psi_p, \frac{\partial \boldsymbol U}{\partial t})+\rho_{m} r_m \int_0^L \big(\boldsymbol V|_{\Gamma} - \frac{\partial (\boldsymbol U|_{\Gamma})}{\partial t} \big) \cdot \boldsymbol \phi^p|_{\Gamma} d \boldsymbol x
\end{equation*}
\begin{equation*}
 +\rho_{m} r_m \int_0^L \frac{\partial (\boldsymbol  V|_{\Gamma})}{\partial t}\cdot  \boldsymbol \varphi^p|_{\Gamma} dx + a_m (\boldsymbol U|_{\Gamma}, \boldsymbol \varphi^p|_{\Gamma})- c_{ep}(p_p, \boldsymbol \varphi^p)+ c_{ep}(\psi^p, \frac{\partial \boldsymbol U}{\partial t})
\end{equation*}
\begin{equation}
 - c_{fp}(\psi^p, \boldsymbol v^{n+1}) =  \int_{\Omega^p} \boldsymbol f^s \cdot \boldsymbol \varphi^p d \boldsymbol x+\int_{\Omega^p} s \psi d \boldsymbol x.
\end{equation}
\end{itemize}
To discretize the problem in time, we use the backward Euler scheme, and to discretize in space, we use the finite element method. Thus, we define the finite element spaces $V^f_h \subset V^f, Q^f_h \subset Q^f, V^p_h \subset V^p$ and $Q^p_h \subset Q^p$.  Finally, the fully discrete numerical scheme is given as follows:
\fi
To discretize the problem in space, we use the finite element method. Thus, we define the finite element spaces $V^f_h \subset V^f, Q^f_h \subset Q^f, V^p_h \subset V^p$ and $Q^p_h \subset Q^p$. The definition of these discrete spaces will be made precise at the beginning of Section~5. We assume that all the finite element initial conditions are equal to zero:
$$\boldsymbol v_h^0=0, \quad \boldsymbol U_h^0=0,\quad  \boldsymbol V_h^0=0, \quad \boldsymbol \eta_h^0=0, \quad \boldsymbol \xi_h^0=0,\quad  p_{p,h}^0=0.$$
Finally, the fully discrete numerical scheme is given as follows:
\begin{itemize}
 \item \textbf{Step 1.} Given $t^{n+1} \in (0,T], n = 0, \ldots, N-1,$ find $\boldsymbol{v}_h^{n+1} \in V_h^f$ and $p_{f,h}^{n+1} \in Q_h^f$ such that for all $(\boldsymbol \varphi_h^f,\psi_h^f) \in V_h^f \times Q_h^f$, with $p_{p,h}^n$ obtained at the previous time step:
\begin{equation}
\rho_f \int_{\Omega^f} d_t \boldsymbol v_h^{n+1} \cdot \boldsymbol \varphi_h^f d\boldsymbol x+ d_f(\boldsymbol v_h^{n+1}, \boldsymbol v_h^{n+1}, \boldsymbol \varphi^f_h)+ a_f(\boldsymbol v_h^{n+1}, \boldsymbol \varphi_h^f)+\rho_{m}r_m \int_{\Gamma} (d_t \boldsymbol v_h|_{\Gamma}^{n+1} \cdot \boldsymbol \tau) (\boldsymbol \varphi^f_h|_{\Gamma} \cdot \boldsymbol \tau) dx \nonumber
 \end{equation}
\begin{equation*}
- b_f(p^{n+1}_{f,h}, \boldsymbol \varphi_h^f) +b_f(\psi_h^f, \boldsymbol v_h^{n+1})+ c_{fp}(p^n_{p,h}, \boldsymbol \varphi_h^f) =\int_{\Omega^f} \boldsymbol f^{f}(t^{n+1})\cdot \boldsymbol \varphi^f_h d\boldsymbol x
 \end{equation*}
 \begin{equation}
 + \int_{\Gamma_{in}} p_{in}(t^{n+1}) \varphi^f_{x,h} dy. \label{S1discrete}
 \end{equation}
\item \textbf{Step 2.} Given $\boldsymbol v_h^{n+1}$ computed in Step 1, find $\boldsymbol U_h^{n+1} \in V_h^p, \boldsymbol V_h^{n+1} \in V_h^p, \boldsymbol\eta_h^{n+1} \in V_h^m, \boldsymbol\xi_h^{n+1} \in V_h^m$ and $p_{p,h}^{n+1} \in Q_h^p,$ with $\boldsymbol\eta_h^{n+1} = \boldsymbol U_h|_{\Gamma}^{n+1}$ and $\boldsymbol\xi_h^{n+1} = \boldsymbol V_h|_{\Gamma}^{n+1}$, such that for all $(\boldsymbol \varphi_h^p,\boldsymbol \phi_h^p, \psi_h^p) \in V_h^p \times V_h^p \times Q_h^p:$ 
 \begin{equation*}
\rho_{p} \int_{\Omega^p} (\boldsymbol V_h^{n+1/2} - d_t \boldsymbol U_h^{n+1} ) \cdot \boldsymbol \phi_h^p d \boldsymbol x +\rho_{p} \int_{\Omega^p} d_t \boldsymbol V_h^{n+1}\cdot \boldsymbol \varphi^p_h d \boldsymbol x + a_e(\boldsymbol U_h^{n+1/2}, \boldsymbol \varphi_h^p)+ \int_{\Omega^p} s_0 d_t p_{p,h}^{n+1} \psi^p_h d \boldsymbol x
\end{equation*}
\begin{equation*}
  +a_p(p^{n+1}_{p,h}, \psi^p_h)- b_{ep}(p^{n+1}_{p,h}, \boldsymbol \varphi^p_h)
  + b_{ep}(\psi_h^p, d_t \boldsymbol U_h^{n+1})+\rho_{m} r_m \int_0^L (\boldsymbol V_h|_{\Gamma}^{n+1/2} - d_t \boldsymbol U_h|_{\Gamma}^{n+1}) \cdot \boldsymbol \phi_h^p|_{\Gamma} d \boldsymbol x 
\end{equation*}
\begin{equation*}
+\rho_{m} r_m \int_0^L d_t \boldsymbol  V_h|_{\Gamma}^{n+1} \cdot  \boldsymbol \varphi_h^p|_{\Gamma} dx + a_m (\boldsymbol U_h|_{\Gamma}^{n+1/2}, \boldsymbol \varphi_h^p|_{\Gamma})- c_{ep}(p^{n+1}_{p,h}, \boldsymbol \varphi^p_h) + c_{ep}(\psi^p_h, d_t \boldsymbol U_h^{n+1})
\end{equation*}
\begin{equation}
- c_{fp}(\psi^p_h, \boldsymbol v_h^{n+1}) =  -\int_{\Gamma^p_{ext}} p_e \boldsymbol \varphi^p_{y,h} dx +\int_{\Omega^p} \boldsymbol f^s(t^{n+1}) \cdot \boldsymbol \varphi^p_h d \boldsymbol x 
  +\int_{\Omega^p} s(t^{n+1}) \psi^p_h d \boldsymbol x. \label{S2discrete}
\end{equation}

\end{itemize}


\subsection{Stability analysis}
To present our results in a more compact manner, in the analysis we study the Stokes equations instead of the Navier-Stokes equations. Handling the convective term in the Navier-Stokes equations can be done using classical approaches, see for example~\cite{temam2001navier,BorSunMulti}. 
%
%
Let us introduce the following time discrete norms:
 \begin{equation*}
  ||\boldsymbol \varphi||_{l^2(0,T; H^k(S))} = \bigg(\sum_{n=0}^{N-1} ||\boldsymbol \varphi^{n+1}||^2_{H^k(S)} \Delta t \bigg)^{1/2}, \quad ||\boldsymbol \varphi ||_{l^{\infty}(0,T; H^k(S))} = \max_{0 \le n \le N} ||\boldsymbol \varphi^n||_{H^k(S)},
 \end{equation*}
where $S \in \{\Omega^f, \Omega^p, (0,L)\}$.
Let $\mathcal{E}_f^n$ denote the discrete energy of the fluid problem, $\mathcal{E}_p^n$ denote the discrete energy of the Biot problem, and $\mathcal{E}_m^n$ denote the discrete energy of the Koiter membrane at time level $n$, defined respectively by
\begin{eqnarray}
 \mathcal{E}_f^n &=& \frac{\rho_f}{2} ||\boldsymbol v_h^n||^2_{L^2(\Omega^f)}+ \frac{\rho_{m} r_m}{2} ||\boldsymbol v_h^n \cdot \boldsymbol \tau||^2_{L^2(\Gamma)} \\
 \mathcal{E}_p^n &=& \frac{\rho_{p}}{2} ||\boldsymbol V_h^n||^2_{L^2(\Omega^p)} + \mu_p ||D(\boldsymbol U_h^n)||^2_{L^2(\Omega^p)} + \frac{\lambda_p}{2}||\nabla \cdot \boldsymbol U_h^n||^2_{L^2(\Omega^p)}\nonumber +\frac{s_0}{2} ||p_{p,h}^n||^2_{L^2(\Omega^p)},
\end{eqnarray}
 \begin{eqnarray}
& \mathcal{E}_m^n = \displaystyle\frac{\rho_{m} r_m}{2} ||\xi_x^n ||^2_{L^2(0,L)}+\displaystyle\frac{\rho_{m} r_m}{2} ||\xi_y^n ||^2_{L^2(0,L)} \nonumber \\
&  +\displaystyle r_m \left[
  4 \mu_m \bigg|\bigg|\frac{\hat\eta_y}{R}\bigg|\bigg|^2_{L^2(0,L)} +
  4 \mu_m \bigg|\bigg|\frac{\partial \hat\eta_x}{\partial \hat{x}} \bigg|\bigg|^2_{L^2(0,L)}
  +\frac{4 \mu_m \lambda_m}{\lambda_m+2\mu_m} \bigg|\bigg|\frac{\partial \hat\eta_x}{\partial \hat{x}}+ \frac{\hat\eta_y}{R} \bigg|\bigg|^2_{L^2(0,L)}
 \right].
 \nonumber
\end{eqnarray}
The stability of the loosely-coupled scheme~\eqref{S1discrete}-\eqref{S2discrete} is stated in the following result. The constants that appear in~\eqref{theorem:stability} are defined in the Appendix. 
\begin{theorem}
 Assume that the fluid-poroelastic system is isolated, i.e. $p_{in}=0, p_e=0, \boldsymbol f^f=0, \boldsymbol f^s=0$ and $s=0$. Let $\{(\boldsymbol v_h^n, p_{p,h}^n, \boldsymbol V_h^n, \boldsymbol U_h^n, \boldsymbol \xi_h^n, \boldsymbol \eta_h^n, p_{p,h}^n) \}_{0 \le n \le N}$ be the solution of~\eqref{S1discrete}-\eqref{S2discrete}.
Then, under the condition
\begin{equation}
 \bigg(2 \mu_f - \frac{C^2_K C_{TI} C_T^2 C_{PF} \Delta t}{ s_0 h} \bigg) \ge \gamma >0
\quad
\text{i.e.}
\quad
 \Delta t < \displaystyle\frac{2 \mu_f s_0 h}{C^2_K C_{TI} C^2_T C_{PF}}, \label{CFL}
\end{equation}
the following estimate holds:
\begin{equation*}
 \mathcal{E}_f^N+\mathcal{E}_p^N+\mathcal{E}_m^N   
 + \frac{\Delta t}{2} \rho_f ||d_t\boldsymbol v_h||^2_{l^2(0,T;L^2(\Omega^f))}
 + \frac{\Delta t}{2} \rho_{m} r_m ||d_t\boldsymbol v_h\cdot \boldsymbol \tau||^2_{l^2(0,T;L^2(\Gamma))}
 \end{equation*}
\begin{equation*}
+\gamma ||\boldsymbol v_h||^2_{l^2(0,T;H^1(\Omega^f))} + \delta \Delta t ||p_{f,h}||_{l^2(0,T;L^2(\Omega^f))}^2
 +\frac{\Delta t}{4} s_0 ||d_t p_{p,h}||^2_{l^2(0,T;L^2(\Omega^p))}
 + ||\sqrt{\kappa} p_{p,h}||^2_{l^2(0,T;H^1(\Omega^p))} 
\end{equation*}  
\begin{equation}
 \leq  \mathcal{E}_f^0 + \mathcal{E}_p^0  + \mathcal{E}_m^0.  \label{theorem:stability}
\end{equation}
 \end{theorem}
 \begin{proof}
  To prove the energy estimate, we test the problem~\eqref{S1discrete} with $(\boldsymbol \varphi^f_h, \psi_h^f) = (\boldsymbol v^{n+1}_h, p_{f,h}^{n+1}),$ and problem~\eqref{S2discrete} with
  $(\boldsymbol \varphi^p_h, \boldsymbol \phi^p_h, \psi_h^p) = (d_t \boldsymbol U^{n+1}_h, d_t \boldsymbol V^{n+1}_h, p_{p,h}^{n+1}).$
  Then, after adding them together and multiplying by $\Delta t$   we get
\begin{equation*}\label{stab_fp_0}
 \mathcal{E}_f^{n+1}+\mathcal{E}_p^{n+1}+\mathcal{E}_m^{n+1}+ \frac{\rho_f}{2}
||\boldsymbol v_h^{n+1}-\boldsymbol v_h^n||^2_{L^2(\Omega^f)}  +2 \mu_f \Delta t ||D(\boldsymbol v_h^{n+1})||^2_{L^2(\Omega^f)}+ \frac{\rho_{m} r_m}{2}
||(\boldsymbol v_h^{n+1}-\boldsymbol v_h^n)\cdot \boldsymbol \tau||^2_{L^2(\Gamma)}
 \end{equation*}
\begin{equation*}
 +\frac{s_0}{2}||p_{p,h}^{n+1}-p_{p,h}^n||^2_{L^2(\Omega^p)} + \Delta t ||\sqrt{\kappa} \nabla p_{p,h}^{n+1}||^2_{L^2(\Omega^p)} 
\leq \Delta t c_{fp}(p_{p,h}^{n+1}, \boldsymbol v^{n+1}_h)-\Delta t c_{fp}(p_{p,h}^n, \boldsymbol v^{n+1}_h)+\mathcal{E}_f^n + \mathcal{E}_p^n+ \mathcal{E}_m^n.
\end{equation*}
The term $\Delta t c_{fp}(p_{p,h}^{n+1}-p_{p,h}^n, \boldsymbol v^{n+1}_h)$ arises in classical partitioned schemes for Navier Stokes/Stokes-Darcy coupling, and has been previously 
addressed in~\cite{layton2012long}. Following the similar approach as in~\cite{layton2012long}, we can estimate the interface term using Cauchy-Schwarz inequality~\eqref{CS}, Young's inequality~\eqref{youngs} (for $\epsilon_1 > 0$), and the local trace-inverse inequality~\eqref{traceInverse} in the following way:
\begin{equation*}
 \Delta t c_{fp}(p_{p,h}^{n+1}-p_{p,h}^n, \boldsymbol v^{n+1}_h) = \Delta t \int_{\Gamma}(p_{p,h}^{n+1}-p_{p,h}^n) \boldsymbol v^{n+1}_h \cdot \boldsymbol n dx
\end{equation*}
\begin{equation*}
 \leq \frac{\epsilon_1 \Delta t}{2} ||p_{p,h}^{n+1}-p_{p,h}^n||^2_{L^2(\Gamma)}+\frac{ \Delta t}{2 \epsilon_1} ||\boldsymbol v^{n+1}_h ||^2_{L^2(\Gamma)}
\end{equation*}
\begin{equation*}
 \leq \frac{\epsilon_1 \Delta t C_{TI}}{2 h} ||p_{p,h}^{n+1}-p_{p,h}^n||^2_{L^2(\Omega)}+\frac{ \Delta t}{2 \epsilon_1} ||\boldsymbol v^{n+1}_h ||^2_{L^2(\Gamma)}.
\end{equation*}
Finally, using trace inequality~\eqref{trace}, Poincar\'e-Friedrichs inequality~\eqref{PF}, and Korn's inequality~\eqref{korn}, we have
\begin{equation}\label{stab_fp_0_bis}
 \Delta t c_{fp}(p_{p,h}^{n+1}-p_{p,h}^n, \boldsymbol v^{n+1}_h) \leq \frac{\epsilon_1 \Delta t C_{TI}}{2 h} ||p_{p,h}^{n+1}-p_{p,h}^n||^2_{L^2(\Omega)}+\frac{ \Delta t C_T^2 C^2_K C_{PF}}{2 \epsilon_1} ||D(\boldsymbol v^{n+1}_h )||^2_{L^2(\Omega)}.
 \end{equation}
 In order to recover control on the pressure in the fluid domain, we exploit the \emph{inf-sup} stability of the approximation spaces $V_h^f$ and $Q_h^f$. 
Namely, spaces $V_h^f$ and $Q_h^f$ are~\emph{inf-sup} stable provided
\begin{equation}
 \inf_{p^{n+1}_{f,h}\in Q^f_h} \sup_{\boldsymbol \varphi^f \in V^f_h} \frac{b_f(p^{n+1}_{f,h},\boldsymbol \varphi_h^f)}{||\boldsymbol \varphi_h^f||_{H^1(\Omega^f)} ||p^{n+1}_{f,h}||_{L^2(\Omega^f)}} = \beta_f >0. \label{infsup}
\end{equation}
Combining the \emph{inf-sup} condition~\eqref{infsup} with \eqref{S1discrete} tested with $\psi_h^f=0$ we obtain,
\begin{equation}\label{stab_fp_1}
\beta_f \|p_{f,h}^{n+1}\|_{L^2(\Omega^f)} 
\leq \sup\limits_{\boldsymbol \varphi_h^f \in V_h^f}
\frac{\sum_{k=1,2}\mathcal{T}_k(\boldsymbol \varphi_h^f)}{\|\boldsymbol \varphi_h^f\|_{H^1(\Omega^f)}}
\end{equation}
where $\beta_f >0$ is a constant independent of the mesh characteristic size and $\mathcal{T}_k(\boldsymbol \varphi_h^f)$ is a shorthand notation for the following terms,
\begin{align*}
\mathcal{T}_1(\boldsymbol \varphi_h^f) &:=
\rho_f \int_{\Omega^f} d_t \boldsymbol v_h^{n+1} \cdot \boldsymbol \varphi_h^f d\boldsymbol x
+ \rho_{m}r_m \int_{\Gamma} (d_t \boldsymbol v_h^{n+1} \cdot \boldsymbol \tau) (\boldsymbol \varphi_h^f \cdot \boldsymbol \tau) d x,
\\
\mathcal{T}_2(\boldsymbol \varphi_h^f) &:=
a_f(\boldsymbol v_h^{n+1}, \boldsymbol \varphi_h^f)
+ c_{fp}(p_{p,h}^{n}, \boldsymbol \varphi_h^f).
\end{align*}
Exploiting Cauchy-Schwarz~\eqref{CS} and trace~\eqref{trace} inequalities, we obtain the following upper bounds,
\begin{equation*}
\sup\limits_{\boldsymbol \varphi_h^f \in V_h^f}
\frac{\mathcal{T}_1(\boldsymbol \varphi_h^f)}{\|\boldsymbol \varphi_h^f\|_{H^1(\Omega^f)}}
\leq C_T \Big(\rho_f \|d_t \boldsymbol v_h^{n+1}\|_{L^2(\Omega^f)}
+ \rho_{m}r_m \|d_t \boldsymbol v_h^{n+1}\|_{L^2(\Gamma)} \Big),
\end{equation*}
\begin{equation*}
\sup\limits_{\boldsymbol \varphi_h^f \in V_h^f}
\frac{\mathcal{T}_2(\boldsymbol \varphi_h^f)}{\|\boldsymbol \varphi_h^f\|_{H^1(\Omega^f)}}
\leq 2 \mu_f \|D(\boldsymbol v_h^{n+1})\|^2_{L^2(\Omega^f)}
+ C_T C_{PF} \kappa^{-1} \|\sqrt{\kappa} \nabla p_{p,h}^{n}\|_{L^2(\Omega^p)}^2.
\end{equation*}
Let us now multiply \eqref{stab_fp_1} as well as the bounds for $\mathcal{T}_k(\boldsymbol \varphi_h^f)$ by $\epsilon_2 \Delta t^2$ and combine the resulting inequality with \eqref{stab_fp_0} and \eqref{stab_fp_0_bis} to get,
\begin{equation*}
 \mathcal{E}_f^{n+1}+\mathcal{E}_p^{n+1}+\mathcal{E}_m^{n+1}
 + \frac{\Delta t^2}{2}(1 - \epsilon_2C_T^2) \rho_f||d_t \boldsymbol v_h^{n+1}||^2_{L^2(\Omega^f)}  
\end{equation*}
\begin{equation*}
 + \frac{\Delta t^2}{2} \rho_{m} r_m (1-\epsilon_2C_T^2) ||d_t \boldsymbol v_h^{n+1}\cdot \boldsymbol \tau||^2_{L^2(\Gamma)}
 + 2 \mu_f C_K (1 - (2\epsilon_1)^{-1} - \epsilon_2 \mu_f) \Delta t ||\boldsymbol v_h^{n+1}||^2_{H^1(\Omega^f)}
\end{equation*}
\begin{equation*}
 + \frac{\Delta t^2}{2} \left(s_0 - \frac{\epsilon_1 \Delta t C_{TI}}{2 h} \right)  ||d_t p_{p,h}^{n+1}||^2_{L^2(\Omega^p)} 
 + \epsilon_2 \beta_f \Delta t^2 \|p_{f,h}^{n+1}\|_{L^2(\Omega^f)}
\end{equation*}
\begin{equation*}\label{stab_fp_3}
\Delta t \|\sqrt{\kappa} \nabla p_{p,h}^{n+1}\|_{L^2(\Omega^p)}^2 
- \epsilon_2 \Delta t^2 \frac{C_T^2 C_{PF}^2}{\kappa^2} \|\sqrt{\kappa} \nabla p_{p,h}^{n}\|_{L^2(\Omega^p)}^2
\leq \mathcal{E}_f^n + \mathcal{E}_p^n+\mathcal{E}_m^{n}.
\end{equation*}
After summing up with respect to the time index $n$ we observe that
\begin{multline*}
\Delta t  \sum_{n=0}^{N-1} \Big[ \|\sqrt{\kappa} \nabla p_{p,h}^{n+1}\|_{L^2(\Omega^p)}^2 
- \epsilon_2 \Delta t^2 \frac{C_T^2 C_{PF}^2}{\kappa^2} \|\sqrt{\kappa} \nabla p_{p,h}^{n}\|_{L^2(\Omega^p)}^2 \Big]
\\
=
\Delta t \sum_{n=1}^{N-1}  \left(1-\epsilon_2 \Delta t \frac{C_T^2 C_{PF}^2}{\kappa^2} \right)\|\sqrt{\kappa} \nabla p_{p,h}^{n}\|_{L^2(\Omega^p)}^2 
+  \Delta t \|\sqrt{\kappa} \nabla p_{p,h}^{N}\|_{L^2(\Omega^p)}^2 .
\end{multline*}

 By setting $\epsilon_1 = \displaystyle\frac{s_0 h}{2 \Delta t C_{TI}}$,  
 and $\epsilon_2 = \frac12 \min \left( \frac{1}{\mu_f}, \Delta t \frac{\kappa^2}{C_T^2 C_{PF}^2} \right)$
 we prove the desired estimate.
 \end{proof}
 
 
\section{Error Analysis}
In this section, we analyze the convergence rate of the proposed method. For the spatial approximation, we apply Lagrangian finite elements of polynomial degree $k$ for all the variables, except for the fluid pressure, for which we use elements of degree $s<k$. 
We assume the regularity assumptions reported in Lemma~1 of the Appendix are satisfied and that our FEM spaces satisfy the usual approximation properties, as well as fluid velocity-pressure spaces  satisfy the discrete \emph{inf-sup} condition~\eqref{infsup}. 
 We will consider the fluid problem over the discretely divergence free velocity space 
\begin{equation}
X^f_h: = \{\boldsymbol v_h \in V^f_h | \; (\psi^f_h, \nabla \cdot \boldsymbol v_h)_{\Omega^f} = 0, \quad \textrm{for all} \; \psi^f_h \in Q^f_h \}.
\end{equation}
Let $S_h $ be an orthogonal projection operator with respect to $a_f(\cdot, \cdot)$, onto $X_h^f$, given by
\begin{align}
& a_f(\boldsymbol v - S_h \boldsymbol v, \boldsymbol \varphi^f_h) = 0 & \forall \boldsymbol \varphi^f_h \in  X^f_h,  \label{sh1}
\end{align}
$P_h$ be the Lagrangian interpolation operator onto $V_h^p$, and let $\Pi^{f/p}_h$ be 
the $L^2$-orthogonal projection onto $Q_h^{f/p}$, satisfying
\begin{equation}
(p_{r} - \Pi^{r}_h p_{r}, \psi_h) = 0, \quad \forall \psi_h \in Q^{r}_h, r \in \{f,p\}. \label{pi}
\end{equation}
Then, using piecewise polynomials of degree $k$, the finite element theory for Lagrangian and $L^2$ projections~\cite{ciarlet1978finite} gives the  classical approximation properties reported in Lemma \ref{approximation_properties}. 
Since $P_h$ is the Lagrangian interpolant, so is its trace on $\Gamma$. Therefore, we inherit optimal approximation properties also on this subset. We refer to Lemma \ref{approximation_properties} for a precise statement of these properties.

To analyze the error of our numerical scheme, we start by subtracting~\eqref{S1discrete}-\eqref{S2discrete} from the continuous problem. 
 We assume that the continuous fluid velocity lives in the space $\{\boldsymbol v \in V^f | \; \nabla \cdot \boldsymbol v = 0 \}$. 
This gives rise to the following error equations:
\begin{equation}
\rho_f \int_{\Omega^f}(d_t \boldsymbol v^{n+1} - d_t\boldsymbol v_h^{n+1}) \cdot \boldsymbol \varphi_h^f d\boldsymbol x+ a_f(\boldsymbol v^{n+1} -\boldsymbol v_h^{n+1}, \boldsymbol \varphi_h^f)- b_f(p_f^{n+1}, \boldsymbol \varphi_h^f) + c_{fp}(p_p^n-p^n_{p,h}, \boldsymbol \varphi_h^f)\nonumber
 \end{equation}
\begin{equation}
 +  \rho_{m}r_m \int_{\Gamma} (d_t \boldsymbol v^{n+1} \cdot \boldsymbol \tau - d_t \boldsymbol v_h^{n+1} \cdot \boldsymbol \tau) (\boldsymbol \varphi_h^f \cdot \boldsymbol \tau) d x = \mathcal{R}^{n+1}_f(\boldsymbol \varphi_h^f),\quad  \forall \boldsymbol \varphi_h^f \in X^f_h, 
 \end{equation}
 \begin{equation*}
\rho_{p} \int_{\Omega^p} (\boldsymbol V^{n+1/2} - \boldsymbol V_h^{n+1/2} -(d_t \boldsymbol U^{n+1}- d_t \boldsymbol U_h^{n+1} )) \cdot \boldsymbol \phi_h^p d \boldsymbol x +\rho_{p}\int_{\Omega^p} (d_t \boldsymbol V^{n+1} - d_t \boldsymbol V_h^{n+1} )\cdot \boldsymbol \varphi^p_h d \boldsymbol x
\end{equation*}
\begin{equation*}
 + a_e(\boldsymbol U^{n+1/2} - \boldsymbol U_h^{n+1/2}, \boldsymbol \varphi_h^p)
 + \int_{\Omega^p} s_0 ( d_t p_{p}^{n+1} -d_t p_{p,h}^{n+1}) \psi^p_h d \boldsymbol x  +a_p(p^{n+1}_{p}-p^{n+1}_{p,h}, \psi^p_h) 
\end{equation*}
\begin{equation*}
- b_{ep}(p^{n+1}_{p}-p^{n+1}_{p,h}, \boldsymbol \varphi^p_h)+ b_{ep}(\psi_h^p,d_t \boldsymbol U^{n+1}- d_t \boldsymbol U_h^{n+1})+\rho_{m} r_m \int_{\Gamma} (\boldsymbol V^{n+1/2}-\boldsymbol V_h^{n+1/2} )\cdot \boldsymbol \phi_h^p dx 
\end{equation*}
\begin{equation*}
-\rho_{m}r_m \int_{\Gamma} (d_t \boldsymbol U^{n+1} - d_t \boldsymbol U_h^{n+1}) \cdot \boldsymbol \phi_h^p d x +\rho_{m} r_m \int_{\Gamma} (d_t \boldsymbol  V^{n+1}-d_t \boldsymbol  V_h )\cdot  \boldsymbol \varphi_h^p dx 
\end{equation*}
\begin{equation*}
 + a_m (\boldsymbol U|_{\Gamma}^{n+1/2}-\boldsymbol U_h|_{\Gamma}^{n+1/2}, \boldsymbol \varphi_h^p|_{\Gamma})- c_{ep}(p^{n+1}_{p}-p^{n+1}_{p,h}, \boldsymbol \varphi^p_h)+ c_{ep}(\psi^p_h, d_t \boldsymbol U^{n+1} -d_t \boldsymbol U_h^{n+1})
\end{equation*}
\begin{equation}
 - c_{fp}(\psi^p_h, \boldsymbol v^{n+1}-\boldsymbol v_h^{n+1}) =  \mathcal{R}_s^{n+1}(\boldsymbol \varphi_h^p)+\mathcal{R}_v^{n+1}(\boldsymbol \phi_h^p)+\mathcal{R}_p^{n+1}(\psi_h^p), \quad \forall (\boldsymbol \varphi_h^p,\boldsymbol \phi_h^p, \psi_h^p) \in V^p_h\times V^p_h \times Q^p_h,
\end{equation}
where the consistency errors are given by
\begin{eqnarray}
 \mathcal{R}_f^{n+1}(\boldsymbol \varphi_h^f) &=& \rho_f \int_{\Omega^f}(d_t \boldsymbol v^{n+1} -\partial_t \boldsymbol v^{n+1}) \cdot \boldsymbol \varphi_h^f d \boldsymbol x +c_{fp} (p_p^n - p^{n+1}_p, \boldsymbol \varphi^f_h) \nonumber \\
 & &   +\rho_{m}r_m \int_{\Gamma} (d_t \boldsymbol v^{n+1} \cdot \boldsymbol \tau-\partial_t \boldsymbol v^{n+1}\cdot \boldsymbol \tau) (\boldsymbol \varphi_h^f \cdot \boldsymbol \tau) d x,  \nonumber\\
 \mathcal{R}_s^{n+1}(\boldsymbol \varphi_h^p) &=&  \rho_{p} \int_{\Omega^p}(d_t\boldsymbol V^{n+1}-\partial_t \boldsymbol V^{n+1}) \cdot \boldsymbol \varphi_h^p d \boldsymbol x+a_e(\frac{\boldsymbol U^{n}-\boldsymbol U^{n+1}}{2}, \boldsymbol \varphi^p_h) \nonumber \\
 & &  +\rho_{m} r_m \int_{\Gamma} (d_t \boldsymbol V^{n+1}-\partial_t \boldsymbol V^{n+1}) \cdot \boldsymbol \varphi_h^p d x+a_m(\frac{\boldsymbol U|_{\Gamma}^{n}-\boldsymbol U|_{\Gamma}^{n+1}}{2}, \boldsymbol \varphi^p_h|_{\Gamma}), \nonumber
 \end{eqnarray}
 \begin{eqnarray}
 \mathcal{R}_v^{n+1}(\boldsymbol \phi_h^p) &=& \rho_{p} \int_{\Omega^p}\frac{\boldsymbol V^{n}-\boldsymbol V^{n+1}}{2} \cdot \boldsymbol \phi_h^p d \boldsymbol x -\rho_{p} \int_{\Omega^p}(d_t \boldsymbol U^{n+1}-\partial_t \boldsymbol U^{n+1}) \cdot \boldsymbol \phi_h^p d \boldsymbol x \nonumber \\
 & &  +\rho_{m}r_m \int_{\Gamma}\frac{\boldsymbol V^{n}-\boldsymbol V^{n+1}}{2} \cdot \boldsymbol \phi_h^p d \boldsymbol x-\rho_{m}r_m \int_{\Gamma}(d_t \boldsymbol U^{n+1}-\partial_t \boldsymbol U^{n+1}) \cdot \boldsymbol \phi_h^p d \boldsymbol x, \nonumber \\
  \mathcal{R}_p^{n+1}(\psi_h^p) &=&  \int_{\Omega^p} s_0 (d_t p_p^{n+1} - \partial_t p_p^{n+1})\psi_h^p d\boldsymbol x + b_{ep}(\psi_h^p, d_t \boldsymbol U^{n+1}-\partial_t \boldsymbol U^{n+1}) \nonumber \\
  & & +c_{ep}(\psi_h^p, d_t \boldsymbol U^{n+1}-\partial_t \boldsymbol U^{n+1}).
 \end{eqnarray}
Let us split the error of the method into the approximation error $\theta_r$ and the truncation error $\delta_r$, with $r=f,fp,u,v,p$, as follows:
\begin{eqnarray*}
 e_f^{n+1} &=& \boldsymbol v^{n+1} - \boldsymbol v_h^{n+1} =  (\boldsymbol v^{n+1} -S_h \boldsymbol v^{n+1})+(S_h \boldsymbol v^{n+1}- \boldsymbol v_h^{n+1}) =: \theta_f^{n+1} + \delta_f^{n+1},  \\
  e_{fp}^{n+1} &=& p_f^{n+1} - p_{f,h}^{n+1} =  (p_f^{n+1} -\Pi^f_h p_f^{n+1})+(\Pi^f_h p_f^{n+1}- p_{f,h}^{n+1}) =: \theta_{fp}^{n+1} + \delta_{fp}^{n+1},  \\
 e_u^{n+1} &=& \boldsymbol U^{n+1} - \boldsymbol U_h^{n+1} =  (\boldsymbol U^{n+1} -P_h \boldsymbol U^{n+1})+(P_h \boldsymbol U^{n+1}- \boldsymbol U_h^{n+1}) =: \theta_u^{n+1} + \delta_u^{n+1}, \\
 e_v^{n+1} &=&  \boldsymbol V^{n+1} - \boldsymbol V_h^{n+1} =  (\boldsymbol V^{n+1} -P_h \boldsymbol V^{n+1})+(P_h \boldsymbol V^{n+1}- \boldsymbol V_h^{n+1}) =: \theta_v^{n+1} + \delta_v^{n+1}, \\
 e_p^{n+1} &=&  p_p^{n+1} - p_{p,h}^{n+1} = (p_p^{n+1}-\Pi^p_h p_p^{n+1})+(\Pi^p_h p_p^{n+1}-p_{p,h}^{n+1})=: \theta_p^{n+1} + \delta_p^{n+1}.
\end{eqnarray*}
Note that $\boldsymbol \eta^{n+1} - \boldsymbol \eta_h^{n+1}= \theta_u|_{\Gamma}^{n+1}+\delta_u|_{\Gamma}^{n+1}$ and $\boldsymbol \xi^{n+1} - \boldsymbol \xi_h^{n+1}= \theta_v|_{\Gamma}^{n+1}+\delta_v|_{\Gamma}^{n+1}$.

Our plan is to rearrange the terms in the error equations so that we have the truncation errors on the left hand side, and the consistency and interpolation errors on the right hand side.
After that, we will choose $\boldsymbol \varphi^f_h = \delta_f^{n+1},  \boldsymbol \varphi^p_h = d_t \delta_u^{n+1}, \boldsymbol \phi^p_h = d_t \delta_v^{n+1},$ and $\psi^p_h =  \delta_p^{n+1}$, and use the stability estimate for the truncation errors. Finally, we will
bound the remaining terms, and use the triangle inequality to get the error estimates for $ e_f, e_u, e_v,$ and $e_p$. 

Rearranging the terms in the error equations, and using properties~\eqref{sh1} and~\eqref{pi} of the projection operators $S_h$ and $\Pi_h$, respectively, we have
\begin{equation}
\rho_f \int_{\Omega^f}d_t \delta_f^{n+1} \cdot \boldsymbol \varphi_h^f d\boldsymbol x+ a_f(\delta_f^{n+1}, \boldsymbol \varphi_h^f)+ c_{fp}(\delta_p^n, \boldsymbol \varphi_h^f) +  \rho_{m}r_m \int_{\Gamma} (d_t \delta_f^{n+1} \cdot \boldsymbol \tau)(\boldsymbol \varphi_h^f \cdot \boldsymbol \tau) d x \nonumber
 \end{equation}
\begin{equation}
 = \mathcal{R}^{n+1}_f(\boldsymbol \varphi_h^f) 
+ b_f(p_f^{n+1}, \boldsymbol \varphi_h^f)
-\rho_f \int_{\Omega^f} d_t\theta_f^{n+1} \cdot \boldsymbol \varphi_h^f d\boldsymbol x
 -\rho_{m}r_m \int_{\Gamma} (d_t \theta_f^{n+1} \cdot \boldsymbol \tau (\boldsymbol \varphi_h^f \cdot \boldsymbol \tau) d x \quad \forall \boldsymbol \varphi_h^f \in X^f_h, \label{error_fluid}
 \end{equation}
 \begin{equation*}
\rho_{p} \int_{\Omega^p} (\delta_v^{n+1/2} - d_t \delta_u^{n+1}) \cdot \boldsymbol \phi_h^p d \boldsymbol x +\rho_{p}\int_{\Omega^p} d_t \delta_v^{n+1}\cdot \boldsymbol \varphi^p_h d \boldsymbol x + a_e(\delta_u^{n+1/2}, \boldsymbol \varphi_h^p) + \int_{\Omega^p} s_0 d_t \delta_{p}^{n+1} \psi^p_h d \boldsymbol x 
\end{equation*}
\begin{equation*}
 +a_p(\delta^{n+1}_{p}, \psi^p_h)- b_{ep}(\delta^{n+1}_{p}, \boldsymbol \varphi^p_h) + b_{ep}(\psi_h^p,d_t \delta_u^{n+1})
+\rho_{m} r_m \int_{\Gamma} (\delta_v^{n+1/2} -d_t \delta_u^{n+1})   \cdot \boldsymbol \phi_h^p d x
\end{equation*}
\begin{equation*}
 +\rho_{m} r_m \int_{\Gamma} d_t \delta_v^{n+1} \cdot  \boldsymbol \varphi_h^p dx  + a_m (\delta_u|_{\Gamma}^{n+1/2}, \boldsymbol \varphi_h^p|_{\Gamma})- c_{ep}(\delta^{n+1}_{p}, \boldsymbol \varphi^p_h)+ c_{ep}(\psi^p_h, d_t \delta_u^{n+1}) - c_{fp}(\psi^p_h, \delta_v^{n+1})
\end{equation*}
\begin{equation*}
 =  \mathcal{R}_s^{n+1}(\boldsymbol \varphi_h^p)+\mathcal{R}_v^{n+1}(\boldsymbol \phi_h^p)+\mathcal{R}_p^{n+1}(\psi_h^p) +\rho_{p}\int_{\Omega^p} d_t \theta_u^{n+1} \cdot \boldsymbol \phi_h^p d \boldsymbol x-\rho_{p}\int_{\Omega^p} \theta_v^{n+1/2} \cdot \boldsymbol \phi_h^p d \boldsymbol x
\end{equation*}
 \begin{equation*}
+\rho_{p}\int_{\Omega^p} d_t \theta_v^{n+1} \cdot \boldsymbol \varphi_h^p d \boldsymbol x-a_e(\theta_u^{n+1/2}, \boldsymbol \varphi_h^p) -a_p(\theta^{n+1}_{p}, \psi^p_h)+ b_{ep}(\theta^{n+1}_{p}, \boldsymbol \varphi^p_h) - b_{ep}(\psi_h^p,d_t \theta_u^{n+1})
\end{equation*}
\begin{equation*}
-\rho_{m} r_m \int_{\Gamma} (\theta_v^{n+1/2} -d_t \theta_u^{n+1})  \cdot \boldsymbol \phi_h^p d x-\rho_{m} r_m \int_{\Gamma} d_t \theta_v^{n+1}\cdot  \boldsymbol \varphi_h^p dx  - a_m (\theta_u|_{\Gamma}^{n+1/2}, \boldsymbol \varphi_h^p|_{\Gamma})
\end{equation*}
\begin{equation}
+ c_{ep}(\theta^{n+1}_{p}, \boldsymbol \varphi^p_h)-c_{ep}(\psi^p_h, d_t \theta_u^{n+1}) + c_{fp}(\psi^p_h, \theta_f^{n+1}), \quad \forall (\boldsymbol \varphi_h^p,\boldsymbol \phi_h^p, \psi_h^p) \in V^p_h\times V^p_h \times Q^p_h. \label{error_biot}
\end{equation}

Let $\mathcal{E}_{\delta}^n$ be defined as
\begin{equation*}
 \mathcal{E}_{\delta}^n = \frac{\rho_f}{2} ||\delta_f^{n}||^2_{L^2(\Omega^f)}+ \frac{\rho_{m} r_m}{2} ||\delta_f^{n} \cdot \boldsymbol \tau||^2_{L^2(\Gamma)} +  \frac{\rho_{p}}{2} ||\delta_v^{n}||^2_{L^2(\Omega^p)} + \mu_p ||E(\delta_u^{n})||^2_{L^2(\Omega^p)} 
\end{equation*}
 \begin{equation*}
 + \frac{\lambda_p}{2}||\nabla \cdot \delta_u^{n}||^2_{L^2(\Omega^p)}+\frac{s_0}{2} ||\delta_p^{n}||^2_{L^2(\Omega^p)}
 + \displaystyle\frac{\rho_{m} r_m}{2} ||\delta_{v}^{n}||^2_{L^2(\Gamma)}+\mathcal{M}(\delta_u|^{n}_{\Gamma}),
\end{equation*}
where
\begin{equation*}
\mathcal{M}(\boldsymbol \xi^n) =  \displaystyle\frac{r_m}{2} \left[
  4 \mu_m \bigg|\bigg|\frac{\hat\xi_y}{R}\bigg|\bigg|^2_{L^2(0,L)} +
  4 \mu_m \bigg|\bigg|\frac{\partial \hat\xi_x}{\partial \hat{x}} \bigg|\bigg|^2_{L^2(0,L)}
  +\frac{4 \mu_m \lambda_m}{\lambda_m+2\mu_m} \bigg|\bigg|\frac{\partial \hat\xi_x}{\partial \hat{x}}+ \frac{\hat\xi_y}{R} \bigg|\bigg|^2_{L^2(0,L)}
 \right].
\end{equation*}
Note that $\mathcal{E}_{\delta}^n$ corresponds to the total discrete energy of the scheme that appears in Theorem~1 in terms of the truncation error.

\begin{theorem}
Consider the solution $(\boldsymbol v_h, p_{p,h}, \boldsymbol V_h, \boldsymbol U_h, \boldsymbol \xi_h, \boldsymbol \eta_h, p_{p,h})$ of~\eqref{S1discrete}-\eqref{S2discrete}. Assume that the time step condition~\eqref{CFL} holds, and that
the true solution $(\boldsymbol v, p_{p}, \boldsymbol V, \boldsymbol U, \boldsymbol \xi, \boldsymbol \eta, p_{p})$ satisfies~\eqref{true_sol}. Then, the following estimate holds:
\begin{equation*}
||\boldsymbol v - \boldsymbol v_h ||^2_{l^\infty(0,T;L^2(\Omega^f)}+\Delta t||d_t(\boldsymbol v - \boldsymbol v_h)||^2_{l^2(0,T;L^2(\Omega^f)} +\Delta t||d_t(\boldsymbol v - \boldsymbol v_h)||^2_{l^2(0,T;L^2(\Gamma)} +\frac{\gamma}{2} ||\boldsymbol v - \boldsymbol v_h ||^2_{l^2(0,T;H^1(\Omega^f))} 
\end{equation*}
\begin{equation*}
+||\boldsymbol V - \boldsymbol V_h ||^2_{l^\infty(0,T;L^2(\Omega^p))}+||p_p - p_{p,h} ||^2_{l^\infty(0,T;L^2(\Omega^p))}+||p_p - p_{p,h} ||^2_{l^2(0,T;H^1(\Omega^p))}+||\boldsymbol \xi - \boldsymbol \xi_h ||^2_{l^\infty(0,T;L^2(0,L))}
  \end{equation*}
  \begin{equation*}
 +||\boldsymbol U - \boldsymbol U_h ||^2_{l^\infty(0,T;H^1(\Omega^p))} +||\boldsymbol \eta - \boldsymbol \eta_h ||^2_{l^\infty(0,T;H^1(0,L))} 
 +\Delta t ||p_f - p_{f,h} ||^2_{l^2(0,T;L^2(\Omega^f))}
  \le C h^{2k} \mathcal{B}_1(\boldsymbol v,\boldsymbol U, \boldsymbol \eta, p_p) 
\end{equation*}
\begin{equation*}
   + C h^{2k+2} \mathcal{B}_2(\boldsymbol U, \boldsymbol V, \boldsymbol \eta, \boldsymbol \xi, p_p) 
   + C \Delta t^2 \mathcal{B}_3(\boldsymbol v, \boldsymbol U, \boldsymbol V, \boldsymbol \eta, \boldsymbol \xi, p_p)
       + C h^{2s+2} ||p_f||^2_{l^2(0,T;H^{s+1}(\Omega^f))},
\end{equation*}
where 
\begin{equation*}
 \mathcal{B}_1(\boldsymbol v,\boldsymbol U, \boldsymbol \eta, p_p)= ||\boldsymbol v||^2_{l^2(0,T;H^{k+1}(\Omega^f))}+||p_p||^2_{l^2(0,T;H^{k+1}(\Omega^p))}  +||\partial_t \boldsymbol v||^2_{l^2(0,T;H^{k+1}(\Omega^f))}+ ||\partial_t \boldsymbol U||^2_{l^2(0,T;H^{k+1}(\Omega^p))}
  \end{equation*}
  \begin{equation*}
+||\partial_t p_p||^2_{l^2(0,T;H^{k+1}(\Omega^p))}+||\partial_{tt} \boldsymbol \eta||^2_{l^2(0,T;H^{k+1}(0,L))}+||p_p||^2_{l^\infty(0,T;H^{k+1}(\Omega^p))}+||\boldsymbol U||^2_{l^\infty(0,T;H^{k+1}(\Omega^p))}
\end{equation*}
\begin{equation*}
+||\boldsymbol \eta||^2_{l^\infty(0,T;H^{k+1}(0,L))},
\end{equation*}
\begin{equation*}
 \mathcal{B}_2(\boldsymbol U, \boldsymbol V, \boldsymbol \eta, \boldsymbol \xi, p_p)= ||\partial_t p_p||^2_{l^2(0,T;H^{k+1}(\Omega^p))}+||\partial_{tt} \boldsymbol U||^2_{l^2(0,T;H^{k+1}(\Omega^p))}+||\partial_{tt} \boldsymbol \eta||^2_{l^2(0,T;H^{k+1}(0,L))}
\end{equation*}
\begin{equation*}
+||\partial_{tt} \boldsymbol \xi||^2_{l^2(0,T;H^{k+1}(0,L))}+||p_p||^2_{l^\infty(0,T;H^{k+1}(\Omega^p))}+||\boldsymbol V||^2_{l^\infty(0,T;H^{k+1}(\Omega^p)} +||\boldsymbol \xi||^2_{l^\infty(0,T;H^{k+1}(0,L)}
\end{equation*}
\begin{equation*}
+ ||\partial_{t} \boldsymbol U||^2_{l^\infty(0,T;H^{k+1}(\Omega^p))}+ ||\partial_{t} \boldsymbol \eta||^2_{l^\infty(0,T;H^{k+1}(0,L))}+ ||\partial_{t} \boldsymbol V||^2_{l^\infty(0,T;H^{k+1}(\Omega^p))}
+ ||\partial_{t} \boldsymbol \xi||^2_{l^\infty(0,T;H^{k+1}(0,L))},
\end{equation*}
\begin{equation*}
 \mathcal{B}_3(\boldsymbol v, \boldsymbol U, \boldsymbol V, \boldsymbol \eta, \boldsymbol \xi, p_p) = ||\partial_{tt} \boldsymbol v||^2_{L^2(0,T;L^2(\Omega^f))}+||\partial_{tt} \boldsymbol \xi ||^2_{L^2(0,T;L^2(0,L))}+||\partial_{t} p_p||^2_{L^2(0,T;H^1(\Omega^p))}
\end{equation*}
\begin{equation*}
+||\partial_{tt} p_p||^2_{L^2(0,T;L^2(\Omega^p))}+||\partial_{tt} \boldsymbol U||^2_{L^2(0,T;H^1(\Omega^p))}+||\partial_{ttt} \boldsymbol U||^2_{L^2(0,T;L^2(\Omega^p))}+||\partial_{tt} \boldsymbol V||^2_{L^2(0,T;L^2(\Omega^p))}
\end{equation*}
\begin{equation*}
+||\partial_{ttt} \boldsymbol \eta||^2_{L^2(0,T;L^2(0,L))}+||\partial_{ttt} \boldsymbol V||^2_{L^2(0,T;L^2(\Omega^p))}+||\partial_{ttt} \boldsymbol \xi||^2_{L^2(0,T;L^2(0,L))}+||\partial_{tt} \boldsymbol \eta||^2_{L^2(0,T;H^1(0,L))}
\end{equation*}
\begin{equation*}
+||\partial_{tt} \boldsymbol U||_{l^\infty(0,T;L^2(\Omega^p))} +||\partial_{tt} \boldsymbol \eta||_{l^\infty(0,T;L^2(0,L))}+ ||\partial_{tt} \boldsymbol V||_{l^\infty(0,T;L^2(\Omega^p))}+||\partial_{tt} \boldsymbol \xi||_{l^\infty(0,T;L^2(0,L))}
\end{equation*}
\begin{equation*}
+||\partial_t \boldsymbol V||_{l^\infty(0,T;L^2(\Omega^p))} +||\partial_t \boldsymbol \xi||_{l^\infty(0,T;L^2(0,L)))}+||\partial_t \boldsymbol U||_{l^\infty(0,T;H^1(\Omega^p))}+||\partial_t \boldsymbol \eta||_{l^\infty(0,T;H^1(0,L))}.
\end{equation*} \label{theorem_errors}
\end{theorem} 
\begin{proof}
With the purpose of presenting the proof in a clear manner, we will separate the proof into four main steps.  

\vskip 0.08in
\noindent \textbf{Step 1: Application of the stability result~\eqref{theorem:stability} to the truncation error equation.}
Choose $\boldsymbol \varphi^f_h = \delta_f^{n+1}, \boldsymbol \varphi^p_h = d_t \delta_u^{n+1}, \boldsymbol \phi^p_h = d_t \delta_v^{n+1},$ and $\psi^p_h =  \delta_p^{n+1}$ in equations~\eqref{error_fluid} and~\eqref{error_biot}. Then, multiplying the equations by $\Delta t,$ adding them together, summing over $0 \le n \le N-1$, and using the stability estimate~\eqref{theorem:stability} for the truncation error, we get
\begin{equation*}
 \mathcal{E}_{\delta}^{N} + \frac{\rho_{f} \Delta t^2}{2} \sum_{n=0}^{N-1}
||d_t \delta_f^{n+1}||^2_{L^2(\Omega^f)} + \frac{\rho_{m} r_m \Delta t^2}{2} \sum_{n=0}^{N-1}
||d_t \delta_f^{n+1} \cdot \boldsymbol \tau||^2_{L^2(\Gamma)} +\gamma \Delta t \sum_{n=0}^{N-1} ||D(\delta_f^{n+1})||^2_{L^2(\Omega^f)}
 \end{equation*}
\begin{equation*}
+ \Delta t \sum_{n=0}^{N-1} ||\sqrt{\kappa} \nabla \delta_p^{n+1}||^2_{L^2(\Omega^p)} \le \mathcal{E}_{\delta}^0 + \Delta t \sum_{n=0}^{N-1} \bigg(\mathcal{R}^{n+1}_f(\delta_f^{n+1})+ \mathcal{R}_s^{n+1}(d_t \delta_u^{n+1})+\mathcal{R}_v^{n+1}(d_t \delta_v^{n+1})+ \mathcal{R}_p^{n+1}(\delta_p^{n+1})\bigg)
\end{equation*}
\begin{equation*}
+\Delta t \sum_{n=0}^{N-1} b_f(p_{f}^{n+1}, \delta_f^{n+1})-\rho_f \Delta t \sum_{n=0}^{N-1}  \int_{\Omega^f}d_t \theta_f^{n+1} \cdot \delta_f^{n+1} d\boldsymbol x -\Delta t \sum_{n=0}^{N-1} c_{fp}(\theta_p^n, \delta_f^{n+1})
\end{equation*}
 \begin{equation*}
 -\rho_{m}r_m \Delta t \sum_{n=0}^{N-1} \int_{\Gamma} (d_t \theta_f^{n+1} \cdot \boldsymbol \tau )(\delta_f^{n+1} \cdot \boldsymbol \tau) d x +\rho_{p} \Delta t \sum_{n=0}^{N-1} \int_{\Omega^p} d_t \theta_u^{n+1} \cdot d_t \delta_v^{n+1} d \boldsymbol x-\rho_{p} \Delta t \sum_{n=0}^{N-1} \int_{\Omega^p} \theta_v^{n+1/2} \cdot d_t \delta_v^{n+1} d \boldsymbol x
\end{equation*}
 \begin{equation*}
+\rho_{p} \Delta t \sum_{n=0}^{N-1} \int_{\Omega^p} d_t \theta_v^{n+1} \cdot d_t \delta_u^{n+1} d \boldsymbol x-\Delta t \sum_{n=0}^{N-1} a_e(\theta_u^{n+1/2}, d_t \delta_u^{n+1}) + \Delta t \sum_{n=0}^{N-1} b_{ep}(\theta^{n+1}_{p}, d_t \delta_u^{n+1})
\end{equation*}
\begin{equation*}
 -  \Delta t \sum_{n=0}^{N-1} b_{ep}(\delta_p^{n+1},d_t \theta_u^{n+1})-\rho_{m} r_m \sum_{n=0}^{N-1} \Delta t \int_{\Gamma} (\theta_v^{n+1/2}-d_t\theta_u^{n+1})  \cdot d_t \delta_v^{n+1} d x 
\end{equation*}
\begin{equation*}
-\rho_{m} r_m \Delta t \sum_{n=0}^{N-1}\int_{\Gamma} d_t \theta_v^{n+1}\cdot d_t \delta_u^{n+1} dx -\Delta t \sum_{n=0}^{N-1} a_p(\theta^{n+1}_{p}, \delta_p^{n+1})- \Delta t \sum_{n=0}^{N-1} a_m (\theta_u|_{\Gamma}^{n+1/2}, d_t \delta_u|_{\Gamma}^{n+1})
\end{equation*}
\begin{equation}
+ \Delta t \sum_{n=0}^{N-1}c_{ep}(\theta^{n+1}_{p}, d_t \delta_u^{n+1})  -\Delta t \sum_{n=0}^{N-1}c_{ep}(\delta_p^{n+1}, d_t \theta_u^{n+1}) + \Delta t \sum_{n=0}^{N-1} c_{fp}(\delta_p^{n+1}, \theta_f^{n+1}). \label{error_energy}
\end{equation}
The right hand side of~\eqref{error_energy} consists of consistency error terms $\mathcal{R}_f^{n+1}, \mathcal{R}_s^{n+1}, \mathcal{R}_v^{n+1}$ and $\mathcal{R}_p^{n+1}$, and mixed truncation and interpolation error terms. We will proceed by bounding the consistency error terms.

\vskip 0.08in
\noindent \textbf{Step 2: The consistency error estimate}. In this step we will use Lemma~\ref{consistency_error} of the Appendix to bound the consistency error terms. 
Referring to the formulation of Lemma~4 and in particular to the terms collected into the expression  $\mathcal{A}(\delta_f, \delta_p, \delta_v, \delta_u),$ we observe that the Gronwall Lemma is required to obtain an upper bound.

\vskip 0.08in
\noindent \textbf{Step 3: The mixed truncation and interpolation error terms estimate}. In this step we estimate the remaining terms of~\eqref{error_energy}, which are terms that contain both truncation and interpolation error. 

Using Cauchy-Schwartz~\eqref{CS}, Young's~\eqref{youngs}, Poincar\'e - Friedrichs~\eqref{PF}, and Korn's~\eqref{korn} inequalities, we have the following:
\begin{equation*}
 -\rho_f \Delta t \sum_{n=0}^{N-1}\int_{\Omega^f} d_t\theta_f^{n+1} \cdot \delta_f^{n+1} d\boldsymbol x \leq C \Delta t \sum_{n=0}^{N-1} ||\nabla d_t  \theta_f^{n+1}||^2_{L^2(\Omega^f)}+\frac{\gamma \Delta t}{8}\sum_{n=0}^{N-1}||D(\delta_f^{n+1})||^2_{L^2(\Omega^f)}. \label{est_b}
\end{equation*}
Furthermore, using Young's~\eqref{youngs}, Korn's~\eqref{korn}, and trace~\eqref{trace} inequalities we can estimate
\begin{equation*}
 - \Delta t \sum_{n=0}^{N-1} \bigg( a_p(\theta^{n+1}_{p}, \delta_p^{n+1})-  c_{fp}(\delta_p^{n+1}, \theta_f^{n+1}) + c_{fp}(\theta_p^n, \delta_f^{n+1}) \bigg)  \leq C \Delta t \sum_{n=0}^{N-1} ||D(\theta_f^{n+1})||^2_{L^2(\Omega^f)}
\end{equation*}
\begin{equation*}
+ \frac{\gamma \Delta t}{8} \sum_{n=0}^{N-1} ||D(\delta_f^{n+1})||^2_{L^2(\Omega^f)}+C \Delta t \sum_{n=0}^{N-1} ||\nabla \theta_p^{n+1}||^2_{L^2(\Omega^p)}+\frac{\Delta t}{8} \sum_{n=0}^{N-1} ||\sqrt{\kappa} \nabla \delta_p^{n+1}||^2_{L^2(\Omega^p)}.
\end{equation*}
In a similar way, we bound
\begin{equation*}
  -\rho_{m}r_m \Delta t \sum_{n=0}^{N-1}\int_{\Gamma} (d_t \theta_f^{n+1}\cdot \boldsymbol \tau) (\delta_f^{n+1} \cdot \boldsymbol \tau) d x \leq C \Delta t \sum_{n=0}^{N-1} ||\nabla d_t \theta_f^{n+1}||^2_{L^2(\Omega^f)}+\frac{\gamma \Delta t}{8} \sum_{n=0}^{N-1} ||D(\delta_f^{n+1})||^2_{L^2(\Omega^f)}.
\end{equation*}
The next two terms can be controlled as follows
\begin{equation*}
  - \Delta t \sum_{n=0}^{N-1} \bigg( b_{ep}(\delta_p^{n+1},d_t \theta_u^{n+1}) + c_{ep}(\delta_p^{n+1}, d_t \theta_u^{n+1}) \bigg)  \leq \frac{\Delta t}{8}\sum_{n=0}^{N-1}||\sqrt{\kappa} \nabla \delta_p^{n+1}||^2_{L^2(\Omega^p)}+C \Delta t \sum_{n=0}^{N-1}||\nabla d_t \theta_u^{n+1} ||^2_{L^2(\Omega^p)}.
\end{equation*}
To bound the pressure term, note that $b_f(p_f^{n+1}, \delta_f^{n+1}) = b_f(p_f^{n+1} - \Pi^f_{h} p_f^{n+1}, \delta_f^{n+1})$ since $\Pi^f_h p_f^{n+1} \in Q^f_h$ and $\delta_f^{n+1} \in X^f_h$. Hence, using notation $\theta_{fp}^{n+1} = p_f^{n+1} - \Pi^f_{h} p_f^{n+1}$, we have
\begin{equation*}
  \Delta t \sum_{n=0}^{N-1} b_f(\theta_{fp}^{n+1}, \delta_f^{n+1}) \leq C \Delta t \sum_{n=0}^{N-1} ||\theta_{fp}^{n+1}||^2_{L^2(\Omega^f)}+\frac{\gamma \Delta t}{8} \sum_{n=0}^{N-1} ||D(\delta_f^{n+1})||^2_{L^2(\Omega^f)}.
\end{equation*}
To estimate the remaining terms, we use discrete integration by parts in time.
Using equation~\eqref{int_by_parts}, we have
\begin{equation*}
 \Delta t \sum_{n=0}^{N-1} b_{ep}(\theta^{n+1}_{p}, d_t \delta_u^{n+1}) = \alpha \int_{\Omega_p} \theta_p^N \nabla \cdot \delta_u^{N} d \boldsymbol x - \alpha \Delta t \sum_{n=0}^{N-1} \int_{\Omega^p} d_t \theta_p^{n+1} \nabla \cdot \delta_u^n d \boldsymbol x \leq C_{\epsilon} ||\theta_p^N|| ^2_{L^2(\Omega^p)}
\end{equation*}
\begin{equation*}
+ \epsilon ||D(\delta_u^N)|| ^2_{L^2(\Omega^p)}+ C \Delta t \sum_{n=0}^{N-1} ||d_t \theta_p^{n+1}||^2_{L^2(\Omega^p)}+ C \Delta t \sum_{n=0}^{N-1} ||D(\delta_u^n)||^2_{L^2(\Omega^p)},
\end{equation*}
and
\begin{equation*}
 \Delta t \sum_{n=0}^{N-1} c_{ep}(\theta^{n+1}_{p}, d_t \delta_u^{n+1}) = \alpha \int_{\Gamma} \theta_p^N \delta_u^{N}\cdot \boldsymbol n dx - \alpha  \Delta t \sum_{n=0}^{N-1} \int_{\Gamma} d_t\theta_p^{n+1} \delta_u^n \cdot \boldsymbol n dx  \leq C_{\epsilon} ||\nabla \theta_p^N|| ^2_{L^2(\Omega^p)}
\end{equation*}
\begin{equation*}
+ \epsilon ||D(\delta_u^N)|| ^2_{L^2(\Omega^p)}+ C \Delta t \sum_{n=0}^{N-1} ||\nabla d_t \theta_p^{n+1}||^2_{L^2(\Omega^p)}+ C \Delta t \sum_{n=0}^{N-1} ||D(\delta_u^n)||^2_{L^2(\Omega^p)}.
\end{equation*}
Also,
\begin{equation*}
 \rho_{p} \Delta t \sum_{n=0}^{N-1} \int_{\Omega^p} d_t \theta_u^{n+1} \cdot d_t \delta_v^{n+1} d \boldsymbol x = \rho_{p}\int_{\Omega^p} d_t \theta_u^{N}\cdot \delta_v^N d\boldsymbol x-  \rho_{p} \Delta t \sum_{n=1}^{N-1}\int_{\Omega^p}  d_{tt}\theta_u^{n+1} \cdot \delta_v^{n} d \boldsymbol x
\end{equation*}
\begin{equation*}
\leq C_{\epsilon} ||d_t \theta_u^{N}||^2_{L^2(\Omega^p)} + \epsilon  ||\delta_v^N||^2_{L^2(\Omega^p)}+C \Delta t\sum_{n=1}^{N-1} ||d_{tt} \theta_u^{n+1}||^2_{L^2(\Omega^p)}+C \Delta t\sum_{n=1}^{N-1} ||\delta_v^n||^2_{L^2(\Omega^p)},
\end{equation*}
\begin{equation*}
- \rho_{p} \Delta t \sum_{n=0}^{N-1} \int_{\Omega^p} \theta_v^{n+1/2} \cdot d_t \delta_v^{n+1} d \boldsymbol x = -\rho_{p}\int_{\Omega^p} \theta_v^{N-1/2} \cdot \delta_u^N d\boldsymbol x+  \rho_{p} \Delta t \sum_{n=1}^{N-1}\int_{\Omega^p}  \frac{\theta_v^{n+1}-\theta_v^{n-1}}{2 \Delta t} \cdot \delta_v^{n} d \boldsymbol x
\end{equation*}
\begin{equation*}
 \le C_{\epsilon} (||\theta_v^N||^2_{L^2(\Omega^p)}+  ||\theta_v^{N-1}||^2_{L^2(\Omega^p)} )+ \epsilon  ||D(\delta_u^N)||^2_{L^2(\Omega^p)} +C \Delta t \sum_{n=1}^{N-1} \bigg|\bigg|\frac{\theta_v^{n+1}-\theta_v^{n-1}}{2 \Delta t}\bigg| \bigg|^2_{L^2(\Omega^p)}
\end{equation*}
\begin{equation*}
+ C\Delta t \sum_{n=1}^{N-1} ||D(\delta_u^n)||^2_{L^2(\Omega^p)},
\end{equation*}
and
\begin{equation*}
 \rho_{p} \Delta t \sum_{n=0}^{N-1} \int_{\Omega^p} d_t \theta_v^{n+1} \cdot d_t \delta_u^{n+1} d \boldsymbol x = \rho_{p}\int_{\Omega^p} d_t \theta_v^{N} \cdot \delta_u^N d\boldsymbol x-  \rho_{p} \Delta t \sum_{n=1}^{N-1}\int_{\Omega^p}  d_{tt}\theta_v^{n+1} \cdot \delta_u^{n} d \boldsymbol x
\end{equation*}
\begin{equation*}
\leq C_{\epsilon} ||d_t \theta_v^{N}||^2_{L^2(\Omega^p)} + \epsilon  ||D(\delta_u^N)||^2_{L^2(\Omega^p)}+C \Delta t\sum_{n=1}^{N-1} ||d_{tt} \theta_v^{n+1}||^2_{L^2(\Omega^p)}+C \Delta t\sum_{n=1}^{N-1} ||D(\delta_u^n)||^2_{L^2(\Omega^p)}.
\end{equation*}
In a similar way,
\begin{equation*}
\rho_{m}r_m \Delta t \sum_{n=0}^{N-1} \int_{\Gamma} d_t \theta_u^{n+1}  \cdot d_t \delta_v^{n+1} d x=\rho_{m}r_m\int_{\Gamma} d_t \theta_u^{N} \cdot  \delta_v^N d x -  \rho_{m}r_m \Delta t \sum_{n=1}^{N-1}\int_{\Gamma} d_{tt} \theta_u^{n+1} \cdot \delta_v^{n} d x
\end{equation*}
\begin{equation*}
 \leq C_{\epsilon} ||d_t \theta_u^{N}||^2_{L^2(\Gamma)} + \epsilon  ||\delta_v^N||^2_{L^2(\Gamma)}
 +C \Delta t \sum_{n=1}^{N-1} ||d_{tt}\theta_u^{n+1}||^2_{L^2(\Gamma)}+C \Delta t \sum_{n=1}^{N-1} ||\delta_v^n||^2_{L^2(\Gamma)},
\end{equation*}
and 
\begin{equation*}
-\rho_{m}r_m \Delta t \sum_{n=0}^{N-1} \int_{\Gamma} d_t \theta_v^{n+1} \cdot d_t \delta_u^{n+1} d x=-\rho_{m}r_m\int_{\Gamma} d_t \theta_v^{N}\cdot  \delta_u^N d x+  \rho_{m}r_m \Delta t \sum_{n=1}^{N-1}\int_{\Gamma}  d_{tt}\theta_v^{n+1} \cdot \delta_u^{n} d x 
\end{equation*}
\begin{equation*}
\leq C_{\epsilon} ||d_t \theta_v^{N}||^2_{L^2(\Gamma)} + \epsilon  ||D(\delta_u^N)||^2_{L^2(\Omega^p)}
 +C \Delta t \sum_{n=1}^{N-1} ||d_{tt} \theta_v^{n+1} ||^2_{L^2(\Gamma)}+C \Delta t \sum_{n=1}^{N-1} ||D(\delta_u^n)||^2_{L^2(\Omega^p)}.
\end{equation*}
Furthermore, 
\begin{equation*}
-\rho_{m} r_m \Delta t \sum_{n=0}^{N-1} \int_{\Gamma} \theta_v^{n+1/2} \cdot d_t \delta_v^{n+1} dx =-\rho_{m}r_m\int_{\Gamma} \theta_v^{N-1/2}\cdot \delta_u^N d x+  \rho_{m}r_m \Delta t \sum_{n=1}^{N-1}\int_{\Gamma}  \frac{\theta_v^{n+1}-\theta_v^{n-1}}{2 \Delta t} \cdot \delta_u^{n} d x
\end{equation*}
\begin{equation*}
 \leq C_{\epsilon} (||\theta_v^N||^2_{L^2(\Gamma)}+  ||\theta_v^{N-1}||^2_{L^2(\Gamma)} )+ \epsilon  ||D(\delta_u^N)||^2_{L^2(\Omega^p)} +C \Delta t \sum_{n=1}^{N-1} \bigg|\bigg|\frac{\theta_v^{n+1}-\theta_v^{n-1}}{2 \Delta t}\bigg| \bigg|^2_{L^2(\Gamma)}
+ C\Delta t \sum_{n=1}^{N-1} ||D(\delta_u^n)||^2_{L^2(\Omega^p)},
\end{equation*}
and
\begin{equation*}
-\Delta t \sum_{n=0}^{N-1} a_e(\theta_u^{n+1/2}, d_t \delta_u^{n+1}) = -a_e(\theta_u^{N-1/2}, \delta_u^N)+ \Delta t \sum_{n=1}^{N-1}a_e(\frac{\theta_u^{n+1}-\theta_u^{n-1}}{2 \Delta t}, \delta_u^{n})  \leq C_{\epsilon} (||D(\theta_u^N)||^2_{L^2(\Omega^p)}
\end{equation*}
\begin{equation*}
+||D(\theta_u^{N-1})||^2_{L^2(\Omega^p)})+\epsilon||D(\delta_u^N)||^2_{L^2(\Omega^p)}+C \Delta t \sum_{n=1}^{N-1} \bigg|\bigg|D\bigg(\frac{\theta_u^{n+1}-\theta_u^{n-1}}{2 \Delta t}\bigg)\bigg| \bigg|^2_{L^2(\Gamma)}+ C\Delta t \sum_{n=1}^{N-1} ||D(\delta_u^n)||^2_{L^2(\Omega^p)}.
\end{equation*}
Lastly,
\begin{equation*}
 - \Delta t \sum_{n=0}^{N-1} a_m (\theta_u|_{\Gamma}^{n+1/2}, d_t \delta_u|_{\Gamma}^{n+1})= -a_m(\theta_u|_{\Gamma}^{N-1/2}, \delta_u|_{\Gamma}^N)+\Delta t \sum_{n=1}^{N-1} a_m (\frac{\theta_u|_{\Gamma}^{n+1}-\theta_u|_{\Gamma}^{n-1}}{2 \Delta t}, \delta_u|_{\Gamma}^{n})
\end{equation*}
\begin{equation*}
 \leq C_{\epsilon} (\mathcal{M}(\theta_u|_{\Gamma}^N)+\mathcal{M}(\theta_u|_{\Gamma}^{N-1}))+\epsilon \mathcal{M}(\delta_u|_{\Gamma}^N)+C \Delta t \sum_{n=1}^{N-1} \mathcal{M}(\frac{\theta_u|_{\Gamma}^{n+1}-\theta_u|_{\Gamma}^{n-1}}{2 \Delta t})+C \Delta t \sum_{n=1}^{N-1} \mathcal{M}(\delta_u|^n_{\Gamma}).
\end{equation*}

Using the estimates from Steps 1-3, we have
\begin{equation*}
 \mathcal{E}_{\delta}^{N} + \frac{\rho_{f} \Delta t^2}{2} \sum_{n=0}^{N-1}
||d_t \delta_f^{n+1}||^2_{L^2(\Omega^f)} + \frac{\rho_{m} r_m \Delta t^2}{2} \sum_{n=0}^{N-1}
||d_t \delta_f^{n+1} \cdot \boldsymbol \tau||^2_{L^2(\Gamma)} +\frac{\gamma}{2} \Delta t \sum_{n=0}^{N-1} ||D(\delta_f^{n+1})||^2_{L^2(\Omega^f)}
\end{equation*}
\begin{equation*}
+ \frac{\Delta t}{2} \sum_{n=0}^{N-1} ||\sqrt{\kappa} \nabla \delta_p^{n+1}||^2_{L^2(\Omega^p)}
\le \epsilon \bigg(||D(\delta_u^N)|| ^2_{L^2(\Omega^p)}+||\nabla \cdot \delta_u^N||^2_{L^2(\Omega^p)}  +||\delta_v^N||^2_{L^2(\Omega^p)}+||\delta_v^N||^2_{L^2(\Gamma)}+\mathcal{M}(\delta_u|_{\Gamma}^N) \bigg)
\end{equation*}
\begin{equation*}
+ C \Delta t \sum_{n=0}^{N-1}\bigg(||D(\delta_u^n)||^2_{L^2(\Omega^p)}+||\nabla \cdot \delta_u^n||^2_{L^2(\Omega^p)}+ ||\delta_v^n||^2_{L^2(\Omega^p)}+ ||\delta_v^n||^2_{L^2(\Gamma)}+ \mathcal{M}(\delta_u|^n_{\Gamma}) \bigg)
\end{equation*}
 \begin{equation*}
+C \Delta t \sum_{n=0}^{N-1} \bigg(||D(\theta_f^{n+1})||^2_{L^2(\Omega^f)}+ ||\theta_{fp}^{n+1}||^2_{L^2(\Omega^f)}+||\nabla \theta_p^{n+1}||^2_{L^2(\Omega^p)}+ ||\nabla d_t \theta_f^{n+1}||^2_{L^2(\Omega^f)} +||\nabla d_t\theta_u^{n+1}||^2_{L^2(\Omega^p)}
\end{equation*}
\begin{equation*}
+||d_t\theta_p^{n+1}||^2_{L^2(\Omega^p)}+||\nabla d_t\theta_p^{n+1}||^2_{L^2(\Omega^p)}+ ||d_{tt}\theta_u^{n+1}||^2_{L^2(\Omega^p)}+||d_{tt}\theta_v^{n+1}||^2_{L^2(\Omega^p)}+ ||d_{tt} \theta_u^{n+1}||^2_{L^2(\Gamma)}+||d_{tt}\theta_v^{n+1}||^2_{L^2(\Gamma)}
\end{equation*}
\begin{equation*}
+ \bigg|\bigg|\frac{\theta_v^{n+1}-\theta_v^{n-1}}{2 \Delta t}\bigg| \bigg|^2_{L^2(\Omega^p)}+\bigg|\bigg|\frac{\theta_v^{n+1}-\theta_v^{n-1}}{2 \Delta t}\bigg| \bigg|^2_{L^2(\Gamma)} +\bigg|\bigg|D \bigg(\frac{\theta_u^{n+1}-\theta_u^{n-1}}{2 \Delta t} \bigg)\bigg| \bigg|^2_{L^2(\Omega^p)} + \mathcal{M}(\frac{\theta_u|_{\Gamma}^{n+1}-\theta_u|_{\Gamma}^{n-1}}{2 \Delta t})\bigg)
\end{equation*}
\begin{equation*}
+C\max_{0 \le n \le N}\bigg(||\theta_p^n|| ^2_{L^2(\Omega^p)}+||\nabla \theta_p^n|| ^2_{L^2(\Omega^p)}+||\theta_v^n||^2_{L^2(\Omega^p)}+ ||\theta_v^n||^2_{L^2(\Gamma)}+ ||D(\theta_u^{n})||^2_{L^2(\Omega^p)}
\end{equation*}
\begin{equation*}
+ \mathcal{M}(\theta_u|_{\Gamma}^n) + ||d_t \theta_u^{n}||^2_{L^2(\Omega^p)}+ ||d_t\theta_u^{n}||^2_{L^2(\Gamma)}+ ||d_t \theta_v^{n}||^2_{L^2(\Omega^p)} + ||d_t \theta_v^{n}||^2_{L^2(\Gamma)}\bigg)
\end{equation*}
\begin{equation*}
 + C \Delta t^2 \bigg( ||\partial_{tt} \boldsymbol v||^2_{L^2(0,T;L^2(\Omega^f))}+||\partial_{tt} \boldsymbol \xi ||^2_{L^2(0,T;L^2(0,L))}+||\partial_{t} p_p||^2_{L^2(0,T;H^1(\Omega^p))}+||\partial_{tt} p_p||^2_{L^2(0,T;L^2(\Omega^p))}
\end{equation*}
\begin{equation*}
+||\partial_{tt} \boldsymbol U||^2_{L^2(0,T;H^1(\Omega^p))}+||\partial_{ttt} \boldsymbol U||^2_{L^2(0,T;L^2(\Omega^p))}+||\partial_{tt} \boldsymbol V||^2_{L^2(0,T;L^2(\Omega^p))}+||\partial_{ttt} \boldsymbol \eta||^2_{L^2(0,T;L^2(0,L))}
\end{equation*}
\begin{equation*}
+||\partial_{ttt} \boldsymbol V||^2_{L^2(0,T;L^2(\Omega^p))}+||\partial_{ttt} \boldsymbol \xi||^2_{L^2(0,T;L^2(0,L))}+||\partial_{tt} \boldsymbol \eta||^2_{L^2(0,T;H^1(0,L))}+||\partial_{tt} \boldsymbol U||_{l^\infty(0,T;L^2(\Omega^p))}
\end{equation*}
\begin{equation*}
 +||\partial_{tt} \boldsymbol \eta||_{l^\infty(0,T;L^2(0,L))}+ ||\partial_{tt} \boldsymbol V||_{l^\infty(0,T;L^2(\Omega^p))}+||\partial_{tt} \boldsymbol \xi||_{l^\infty(0,T;L^2(0,L))}+||\partial_t \boldsymbol V||_{l^\infty(0,T;L^2(\Omega^p))}
\end{equation*}
\begin{equation*}
+ ||\partial_t \boldsymbol \xi||_{l^\infty(0,T;L^2(0,L))}+||\partial_t \boldsymbol U||_{l^\infty(0,T;H^1(\Omega^p))}+||\partial_t \boldsymbol \eta||_{l^\infty(0,T;H^1(0,L))}\bigg),
\end{equation*}
where $\epsilon>0$ can be taken small enough.
Finally, using using Lemma~\ref{lemma_interpolation}, approximation properties~\eqref{ap1}-\eqref{ap2}, triangle inequality, and the discrete Gronwall inequality, we prove the desired estimate, except for the pressure error in the fluid domain.

\vskip 0.08in
\noindent \textbf{Step 4: analysis of the fluid pressure error.}
To control this part of the error we proceed as for the stability estimate.
More precisely, we start from \eqref{error_fluid} and we rearrange it as follows
\begin{align}
\nonumber
&b_f(\delta_{fp}^{n+1}, \boldsymbol \varphi_h^f) = \rho_f \int_{\Omega^f}d_t \delta_f^{n+1} \cdot \boldsymbol \varphi_h^f d\boldsymbol x
+ \rho_{m}r_m \int_{\Gamma} (d_t \delta_f^{n+1} \cdot \boldsymbol \tau)(\boldsymbol \varphi_h^f \cdot \boldsymbol \tau) d x
+ \rho_f \int_{\Omega^f} d_t\theta_f^{n+1} \cdot \boldsymbol \varphi_h^f d\boldsymbol x
\\
\label{error_fp_1}
+& \rho_{m}r_m \int_{\Gamma} (d_t \theta_f^{n+1} \cdot \boldsymbol \tau )(\boldsymbol \varphi_h^f \cdot \boldsymbol \tau) d x,
+ a_f(e_f^{n+1}, \boldsymbol \varphi_h^f)
+ c_{fp}(e_p^n, \boldsymbol \varphi_h^f) 
+ b_f(\theta_{fp}^{n+1}, \boldsymbol \varphi_h^f)
- \mathcal{R}^{n+1}_f(\boldsymbol \varphi_h^f).
\end{align}
For simplicity of notation, let us group the terms on the right hand side of the previous equation,
\begin{multline*}
\mathcal{T}(\boldsymbol \varphi_h^f) :=
\rho_f \int_{\Omega^f} d_t e_f^{n+1} \cdot \boldsymbol \varphi_h^f d\boldsymbol x
+ \rho_{m}r_m \int_{\Gamma} (d_t e_f^{n+1} \cdot \boldsymbol \tau) (\boldsymbol \varphi_h^f \cdot \boldsymbol \tau) d x
\\
+ a_f(e_f^{n+1}, \boldsymbol \varphi_h^f)
+ c_{fp}(e_p^n, \boldsymbol \varphi_h^f) 
+ b_f(\theta_{fp}^{n+1}, \boldsymbol \varphi_h^f)
\end{multline*}
Owing to the \emph{inf-sup} condition~\eqref{infsup} between spaces $V_h^f$ and $Q_h^f$ there exists a positive constant $\beta_f$
independent of the mesh characteristic size such that,
\begin{equation}\label{error_fp_2}
\beta_f \|\delta_{fp}^{n+1}\|_{L^2(\Omega^f)} 
\leq \sup\limits_{\boldsymbol \varphi_h^f \in V_h^f}
\frac{\mathcal{T}(\boldsymbol \varphi_h^f)- \mathcal{R}^{n+1}_f(\boldsymbol \varphi_h^f)}{\|\boldsymbol \varphi_h^f\|_{H^1(\Omega^f)}}.
\end{equation}
Moving along the lines of the stability estimate, the following upper bounds for the right hand side of \eqref{error_fp_2} hold true, with a generic constant $C$ which depends on the trace~\eqref{trace}, Korn~\eqref{korn} and Poincar\'e-Friedrichs~\eqref{PF} inequalities, as well as on the parameters of the problem,
\begin{equation*}
\sup\limits_{\boldsymbol \varphi_h^f \in V_h^f}
\frac{\mathcal{T}(\boldsymbol \varphi_h^f)}{\|\boldsymbol \varphi_h^f\|_{H^1(\Omega^f)}}
\leq C \Big(\|e_f^{n+1}\|_{H^1(\Omega^f)} + \|e_p^{n}\|_{H^1(\Omega^p)} + \|\theta_{fp}^{n+1}\|_{L^2(\Omega^f)}
+ \|d_t e_f^{n+1}\|_{L^2(\Omega^f)} + \|d_t e_f^{n+1}\|_{L^2(\Gamma)}
\Big).
\end{equation*}
Using the bounds that will be detailed in Lemma~\ref{consistency_error} of the Appendix, we get
\begin{equation*}
\sup\limits_{\boldsymbol \varphi_h^f \in V_h^f}
\frac{\mathcal{R}^{n+1}_f(\boldsymbol \varphi_h^f)}{\|\boldsymbol \varphi_h^f\|_{H^1(\Omega^f)}}
\leq C\Big( \|(d_t-\partial_t) \boldsymbol v^{n+1}\|_{L^2(\Omega^f)}
+\|(d_t-\partial_t) \boldsymbol v^{n+1}\cdot \boldsymbol \tau\|_{L^2(\Gamma)}
+\|\nabla(p_p^{n+1}-p_p^n)\|_{L^2(\Omega^p)} \Big).
\end{equation*}
Finally, we replace the previous estimates into \eqref{error_fp_2}, square all terms, sum up with respect to $n$ and multiply by $\Delta t^2$. There exists a positive constant $c$ small enough such that
\begin{multline*}
c \Delta t^2 \sum_{n=0}^{N-1} \|\delta_{fp}^{n+1}\|_{L^2(\Omega^f)}^2 
\leq \Delta t^2 \sum_{n=0}^{N-1} \Big(\|d_t e_f^{n+1}\|_{L^2(\Omega^f)}^2 + \|d_t e_f^{n+1}\|_{L^2(\Gamma)}^2
+ \|e_f^{n+1}\|_{H^1(\Omega^f)}^2 + \|e_p^{n}\|_{H^1(\Omega^p)}^2
\\
+ \|\theta_{fp}^{n+1}\|_{L^2(\Omega^f)}^2 + \|d_t \boldsymbol v^{n+1}-\partial_t \boldsymbol v^{n+1}\|_{L^2(\Omega^f)}^2
+\|d_t \boldsymbol v^{n+1}\cdot \boldsymbol \tau-\partial_t \boldsymbol v^{n+1} \cdot \boldsymbol \tau \|_{L^2(\Gamma)}^2
+\|\nabla(p_p^{n+1}-p_p^n)\|_{L^2(\Omega^p)}^2 \Big).
\end{multline*}
To conclude, combining the triangle inequality with the approximation properties of the discrete pressure space and bounding the right hand side with the available error estimates, we obtain
\begin{multline*}
\Delta t ||p_f - p_{f,h} ||^2_{l^2(0,T;L^2(\Omega^f)} 
\leq C \Big( C h^{2k} \mathcal{B}_1(\boldsymbol v,\boldsymbol U, \boldsymbol \eta, p_p)
\\ + Ch^{2k+2} \mathcal{B}_2(\boldsymbol v, \boldsymbol U, \boldsymbol V, \boldsymbol \eta, \boldsymbol \xi, p_p) 
+ C \Delta t^2 \mathcal{B}_3(\boldsymbol v, \boldsymbol U, \boldsymbol V, \boldsymbol \eta, \boldsymbol \xi, p_p)+C h^{2s+2} ||p_f||^2_{l^2(0,T;H^{s+1}(\Omega^f))}\Big).
\end{multline*}

\end{proof}


\section{Numerical results}\label{sec:NumR}
The focus of this section is on verification of the results presented in this work and exploration of poroelastic effects in the model.
We test the scheme on a classical benchmark problem used for convergence studies of fluid-structure iteration problems~\cite{formaggia2001coupling,bukavc2012fluid,badia2008fluid,burman2009stabilization,barker2010scalable}.
In Example 1, we present the convergence of our scheme in space and time. Furthermore, we validate the necessity of the stability condition~\eqref{CFL}.

In Example~2 we analyze the role of poroelastic effects in blood flow. In particular, we compare our results to the ones obtained using a purely elastic model in Example 2. 
We distinguish a high permeability and a high storativity case, and present a comparison between the two cases and the purely elastic model.

\subsection{Example 1.}
We consider the classical test problem  used in several works~\cite{formaggia2001coupling,bukavc2012fluid,badia2008fluid,burman2009stabilization}
as a benchmark problem for testing the results of fluid-structure interaction algorithms for blood flow.
In our case, the flow is driven by the time-dependent pressure data:
\begin{equation}\label{pressure}
 p_{in}(t) = \left\{\begin{array}{l@{\ } l} 
\frac{p_{max}}{2} (1 - \cos(\frac{2 \pi t}{T_{max}})) & \; \textrm{if} \; t \le T_{max} \\
0   & \; \textrm{if} \; t>T_{max},
 \end{array} \right.
\end{equation}
where $p_{max} = 1.3334 \; dyne/cm^2$ and $T_{max}=0.003 \; s$. For the elastic skeleton, we consider the following equation of linear elasticity:
$$ \rho_{p} \frac{D^2 \boldsymbol U}{D t^2} + \beta \boldsymbol U- \nabla \cdot \boldsymbol \sigma^p = 0. $$ The additional term $\beta \boldsymbol U$ 
comes from the axially symmetric formulation, accounting for the recoil due to the circumferential strain. Namely, it acts like a spring term, keeping the top and
bottom structure displacements connected in 2D, see, e.g.,~\cite{badia2008splitting,barker2010scalable,ma1992numerical}. The values of the parameters used in this example are given in Table~\ref{T1}.
{{
\begin{center}
\begin{table}[ht!]
{\small{
\begin{tabular}{|l l l l |}
\hline
\textbf{Parameters} & \textbf{Values} & \textbf{Parameters} & \textbf{Values}  \\
\hline
\hline
\textbf{Radius} $R$ (cm)  & $0.5$  & \textbf{Length} $L$ (cm) & $6$  \\
\hline
\textbf{Membrane thickness} $r_m $(cm) & $0.02$  & \textbf{Poroelastic wall thickness} $r_p$ (cm) & $0.1$  \\
\hline
\textbf{Membrane density} $\rho_{m} $(g/cm$^3$) & $1.1$  & \textbf{Poroelastic wall density} $\rho_{p} $(g/cm$^3$) & $1.1$  \\
\hline
\textbf{Fluid density} $\rho_f$ (g/cm$^3$)& $1$ &\textbf{Dyn. viscosity} $\mu$ (g/cm s) & $0.035$    \\
\hline
\textbf{Lam\'e coeff.} $\mu_m $(dyne/cm$^2$) & $1.07 \times 10^6$  & \textbf{Lam\'e coeff.} $\lambda_m $(dyne/cm$^2$) & $4.28 \times 10^6$ \\
\hline
\textbf{Lam\'e coeff.} $\mu_p $(dyne/cm$^2$) & $1.07 \times 10^6$  & \textbf{Lam\'e coeff.} $\lambda_p $(dyne/cm$^2$) & $4.28 \times 10^6$ \\
\hline
\textbf{Hydraulic conductivity} $\kappa $(cm$^3$ s/g) & $5 \times 10^{-9}$  & \textbf{Mass storativity coeff.} $s_0 $(cm$^2$/dyne) & $5 \times 10^{-6}$   \\
\hline
\textbf{Biot-Willis constant} $\alpha$  & 1 & \textbf{Spring coeff.} $\beta $(dyne/cm$^4$) & $5 \times 10^7$   \\
\hline
\end{tabular}
}}
\caption{Geometry, fluid and structure parameters that are used in Example 1.}
\label{T1}
\end{table}
\end{center}
}}
Parameters given in Table~\ref{T1} are within the range of physiological values for blood flow. The problem was solved over the time interval $[0,0.006]$ s.

In order to verify the convergence estimates from Theorem~\ref{theorem_errors}, let the errors between the computed and the reference solution be defined as  $e_f = \boldsymbol v-\boldsymbol v_{ref}, e_{fp} = p_f-p_{f,ref}, e_v = \boldsymbol V-\boldsymbol V_{ref}, e_p = p_p-p_{p,ref},$ and $e_u = \boldsymbol U-\boldsymbol U_{ref}$.
We start by computing the rates of convergence in time. In order to do so, fix $\Delta x=0.016$ and define the reference solution to be the one obtained  with $\triangle t = 5 \times 10^{-7}$. 
Table~\ref{T2} shows the error between the reference solution and solutions obtained with $\triangle t = 10^{-6}, 5\times 10^{-6}, 10^{-5},$ and $3 \times 10^{-5}$ for the fluid velocity $\boldsymbol v$, fluid pressure $p_f$, pressure in the pores $p_p$, displacement $\boldsymbol U$ and its velocity $\boldsymbol V$, respectively.
 \begin{table}[ht!]
\begin{center}
\begin{tabular}{| l | c  c | c  c | c  c | c c |}
\hline
$ \triangle t $ & $||e_f||_{l^{\infty}(L^2)}$ & rate & $||e_f||_{l^2(H^1)}$& rate& $||e_{fp}||_{l^2(L^2)}$  & rate & $||e_v||_{l^{\infty}(L^2)}$  & rate \\
\hline
\hline
$3 \times 10^{-5}$ & $7.6 \textrm{e}-1$  & - & $8.8 \textrm{e}-2$  & - & $8.9 \textrm{e}-1$ & -& $3.0 \textrm{e}-2$ & - \\ 
\hline  
$10^{-5}$ & $2.7 \textrm{e}-1$  & 0.93 & $3.1 \textrm{e}-2$  & 0.95 & $3.4\textrm{e}-1$ & 0.88& $1.0 \textrm{e}-2$& 0.84 \\ 
\hline   
$5 \times 10^{-6}$ & $1.3 \textrm{e}-1$  & 1.04& $1.5 \textrm{e}-2$  &1.03&$ 1.6 \textrm{e}-1$& 1.05 & $6.0 \textrm{e}-3$ & 0.97\\ 
\hline   
$10^{-6}$ & $1.5 \textrm{e}-2$  & 1.36 & $1.8 \textrm{e}-3$  & 1.32 & $1.8 \textrm{e}-2$& 1.36 & $6.9 \textrm{e}-4$ & 1.36\\ 
\hline 
\end{tabular}
\vskip 0.1in
\begin{tabular}{| l | c  c | c  c  | c  c |}
\hline
$ \triangle t $ & $||e_p||_{l^{\infty}(L^2)}$  & rate & $||e_p||_{l^2(H^1)}$ & rate & $ ||e_u||_{l^{\infty}(H^1)} $ & rate \\
\hline
\hline
$3 \times 10^{-5}$  &$6.3 \textrm{e}-1$ & - &$4.1 \textrm{e}-1$  & - & $7.7 \textrm{e}-5$  & - \\ 
\hline  
$10^{-5}$ & $2.2 \textrm{e}-1$  & 0.95 & $1.5 \textrm{e}-1$  & 0.92 & $3.0 \textrm{e}-5$  & 0.85 \\ 
\hline   
$5 \times 10^{-6}$  & $1.1 \textrm{e}-1$ & 1.04  & $7.4\textrm{e}-2$ & 1.02 & $1.5 \textrm{e}-5$ & 0.97\\ 
\hline   
$10^{-6}$ & $1.2 \textrm{e}-2$  & 1.35 & $8.5 \textrm{e}-3$  & 1.34 & $1.8 \textrm{e}-6$  & 1.32  \\ 
\hline 
\end{tabular}
\end{center}
\caption{Convergence in time. }
\label{T2}
\end{table}

To study the convergence in space, we take $\Delta t = 5 \times 10^{-6}$ and define the reference solution to be the one obtained with $\Delta x = r_p/14 = 0.007$. Table~\ref{T3} shows errors between the 
reference solution and the solutions obtain using $\Delta x = 0.01, 0.0125, 0.0167$ and $0.025.$
 \begin{table}[ht!]
\begin{center}
\begin{tabular}{| l | c  c | c c | c  c | c  c |}
\hline
$ \triangle x $ & $||e_f||_{l^{\infty}(L^2)}$ & rate & $||e_f||_{l^2(H^1)}$ & rate & $||e_{fp}||_{l^2(L^2)}$  & rate & $||e_v||_{l^{\infty}(L^2)}$  & rate \\
\hline
\hline
$r_p/4$ & $2.4 \textrm{e}-1$  & - & $6.8 \textrm{e}-1$  & - & $3.1 \textrm{e}-1$ & - & $2.5 \textrm{e}-2$ & - \\ 
\hline  
$r_p/5$ & $1.9\textrm{e}-1$  & 1.01 & $6.2 \textrm{e}-1$  & 0.43 & $2.6 \textrm{e}-1 $& 0.83 & $1.8 \textrm{e}-2$ & 1.46 \\ 
\hline   
$r_p/6$ & $1.5\textrm{e}-1$  & 1.55& $5.3 \textrm{e}-1$  &0.87 & $2.1\textrm{e}-1 $& 1.12 & $1.4 \textrm{e}-2$ & 1.37 \\ 
\hline   
$r_p/7$ & $1.2\textrm{e}-1$  & 1.37 & $4.5 \textrm{e}-1$  & 1.06 & $1.7 \textrm{e}-1$ & 1.32 & $1.0 \textrm{e}-2$ & 1.6\\ 
\hline 
\end{tabular}
\vskip 0.1in
\begin{tabular}{| l | c  c | c c | c  c  |}
\hline
$ \triangle x $ & $||e_p||_{l^{\infty}(L^2)}$ & rate & $||e_p||_{l^2(H^1)}$  & rate & $ ||e_u||_{l^{\infty}(H^1)} $ & rate \\
\hline
\hline
$r_p/4$ & $1.2 \textrm{e}+0$  & - & $4.0 \textrm{e}-1$  & - & $3.1 \textrm{e}-1$  & - \\ 
\hline  
$r_p/5$ & $0.9 \textrm{e}-1$  & 0.92 & $3.6 \textrm{e}-1$  & 0.47 & $2.6 \textrm{e}-1$  & 0.84\\ 
\hline   
$r_p/6$ & $0.8 \textrm{e}-1$  & 1.10 & $3.2 \textrm{e}-1$  & 0.73 & $2.1 \textrm{e}-1$  & 1.08 \\ 
\hline   
$r_p/7$ & $0.6 \textrm{e}-1 $  & 1.34 & $2.7 \textrm{e}-1$  & 1.01 & $1.8 \textrm{e}-1$  & 1.08 \\ 
\hline 
\end{tabular}
\end{center}
\caption{Convergence in space.}
\label{T3}
\end{table}

\begin{figure}[ht!]
 \centering{
 \includegraphics[scale=1]{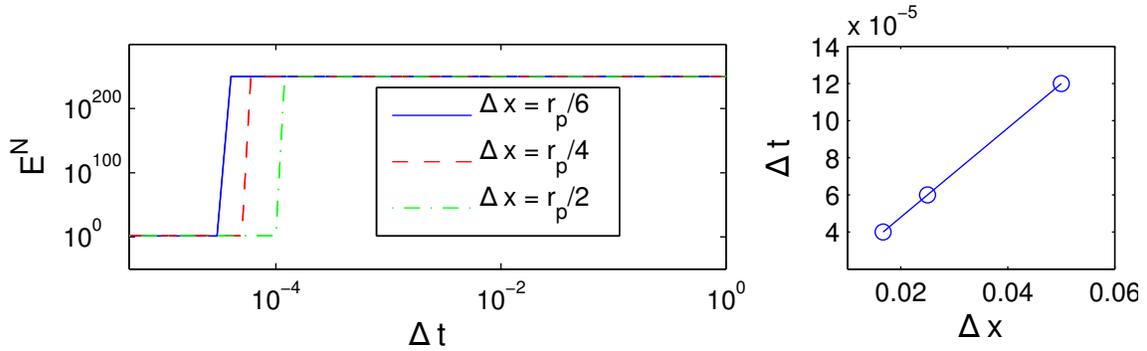}
 }
 \caption{Verification of the time-step condition~\eqref{CFL}. Left: Relation between the total energy of the system and the time step. Right: Relation between $\Delta x$ and the critical $\Delta t$.}
\label{F3}
 \end{figure}
To verify the necessity of the time-step condition~\eqref{CFL}, we compute the total energy $E^N$ of the system using different time steps. The time at which $E^N$ is computed
is either the time when $E^N$ becomes greater than $10^{250}$, or the final time $t^N=6$ ms. Figure~\ref{F3} shows the relation of the energy of the system and the time step (left), and the relation between $\Delta x$ and the critical $\Delta t$ (right).
Indeed, we observe a linear relation between $\Delta x$ and the critical value of $\Delta t$, with the proportionality constant $2.4 \textrm{e}-3$. This is less restrictive than the prediction~\eqref{CFL} from the theory, where the proportionality constant for the parameters in Table~\ref{T1} can be estimated as $3.5\textrm{e}-7$, indicating that the scheme enjoys better stability properties then prescribed by~\eqref{CFL}.

\subsection{Example 2.}
In this example we compare our numerical results to the ones obtained using a purely elastic model for the outer layer of the arterial wall. 
More precisely, while the fluid and the membrane are modeled as before, we assume there is no fluid contained within the wall, and we model the thick wall using 2D linear elasticity
\begin{equation*}
\rho_{p} \frac{D^2 \boldsymbol U}{D t^2}+ \beta \boldsymbol U - \nabla \cdot \boldsymbol \sigma^E= 0 \; \; \: \quad     \textrm{in} \; \Omega^p(t) \; \textrm{for} \; t \in (0,T).
\end{equation*}
The problem is solved using an operator-splitting approach performed in the same spirit as in this manuscript. 

\if 1=0
For the purpose of understanding the poroelastic effects to the structure displacement, we will distinguish two cases: the high storativity case, and the high permeability case. 
Namely, if $s_0>>\kappa$ (the high storativity case), from equation~\eqref{pressurefo} we can formally write $p_p \simeq -\frac{\alpha}{s_0}\nabla \cdot \boldsymbol U$. Substituting $p_p$ in equation~\eqref{structurefo},
the structure equation becomes
\begin{equation*}
\rho_{p} \frac{D^2 \boldsymbol U}{D t^2}+ \beta \boldsymbol U - 2 \mu_p \nabla \cdot \boldsymbol D(\boldsymbol U)- \lambda_p \nabla (\nabla \cdot \boldsymbol U)  -\frac{\alpha^2}{s_0} \nabla (\nabla \cdot \boldsymbol U) = 0 \; \; \: \quad     \textrm{in} \; \Omega^p(t) \; \textrm{for} \; t \in (0,T).
\end{equation*}
In this case, we recover a 2D elasticity equation, with a modified Lam\'e's first coefficient $\bar{\lambda}_p = \lambda_p+\frac{\alpha^2}{s_0}$. The effect of the pressure in this case is in changing the material properties of the structure, and altering the wave propagation speed. 

In the case when $\kappa > > s_0$ (the high permeability case), again from equation~\eqref{pressurefo} we can formally write $\nabla p_p \simeq \frac{\alpha}{\kappa}\frac{\partial \boldsymbol U}{\partial t}$. 
Substituting $\nabla p_p$ in equation~\eqref{structurefo}, we obtain the following equation
\begin{equation*}
\rho_{p} \frac{D^2 \boldsymbol U}{D t^2}+ \beta \boldsymbol U - 2 \mu_p \nabla \cdot \boldsymbol D(\boldsymbol U)- \lambda_p \nabla (\nabla \cdot \boldsymbol U)  +\frac{\alpha^2}{\kappa} \frac{\partial \boldsymbol U}{\partial t} = 0 \; \; \: \quad     \textrm{in} \; \Omega^p(t) \; \textrm{for} \; t \in (0,T).
\end{equation*}
In this case, the contribution from the pressure is a damping term, with the damping coefficient $\frac{\alpha^2}{\kappa}$.
\fi

For the purpose of understanding the poroelastic effects to the structure displacement, we distinguish two cases: the high storativity case $s_0>>\kappa$, and the high permeability case $\kappa>>s_0$. 
We give a comparison of the results obtained using the elastic model for the outer wall, and poroelastic model using two different values for $s_0$, and two different values for $\kappa$. The first test case for the poroelastic wall will correspond to the parameters $s_0$ and $\kappa$ from Example 1 ($s_0=5\times 10^{-6}, \kappa=5 \times 10^{-9}$),
the second test case will correspond to the increased value of $s_0$ ($s_0=2\times 10^{-5}, \kappa=5 \times 10^{-9}$), and the third example to the increased value of $\kappa$ ($s_0=5\times 10^{-6}, \kappa=10^{-4}$).
Figure~\ref{F5} shows the pressure pulse (colormap) and velocity streamlines obtained with the two models. The velocity magnitude is shown in Figure~\ref{F6}. 
  \begin{figure}[ht!]
 \centering{
 \includegraphics[scale=0.95]{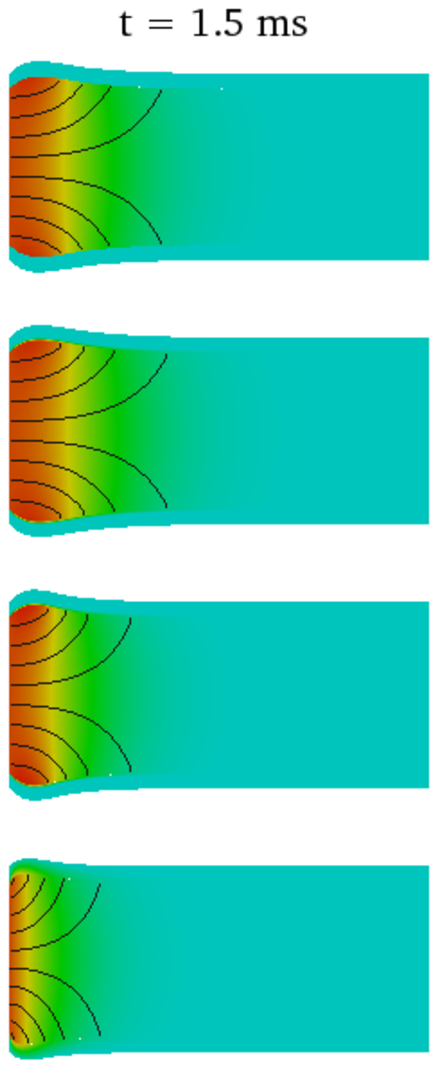}
  \includegraphics[scale=0.95]{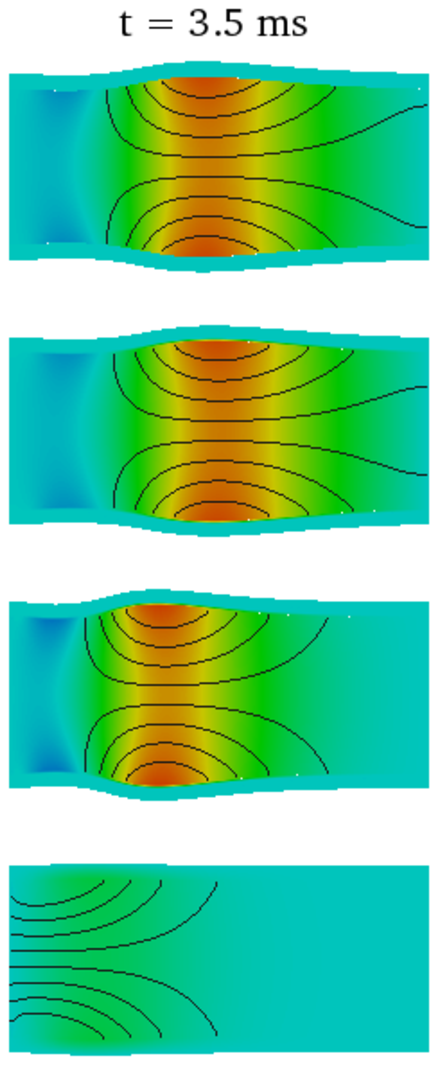}
   \includegraphics[scale=0.95]{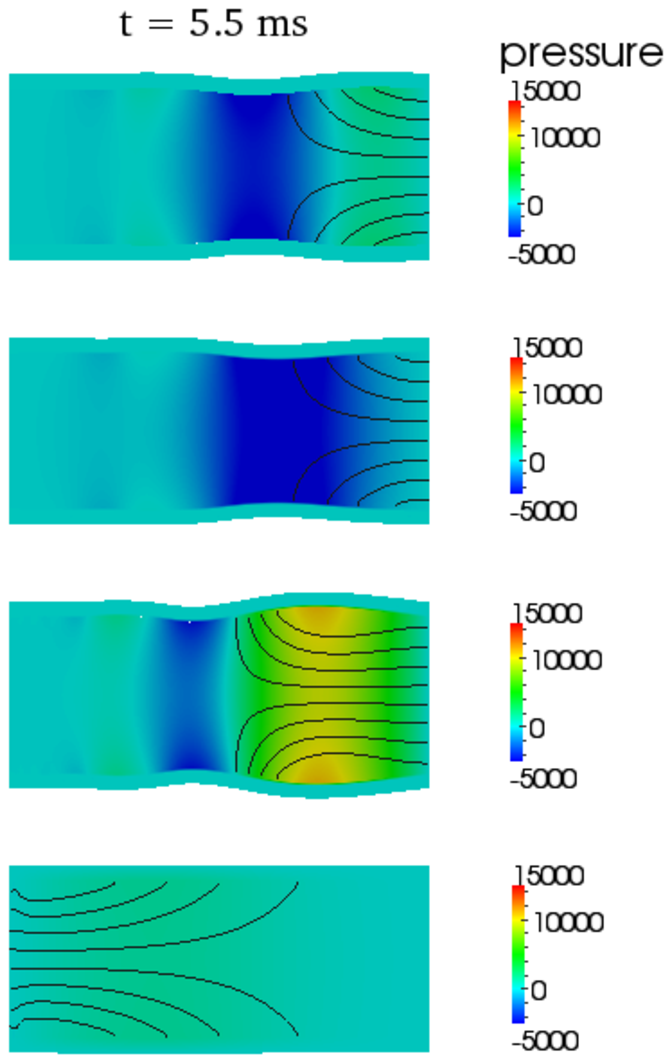}
 }
 \caption{Pressure in the lumen, velocity streamlines, and pressure in the wall at times $t=1.5$ ms, $t=3.5$ ms and $t=5.5$ ms. The outer layer of the arterial wall is model using a elastic model (top), poroelastic model with $s_0=5\times 10^{-6}, \kappa=5 \times 10^{-9}$ (middle top), poroelastic model with $s_0=2\times 10^{-5}, \kappa=5 \times 10^{-9}$ (middle bottom), and  poroelastic model with $s_0=5\times 10^{-6}, \kappa=10^{-4}$ (bottom).}
\label{F5}
 \end{figure}
   \begin{figure}[ht!]
 \centering{
 \includegraphics[scale=0.95]{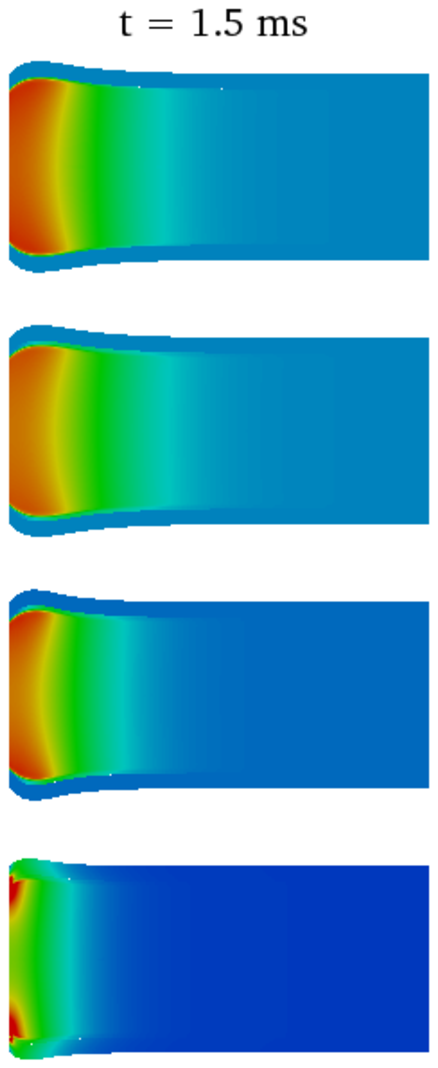}
  \includegraphics[scale=0.95]{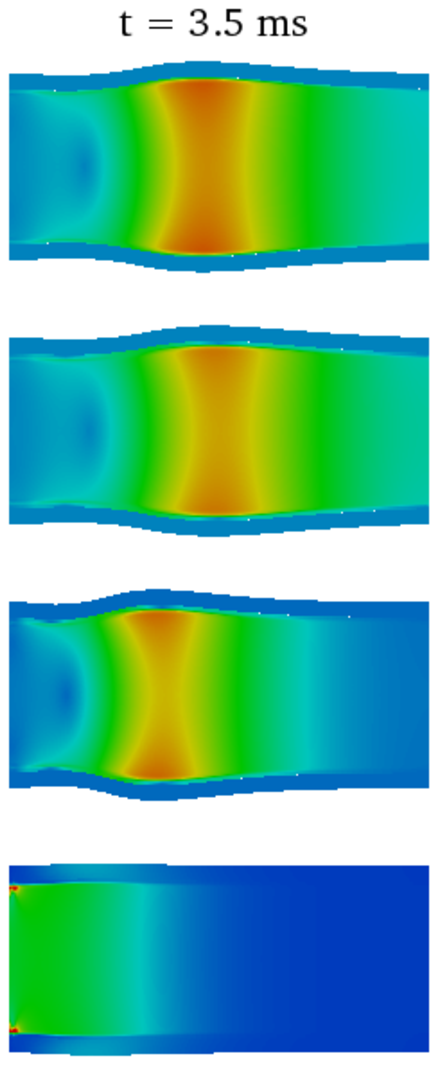}
   \includegraphics[scale=0.95]{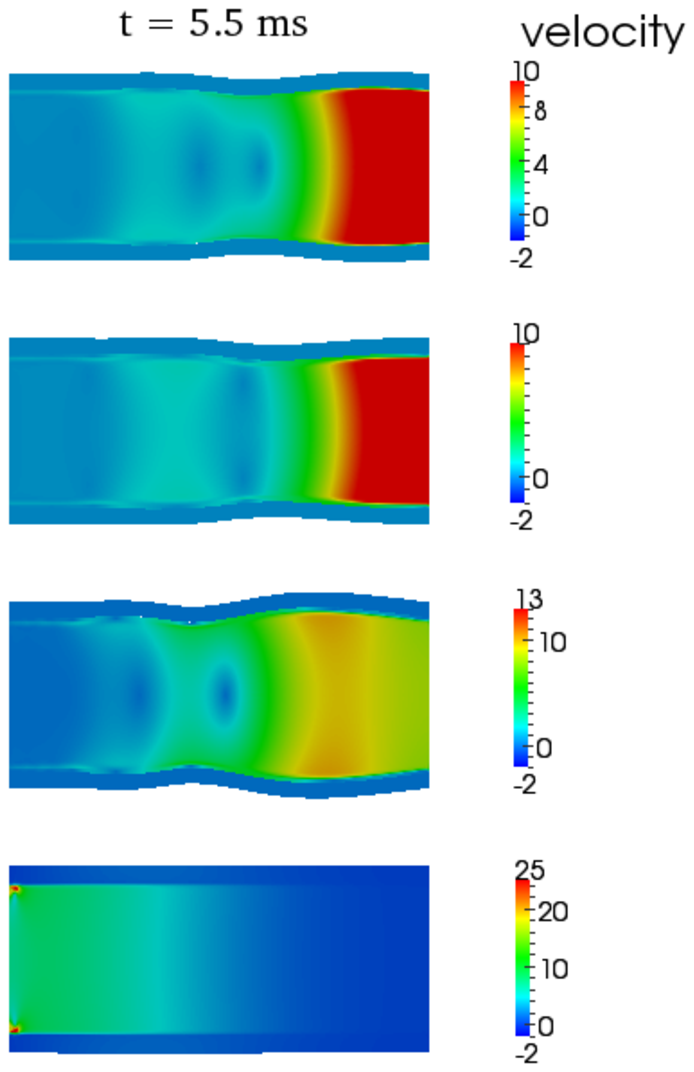}
 }
 \caption{Velocity magnitude at times $t=1.5$ ms, $t=3.5$ ms and $t=5.5$ ms. The outer layer of the arterial wall is model using a elastic model (top), poroelastic model with $s_0=5\times 10^{-6}, \kappa=5 \times 10^{-9}$ (middle top), poroelastic model with $s_0=2\times 10^{-5}, \kappa=5 \times 10^{-9}$ (middle bottom), and  poroelastic model with $s_0=5\times 10^{-6}, \kappa=10^{-4}$ (bottom).}
\label{F6}
 \end{figure}
 
To quantify the differences, we compute average quantities on each vertical line $S^{r}_i$ of the computational mesh $\Omega^r$, corresponding to the position $x_i=i \cdot \Delta x$, where $\Delta x =0.016$ and $r \in \{f,m\}$. The quantities of interest are
membrane displacement, the mean pressure, and the flowrate in the lumen:
\begin{equation*}
 \bar{p}_f (x_i) = \frac{1}{S^f_i} \int_{S^f_i} p_{f,h} ds, \quad  \bar{p}_p (x_i) = \frac{1}{S^p_i} \int_{S^p_i} p_{p,h} ds, \quad Q(x_i) = \int_{S^f_i} \boldsymbol v_h \cdot \boldsymbol e_x ds.
\end{equation*}
Figure~\ref{F4} shows a comparison between the flowrate in the lumen,  membrane displacement, and the mean pressure in the lumen and in the wall, obtained using a poroelastic model and an elastic model.
In the high permeability regime the structure displacement is the smallest, while in this case we observe the largest mean pressure in the wall. In the high storativity regime, we observe a delay in the pressure wave propagation speed, and qualitatively different displacement. 
\begin{figure}[ht!]
 \centering{
 \includegraphics[scale=0.75]{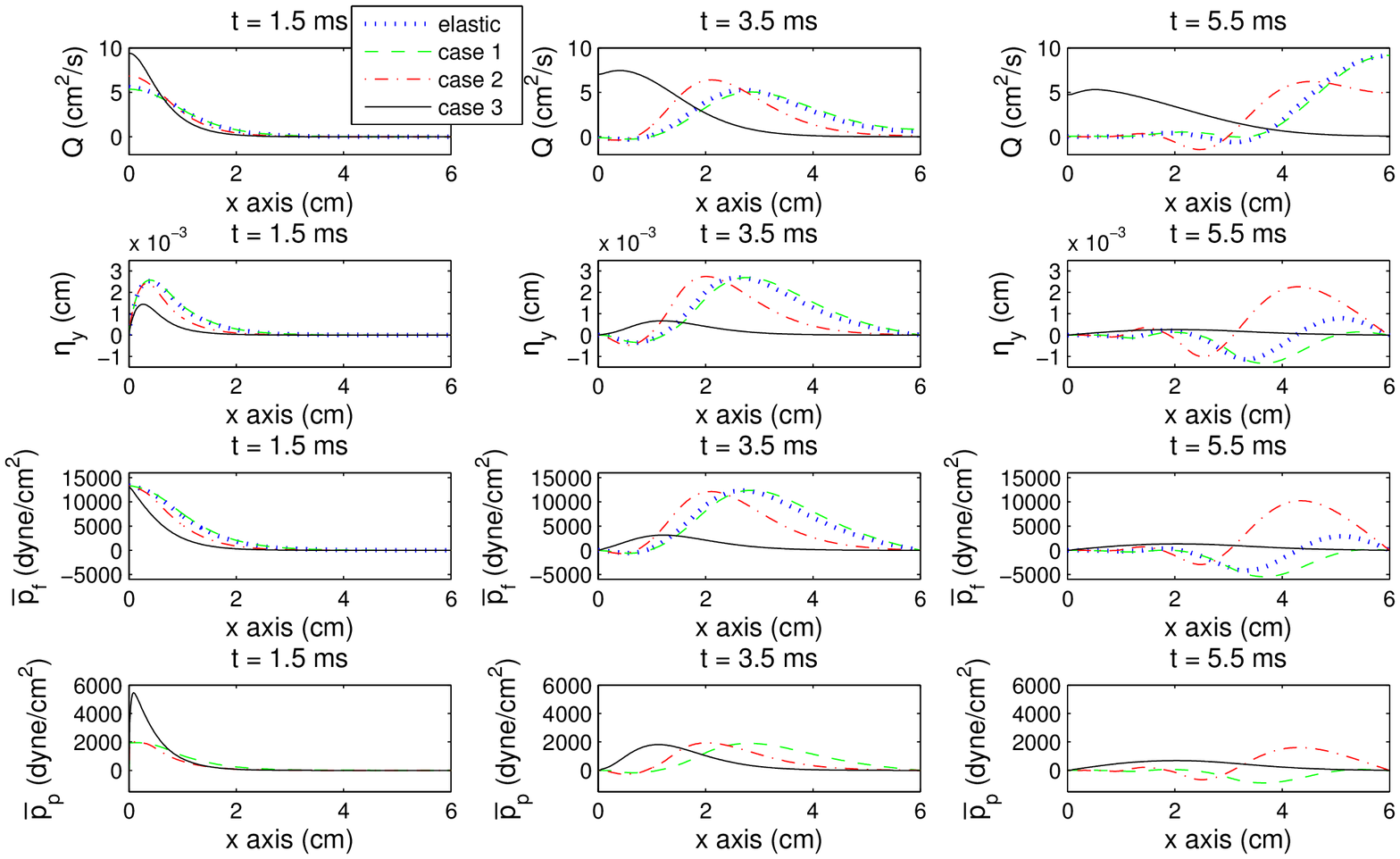} 
 }
 \caption{From top to bottom: Flowrate, radial displacement of the membrane, mean pressure in the lumen, and mean pressure in the pores. Results were obtained using the elastic model (dotted line), and poroelastic model with the following parameters: case 1 ($s_0=5 \times 10^{-6}, \kappa=5\times 10^{-9}$; dashed line), 
 case 2 ($s_0=2 \times 10^{-5}, \kappa=5\times 10^{-9}$; dash dot line), and case 3 ($s_0=5 \times 10^{-6}, \kappa=1\times 10^{-4}$; solid line).}
\label{F4}
 \end{figure}

 Finally, Figure~\ref{F8} shows the axial and radial velocity profiles at the center of the tube $x = 3$ cm for the three poroelastic test cases, computed at 6 different times. Again, we observe a different in the amplitude, and in the wave propagation speed between the three cases. In particular, there is
 a significant difference in the high permeability regime, where the amplitude of axial velocity is much smaller than in the other two cases, and the vertical velocity is always positive due to the high dissipation in the structure. 
 \begin{figure}[ht!]
 \centering{
 \includegraphics[scale=1]{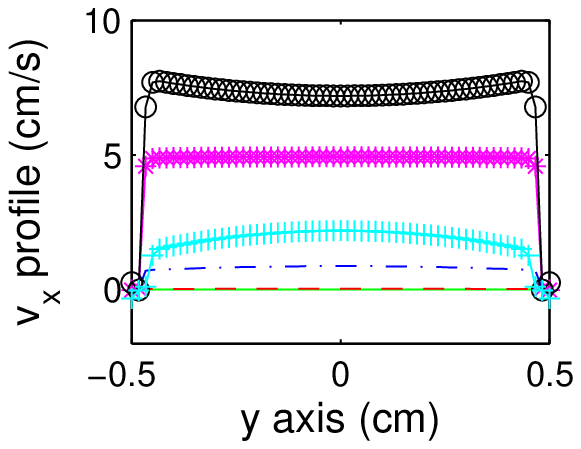} 
  \includegraphics[scale=1]{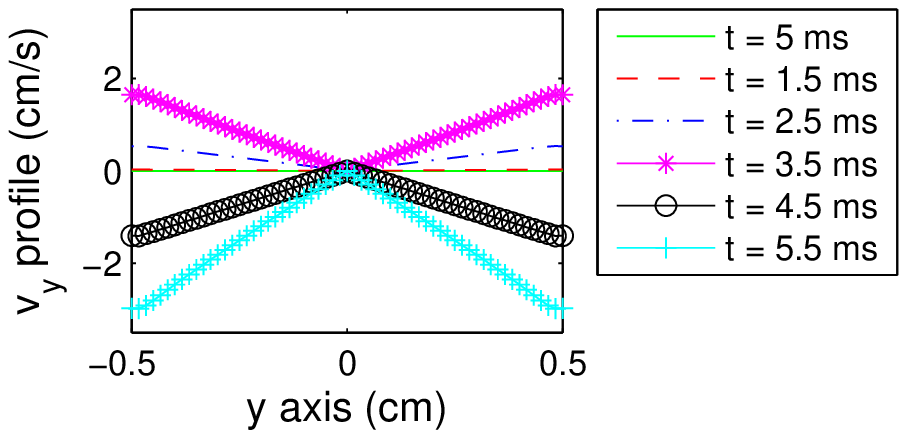} 
 \includegraphics[scale=1]{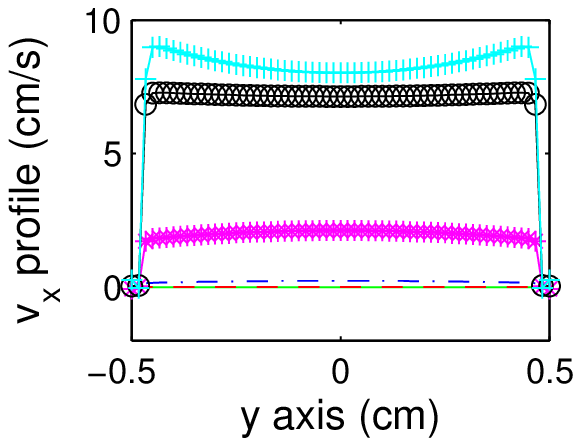} 
  \includegraphics[scale=1]{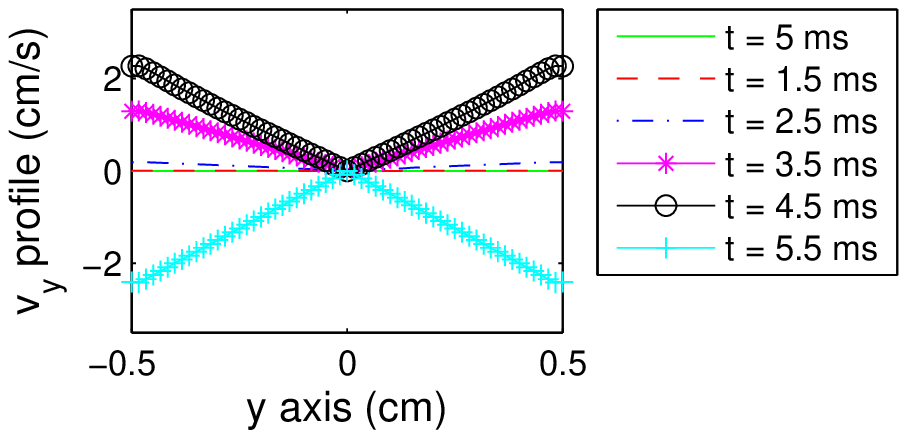} 
 \includegraphics[scale=1]{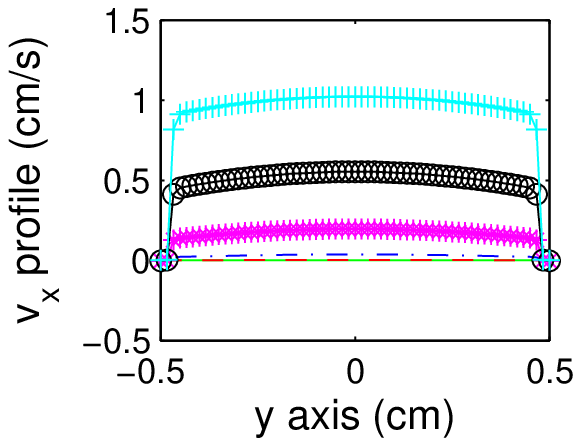} 
  \includegraphics[scale=1]{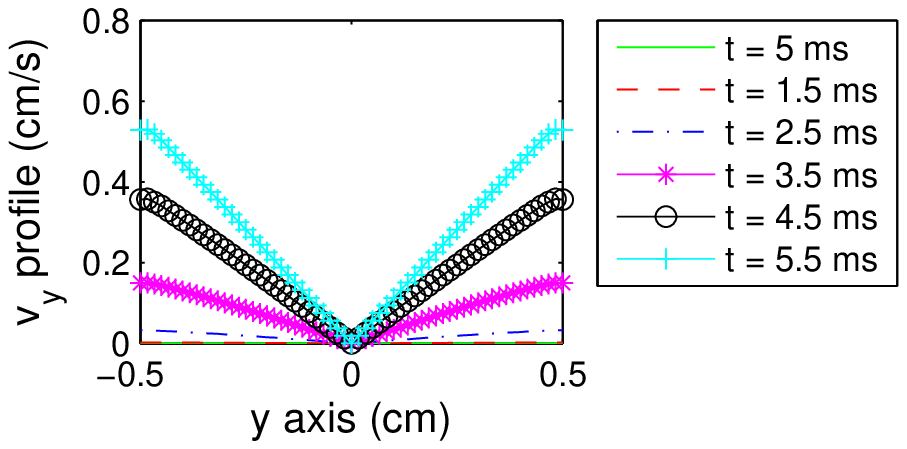} 
 }
 \caption{Fluid velocity profiles at six different times obtained using a poroelastic model with coefficients $s_0=5 \times 10^{-6}, \kappa=5 \times 10^{-9}$ (top), $s_0=2 \times 10^{-6}, \kappa=5 \times 10^{-9}$ (middle), and $s_0=5 \times 10^{-6}, \kappa= 10^{-4}$ (bottom).}
\label{F8}
 \end{figure}
 
\section{Conclusion}
The focus of this paper is on modeling and implementation of a fluid-poroelastic structure interaction problem. In particular, we study the interaction between the fluid a multilayered wall, where the wall consists of
a thin membrane, and a thick poroelastic medium. We proposed an explicit numerical algorithm based on the Lie operator splitting scheme. 
An alternative discrete problem formulation based on Nitsche's method for the enforcement of the interface conditions is under study. This new method can accommodate a mixed formulation for the Darcy's equations.

 We prove the conditional stability of the algorithm, and derive error estimates. 
Stability and convergence results are validated by the numerical simulations. The drawback of the scheme is that it requires pressure formulation for the Darcy equation. Concerning the application of the scheme to blood flow in arteries,  we test numerically the porous effects in the wall, comparing results obtained with different coefficients to the ones obtained using a purely elastic model.
We observe different behavior depending on the storativity or permeability dominant regime.

\clearpage


\appendix


\section{Auxiliary results}

We collect in this section some auxiliary results and proofs that complement the stability and convergence analysis of the proposed scheme. They are either a consequence of the standard theory of the finite element method or they follow from basic results of approximation theory. For the reader's convenience we report them separately from the main body of the manuscript.

For any real numbers $a,b$ the following algebraic identities are satisfied:
\begin{equation}
 (a-b)a = \frac{1}{2} a^2 - \frac{1}{2} b^2+ \frac{1}{2}(a-b)^2,
\end{equation}
\begin{equation}
 \Delta t \sum_{n=0}^{N-1} a^{n+1} d_t b^{n+1}  = a^Nb^N - \Delta t \sum_{n=0}^{N-1}d_t a^{n+1} b^n - a^0 b^0, \label{int_by_parts}
\end{equation}
and for non-negative real numbers $a,b,$ and $\epsilon>0$
\begin{equation}
 ab\le \frac{a^2}{2\epsilon}+\frac{\epsilon b^2}{2}. \label{youngs}
\end{equation}

\begin{lemma}
Given the functional spaces $V^f_h \subset V^f, Q^f_h \subset Q^f, V^p_h \subset V^p$ and $Q^p_h \subset Q^p$, the following inequalities hold true:
\begin{itemize}
 \item local trace-inverse inequality:
\begin{equation}
 ||\psi_h||^2_{L^2(\Gamma)} \leq \frac{C_{TI}}{h} ||\psi_h||^2_{L^2(\Omega^p)}, \quad \forall \psi_h \in Q^p_h; \label{traceInverse}
\end{equation}
 \item Cauchy-Schwarz inequality:
 \begin{equation}
  \bigg|\int_{\Omega^{f/p}} \boldsymbol v \cdot \boldsymbol u d \boldsymbol x\bigg| \le ||\boldsymbol v||_{L^2(\Omega^{f/p})} ||\boldsymbol u||_{L^2(\Omega^{f/p})} \quad \forall \boldsymbol v, \boldsymbol u \in V^{f/p}; \label{CS}
 \end{equation}
  \item Poincar\'e - Friedrichs inequality: 
 \begin{equation}
  ||\boldsymbol v||_{L^2(\Omega^{f/p})} \leq C_{PF} ||\nabla \boldsymbol v||_{L^2(\Omega^{f/p})} \quad \forall \boldsymbol v \in V^{f/p}; \label{PF}
 \end{equation}
 \item trace inequality:
 \begin{equation}
 ||\boldsymbol v||_{L^2(\Gamma)} \leq C_T ||\boldsymbol v||^{1/2}_{L^2(\Omega^{f/p})} ||\nabla \boldsymbol v||^{1/2}_{L^2(\Omega^{f/p})} \quad \forall \boldsymbol v \in V^{f/p}; \label{trace}
 \end{equation}
 \item Korn inequality:
 \begin{equation}
  ||\nabla \boldsymbol v||_{L^2(\Omega^{f/p})} \leq C_K  ||D(\boldsymbol v)||_{L^2(\Omega^{f/p})} \quad \forall \boldsymbol v \in V^{f/p}. \label{korn}
 \end{equation}
Here constants $C_{PF}, C_T$ and $C_K$ depend on the domain $\Omega$, and constant $C_{TI}$ depends on the angles in the finite element mesh.
\end{itemize}
\end{lemma}

Our analysis holds provided that the following regularity assumptions are satisfied by the exact solution of the problem.
\begin{assumption}\label{regularity_aassumptions}
 Let $X$ be a Banach space, and $(0,T)\subset \mathbb{R}$ a time interval. We define the $L^2$-space of functions $u:(0,T)\rightarrow X $ by
\begin{equation*}
L^2(0,T;X) = \{u| \;\textrm{u is measurable and } \;\int_0^T||u||^2_{X}dt <\infty  \},
\end{equation*}
and $L^\infty$-space by
\begin{equation*}
 L^\infty(0,T;X) = \{u| \; \textrm{u is measurable and } \;||u||_{X}\;\textrm{is essentially bounded} \}.
\end{equation*}
Then, the Sobolev space $W^{k,2}(0,T; X)=H^k(0,T;X)$ is defined to be the set of all functions $u \in L^2(0,T;X)$ whose distributional time derivative $D_t^{\alpha}u$ belongs to $L^2(0,T;X)$, for every $\alpha$ with $|\alpha|\leq k$.
We assume that the weak solution of \eqref{sys1}, complemented by the prescribed, interface, boundary and initial conditions, is such that
\begin{eqnarray}
 & &  \boldsymbol v \in H^1(0,T; H^{k+1}(\Omega^f)) \cap H^2(0,T; L^2(\Omega^f)), \nonumber \\
 & & p_f \in L^2(0,T; H^{s+1}(\Omega^f)), \nonumber\\
 & &  \boldsymbol \eta \in L^{\infty}(0,T; H^{k+1}(0,L)) \cap H^{2}(0,T; H^{k+1}(0,L)) \cap H^3(0,T; L^{2}(0,L)), \nonumber \\
 & &  \boldsymbol \xi \in L^{\infty}(0,T; H^{k+1}(0,L)) \cap H^{2}(0,T; H^{k+1}(0,L)) \cap H^3(0,T; L^{2}(0,L)), \nonumber  \\
 & &  \partial_t \boldsymbol \xi \in  L^{\infty}(0,T; H^{k+1}(0,L)), \nonumber \\
 & &  \partial_{tt} \boldsymbol \xi \in  L^{\infty}(0,T; L^2(0,L)), \label{true_sol} \\
 & &  \boldsymbol U \in L^{\infty}(0,T; H^{k+1}(\Omega^p))  \cap H^{2}(0,T; H^{k+1}(\Omega^p)) \cap H^3(0,T; L^{2}(\Omega^p)), \nonumber \\
 & &  \boldsymbol V \in L^{\infty}(0,T; H^{k+1}(\Omega^p)) \cap H^3(0,T; L^{2}(\Omega^p)), \nonumber \\
 & &  \partial_t \boldsymbol  V \in L^{\infty}(0,T; H^{k+1}(\Omega^p)), \nonumber \\
 & &  \partial_{tt} \boldsymbol  V \in L^{\infty}(0,T; L^2(\Omega^p)), \nonumber \\
 & & p_p \in L^{\infty}(0,T; H^{k+1}(\Omega^p)) \cap H^1(0,T; H^{k+1}(\Omega^p)) \cap H^2(0,T; L^2(\Omega^p)). \nonumber
 \end{eqnarray}
\end{assumption}

Then, our finite element spaces satisfy the approximation properties reported below.
\begin{lemma}\label{approximation_properties} 
Let $S_h $ be an orthogonal projection operator with respect to $a_f(\cdot, \cdot)$, onto $X_h^f$,  defined in~\eqref{sh1},
$P_h$ be the Lagrangian interpolation operator onto $V_h^p$, and let $\Pi^{f/p}_h$ be the $L^2$-orthogonal projections onto $Q_h^{f/p}$. Using piecewise polynomials of degree $k$ and $s$, we have:
\begin{eqnarray}
& & ||\boldsymbol v - S_h\boldsymbol v||_{H^1(\Omega^f)} \le C h^k ||\boldsymbol v||_{H^{k+1}(\Omega^f)}, \label{ap1}\\
& & ||p_{f} - \Pi^f_h p_{f}||_{L^2(\Omega^f)} \le C h^{s+1} ||p_{f}||_{H^{s+1}(\Omega^f)},  \label{pif}\\
& & ||\boldsymbol U - P_h\boldsymbol U||_{L^2(\Omega^p)} \le C h^{k+1} ||\boldsymbol U||_{H^{k+1}(\Omega^p)}, \\
& & ||\boldsymbol U - P_h\boldsymbol U||_{H^1(\Omega^p)} \le C h^k ||\boldsymbol U||_{H^{k+1}(\Omega^p)}, \\
& & ||p_p - \Pi^p_h p_p||_{L^2(\Omega^p)} \le C h^{k+1} ||p_p||_{H^{k+1}(\Omega^p)}, \\
& & ||p_p - \Pi^p_h p_p||_{H^1(\Omega^p)} \le C h^{k} ||p_p||_{H^{k+1}(\Omega^p)}. \label{ap2}
\end{eqnarray}
Finally, since $P_h|_{\Gamma}$  is a Lagrangian interpolant, we obtain:
\begin{eqnarray}
& & ||\boldsymbol U - P_h\boldsymbol U||_{L^2(\Gamma)} =  ||\boldsymbol \eta - (P_h|_{\Gamma})\boldsymbol \eta||_{L^2(\Gamma)} \le C h^{k+1} ||\boldsymbol \eta||_{H^{k+1}(\Gamma)}, \\
& & ||\boldsymbol U - P_h\boldsymbol U||_{H^1(\Gamma)} =  ||\boldsymbol \eta - (P_h|_{\Gamma})\boldsymbol \eta||_{H^1(\Gamma)}\le C h^k ||\boldsymbol \eta||_{H^{k+1}(\Gamma)}, \\
& & || \boldsymbol V - P_h \boldsymbol V||_{L^2(\Gamma)}  =  ||\boldsymbol \xi - (P_h|_{\Gamma})\boldsymbol \xi||_{L^2(\Gamma)}\le C h^{k+1} ||\boldsymbol \xi||_{H^{k+1}(\Gamma)}.
\end{eqnarray}
\end{lemma}
\begin{proof}
The proof of~\eqref{ap1} follows from~\cite{girault1979finite}. Precisely, since we assumed that $V_h^f$ and $Q_h^f$ satisfy the discrete \emph{inf-sup} condition~\eqref{infsup}, there exists a constant $C$ 
such that for $\boldsymbol v \in V^f$, with $\nabla \cdot \boldsymbol v =0,$ we have
\begin{equation*}
 \inf_{\boldsymbol x_h \in X^f_h} ||\boldsymbol v -\boldsymbol x_h||_{H^1(\Omega^f)} \le C   \inf_{\boldsymbol v_h \in V^f_h} ||\boldsymbol v -\boldsymbol v_h||_{H^1(\Omega^f)}.
\end{equation*}
For the proof of the other inequalities, see~\cite{ciarlet1978finite}.
\end{proof}

\begin{lemma} \label{Lemma_consistency}
 (Consistency errors:) The following inequalities hold:
\begin{equation*}
\Delta t \sum_{n=0}^{N-1}||d_t \boldsymbol \varphi^{n+1}-\partial_t \boldsymbol \varphi^{n+1}||^2_{L^2(\Omega)} \leq C \Delta t^2 ||\partial_{tt} \boldsymbol \varphi||^2_{L^2(0,T;L^2(\Omega))},
\end{equation*}
\begin{equation*}
\Delta t\sum_{n=1}^{N-1}||d_{tt} \boldsymbol \varphi^{n+1}-\partial_t (d_t \boldsymbol \varphi^{n+1})||^2_{L^2(\Omega)}\leq C \Delta t^2 ||\partial_{ttt} \boldsymbol \varphi||^2_{L^2(0,T;L^2(\Omega))},
\end{equation*}
\begin{equation*}
\Delta t^3 \sum_{n=1}^{N-1}||d_{tt} \boldsymbol \varphi^{n+1}||^2_{L^2(\Omega)} \leq C \Delta t^2 ||\partial_{tt} \boldsymbol \varphi||^2_{L^2(0,T;L^2(\Omega))},
\end{equation*}
\begin{equation*}
 \Delta t  \sum_{n=0}^{N-1} ||\nabla (\boldsymbol \varphi^{n+1}-\boldsymbol \varphi^n)||^2_{L^2(\Omega)} \leq C \Delta t^2 ||\partial_t \boldsymbol \varphi||^2_{L^2(0,T; H^1(\Omega))},
\end{equation*}
\begin{equation*}
 \Delta t\sum_{n=0}^{N-1}||\nabla (d_t \boldsymbol \varphi^{n+1}-\partial_t \boldsymbol \varphi^{n+1})||^2_{L^2(\Omega)} \leq C \Delta t^2 ||\partial_{tt} \boldsymbol \varphi||^2_{L^2(0,T;H^1(\Omega))},
\end{equation*}
\begin{equation*}
\Delta t^3\sum_{n=1}^{N-1} ||D(d_{tt} \boldsymbol \varphi^{n+1})||^2_{L^2(\Omega)} \leq C \Delta t^2 ||\partial_{tt} \boldsymbol \varphi||^2_{L^2(0,T;H^1(\Omega))},
\end{equation*}
\begin{equation*}
\Delta t^3 \sum_{n=1}^{N-1} \mathcal{M}(d_{tt} \boldsymbol \eta^{n+1}) \leq C \Delta t^2 ||\partial_{tt} \boldsymbol \eta||^2_{L^2(0,T;H^1(0,L))},
\end{equation*}
\begin{equation*}
||d_t \boldsymbol \varphi^{N}-\partial_t \boldsymbol \varphi^{N}  ||^2_{L^2(\Omega)} \le \Delta t^2  \max_{0 \le n \le N} ||\partial_{tt} \boldsymbol \varphi^n||^2_{L^2(\Omega)}= \Delta t^2 ||\partial_{tt} \boldsymbol \varphi||_{l^\infty(0,T;L^2(\Omega))},
\end{equation*}
\begin{equation*}
 \Delta t^2 ||d_t \boldsymbol \varphi^{N}||^2_{L^2(\Omega)} \leq \Delta t^2 \max_{0\le n \le N} ||\partial_t \boldsymbol \varphi^n||^2_{L^2(\Omega)} = \Delta t^2 ||\partial_t \boldsymbol \varphi||_{l^\infty(0,T;L^2(\Omega))},
\end{equation*}
\begin{equation*}
 \Delta t^2 ||D(d_t \boldsymbol \varphi^{N})||^2_{L^2(\Omega)} \leq \Delta t^2 \max_{0\le n \le N} ||\partial_t D(\boldsymbol \varphi^n)||^2_{L^2(\Omega)} = \Delta t^2 ||\partial_t \boldsymbol \varphi||_{l^\infty(0,T;H^1(\Omega))},
\end{equation*}
\begin{equation*}
 \Delta t^2 \mathcal{M}(d_t \boldsymbol \eta^{N}) \leq \Delta t^2 \max_{0\le n \le N} \mathcal{M}( \partial t \boldsymbol \eta^n) = \Delta t^2 ||\partial_t \boldsymbol \eta||_{l^\infty(0,T;H^1(0,L))}.
\end{equation*}
\end{lemma}
\begin{proof}
We will prove the first three inequalities. The proofs for other inequalities are similar. Using Cauchy-Schwartz inequality, we have
\begin{equation*}
\Delta t \sum_{n=0}^{N-1}||d_t \boldsymbol \varphi^{n+1}-\partial_t \boldsymbol \varphi^{n+1}||^2_{L^2(\Omega)} = \Delta t  \sum_{n=0}^{N-1}  \int_{\Omega} \bigg|\frac{1}{\Delta t} \int_{t^n}^{t^{n+1}} (t-t^n)\partial_{tt} \boldsymbol \varphi (t) dt \bigg|^2 dx 
\end{equation*}
\begin{equation*}
 \le \frac{1}{\Delta t} \int_{\Omega} \sum_{n=0}^{N-1} \bigg( \int_{t^n}^{t^{n+1}} |t-t^n|^2 dt \int_{t^n}^{t^{n+1}} |\partial_{tt}  \boldsymbol \varphi|^2 dt \bigg) dx \le C \Delta t^2 \int_{\Omega} \int_0^T |\partial_{tt} \boldsymbol \varphi|^2 dt dx \le C \Delta t^2  ||\boldsymbol \varphi||^2_{L^2(0,T; L^2(\Omega)).}
\end{equation*}
To prove the next two inequalities, we integrate by parts twice, and use Cauchy-Schwartz inequality:
\begin{equation*}
\Delta t\sum_{n=1}^{N-1}||d_{tt} \boldsymbol \varphi^{n+1}-\partial_t (d_t \boldsymbol \varphi^{n+1})||^2_{L^2(\Omega)} = \Delta t \sum_{n=0}^{N-1} \bigg|\bigg|\frac{1}{\Delta t^2} \bigg( \int_{t^n}^{t^{n+1}} (t-t^n)\partial_{tt} \boldsymbol \varphi dt + \int_{t^{n}}^{t^{n-1}} (t-t^{n-1})\partial_{tt} \boldsymbol \varphi dt \bigg)\bigg|\bigg|^2_{L^2(\Omega)} 
\end{equation*}
\begin{equation*}
 = \frac{1}{\Delta t^3} \sum_{n=0}^{N-1} \bigg|\bigg| \bigg( \frac{\Delta t^2}{2}\int_{t^n}^{t^{n+1}} \partial_{ttt} \boldsymbol \varphi dt -\int_{t^n}^{t^{n+1}} \frac{(t-t^n)^2}{2} \partial_{ttt} \boldsymbol \varphi dt - \int_{t^{n}}^{t^{n-1}} \frac{(t-t^{n-1})^2}{2}\partial_{ttt} \boldsymbol \varphi dt \bigg)\bigg|\bigg|^2_{L^2(\Omega)} 
\end{equation*}
\begin{equation*}
 \le \frac{1}{\Delta t^3} \int_{\Omega} \sum_{n=0}^{N-1} \bigg(\frac{\Delta t^4}{4} \bigg|\int_{t^n}^{t^{n+1}} \partial_{ttt} \boldsymbol \varphi dt\bigg|^2+\bigg|\int_{t^n}^{t^{n+1}} \frac{(t-t^n)^2}{2} \partial_{ttt} \boldsymbol \varphi dt \bigg|^2 +\bigg|\int_{t^{n}}^{t^{n-1}} \frac{(t-t^{n-1})^2}{2}\partial_{ttt} \boldsymbol \varphi dt  \bigg|^2 \bigg)dx
\end{equation*}
\begin{equation*}
 \le \frac{1}{\Delta t^3} \int_{\Omega} \sum_{n=0}^{N-1} \bigg(\frac{\Delta t^5}{4} \int_{t^n}^{t^{n+1}} |\partial_{ttt} \boldsymbol \varphi|^2 dt+\frac{\Delta t^5}{10}\int_{t^n}^{t^{n+1}} |\partial_{ttt} \boldsymbol \varphi|^2 dt +\frac{\Delta t^5}{10}\int_{t^{n}}^{t^{n-1}} |\partial_{ttt} \boldsymbol \varphi|^2 dt\bigg)dx
\end{equation*}
\begin{equation*}
\leq C \Delta t^2 ||\partial_{ttt} \boldsymbol \varphi||^2_{L^2(0,T;L^2(\Omega))}
\end{equation*}
Finally,
\begin{equation*}
\Delta t^3 \sum_{n=1}^{N-1}||d_{tt} \boldsymbol \varphi^{n+1}||^2_{L^2(\Omega)} = \frac{1}{\Delta t} \sum_{n=0}^{N-1} \bigg|\bigg| \bigg( \int_{t^n}^{t^{n+1}}\partial_{t} \boldsymbol \varphi dt + \int_{t^{n}}^{t^{n-1}} \partial_{t} \boldsymbol \varphi dt \bigg)\bigg|\bigg|^2_{L^2(\Omega)} 
\end{equation*}
\begin{equation*}
=\frac{1}{\Delta t} \sum_{n=0}^{N-1} \bigg|\bigg| \bigg( \int_{t^n}^{t^{n+1}} (t-t^{n+1}) \partial_{tt} \boldsymbol \varphi dt + \int_{t^{n}}^{t^{n-1}} (t-t^{n-1})\partial_{tt} \boldsymbol \varphi dt \bigg)\bigg|\bigg|^2_{L^2(\Omega)} 
\end{equation*}
\begin{equation*}
 \le \frac{1}{\Delta t} \int_{\Omega} \sum_{n=0}^{N-1} \bigg( \int_{t^n}^{t^{n+1}} |t-t^{n+1}|^2 dt \int_{t^n}^{t^{n+1}}  |\partial_{tt}  \boldsymbol \varphi|^2 dt +\int_{t^n}^{t^{n-1}} |t-t^{n-1}|^2 dt \int_{t^n}^{t^{n-1}}  |\partial_{tt}  \boldsymbol \varphi|^2 dt\bigg) dx
\end{equation*}
\begin{equation*}
\le C \Delta t^2 \int_{\Omega} \int_0^T |\partial_{tt} \boldsymbol \varphi|^2 dt dx \leq C \Delta t^2 ||\partial_{tt} \boldsymbol \varphi||^2_{L^2(0,T;L^2(\Omega))}. 
\end{equation*}
\end{proof}


\begin{lemma}\label{consistency_error}
The following estimate holds:
\begin{equation*}
 \Delta t \sum_{n=0}^{N-1} \big( \mathcal{R}^{n+1}_f(\delta_f^{n+1})+\mathcal{R}^{n+1}_s(\delta_f^{n+1})+\mathcal{R}^{n+1}_v(\delta_f^{n+1})+\mathcal{R}^{n+1}_p(\delta_f^{n+1}) \big) \le  C \Delta t^2 \bigg( ||\partial_{tt} \boldsymbol v||^2_{L^2(0,T;L^2(\Omega^f))}
\end{equation*}
\begin{equation*}
+||\partial_{tt} \boldsymbol \xi ||^2_{L^2(0,T;L^2(0,L))}+||\partial_{t} p_p||^2_{L^2(0,T;H^1(\Omega^p))}+||\partial_{tt} p_p||^2_{L^2(0,T;L^2(\Omega^p))}+||\partial_{tt} \boldsymbol U||^2_{L^2(0,T;H^1(\Omega^p))}
\end{equation*}
\begin{equation*}
+||\partial_{ttt} \boldsymbol U||^2_{L^2(0,T;L^2(\Omega^p))}+||\partial_{tt} \boldsymbol V||^2_{L^2(0,T;L^2(\Omega^p))}+||\partial_{ttt} \boldsymbol \eta||^2_{L^2(0,T;L^2(0,L))}+||\partial_{ttt} \boldsymbol V||^2_{L^2(0,T;L^2(\Omega^p))}
\end{equation*}
\begin{equation*}
+||\partial_{ttt} \boldsymbol \xi||^2_{L^2(0,T;L^2(0,L))}+||\partial_{tt} \boldsymbol \eta||^2_{L^2(0,T;H^1(0,L))}+||\partial_{tt} \boldsymbol U||_{l^\infty(0,T;L^2(\Omega^p))} +||\partial_{tt} \boldsymbol \eta||_{l^\infty(0,T;L^2(0,L))}
\end{equation*}
\begin{equation*}
+ ||\partial_{tt} \boldsymbol V||_{l^\infty(0,T;L^2(\Omega^p))}+||\partial_{tt} \boldsymbol \xi||_{l^\infty(0,T;L^2(0,L))}+||\partial_t \boldsymbol V||_{l^\infty(0,T;L^2(\Omega^p))}+ ||\partial_t \boldsymbol \xi||_{l^\infty(0,T;L^2(0,L))}
\end{equation*}
\begin{equation*}
+||\partial_t \boldsymbol U||_{l^\infty(0,T;H^1(\Omega^p))} +||\partial_t \boldsymbol \eta||_{l^\infty(0,T;H^1(0,L))}\bigg) + \mathcal{A}(\delta_f, \delta_p, \delta_v, \delta_s), 
\end{equation*}
where 
\begin{equation*}
\mathcal{A}(\delta_f, \delta_p, \delta_v, \delta_s) = \frac{\gamma \Delta t}{4} \sum_{n=0}^{N-1} ||D(\delta_f^{n+1})||^2_{L^2(\Omega^f)}  +\frac{\Delta t}{4}\sum_{n=0}^{N-1}||\sqrt{\kappa} \nabla \delta_p^{n+1}||^2_{L^2(\Omega^p)}  + \epsilon \bigg(  ||\delta_v^N||^2_{L^2(\Omega^p)} +   ||\delta_v|_{\Gamma}^N||^2_{L^2(0,L)}
\end{equation*}
\begin{equation*}
+ ||D(\delta_s^N)||^2_{L^2(\Omega^p)}  +||\nabla \cdot \delta_s^N||^2_{L^2(\Omega^p)}+  \mathcal{M}(\delta_s|_{\Gamma}^N) \bigg)+ C \Delta t \sum_{n=1}^{N-1} \bigg(  ||\delta_v^n||^2_{L^2(\Omega^p)}+||\delta_v|_{\Gamma}^n||^2_{L^2(0,L)}
\end{equation*}
\begin{equation*}
+ ||\delta_s^n||^2_{L^2(\Omega^p)}+||D(\delta_s^n)||^2_{L^2(\Omega^p)}
+||\nabla \cdot \delta_s^n||^2_{L^2(\Omega^p)}   
+ ||\delta_s|_{\Gamma}^n||^2_{L^2(\Gamma)}+ \mathcal{M}(\delta_s|_{\Gamma}^n) \bigg).
\end{equation*}
\end{lemma}

\begin{proof}
Using the formula for integration by parts in time~\eqref{int_by_parts}, the consistency errors are bounded as follows:
\begin{itemize}
\item 
$ \Delta t \displaystyle\sum_{n=0}^{N-1}\mathcal{R}^{n+1}_f(\delta_f^{n+1}) \le C \Delta t \sum_{n=0}^{N-1}||d_t \boldsymbol v^{n+1}-\partial_t \boldsymbol v^{n+1}||^2_{L^2(\Omega^f)} + C \Delta t \sum_{n=0}^{N-1} ||\nabla(p_p^{n+1}-p_p^n)||^2_{L^2(\Omega^p)}
$
\begin{equation*}
 +C \Delta t \sum_{n=0}^{N-1}||d_t \boldsymbol v^{n+1}\cdot \boldsymbol \tau-\partial_t \boldsymbol v^{n+1} \cdot \boldsymbol \tau ||^2_{L^2(\Gamma)}+ \frac{\gamma \Delta t}{4} \sum_{n=0}^{N-1} ||D(\delta_f^{n+1})||^2_{L^2(\Omega^f)},
\end{equation*}
\item
$ \Delta t \displaystyle\sum_{n=0}^{N-1}\mathcal{R}^{n+1}_p(\delta_p^{n+1}) \le C \Delta t \sum_{n=0}^{N-1}||d_t p_p^{n+1}-\partial_t p_p^{n+1} ||^2_{L^2(\Omega^p)}  +\frac{\Delta t}{4}\sum_{n=0}^{N-1}||\sqrt{\kappa} \nabla \delta_p^{n+1}||^2_{L^2(\Omega^p)}
$
\begin{equation*}
+C \Delta t \sum_{n=0}^{N-1}||\nabla ( d_t \boldsymbol U^{n+1}-\partial_t \boldsymbol U^{n+1})||^2_{L^2(\Omega^p)},
\end{equation*}
\item
$\Delta t \displaystyle\sum_{n=0}^{N-1}\mathcal{R}^{n+1}_v(d_t \delta_v^{n+1}) = -\rho_{p} \int_{\Omega^p} (d_t \boldsymbol U^N - \partial_t \boldsymbol U^N ) \cdot \delta_v^N d \boldsymbol x
-\frac{\rho_{p}}{2} \Delta t \int_{\Omega^p} d_t \boldsymbol V^N \cdot \delta_v^N d \boldsymbol x$
\begin{equation*}
 + \rho_{p} \Delta t \sum_{n=1}^{N-1} \int_{\Omega^p} (d_{tt}\boldsymbol U^{n+1} - \partial_t (d_t \boldsymbol U^{n+1})) \cdot \delta_v^n d \boldsymbol x    + \frac{\rho_{p}}{2} \Delta t^2 \sum_{n=1}^{N-1} \int_{\Omega^p} d_{tt}\boldsymbol V^{n+1} \cdot \delta_v^n d \boldsymbol x
 \end{equation*}
 \begin{equation*}
 -\rho_{m}r_m \int_{\Gamma} (d_t \boldsymbol U^N - \partial_t \boldsymbol U^N ) \cdot \delta_v^N d x + \rho_{m}r_m \Delta t \sum_{n=1}^{N-1} \int_{\Gamma} (d_{tt} \boldsymbol U^{n+1}- \partial_t (d_t \boldsymbol U^{n+1}) ) \cdot \delta_v^n d x 
  \end{equation*}
\begin{equation*}
 -\frac{\rho_{m}r_m}{2} \Delta t  \int_{\Gamma} d_t \boldsymbol V^N \cdot \delta_v^N d \boldsymbol x   + \frac{\rho_{m}r_m}{2} \Delta t^2 \sum_{n=1}^{N-1} \int_{\Gamma} d_{tt} \boldsymbol V^{n+1} \cdot \delta_v^n d x 
\leq C_{\epsilon} ||d_t \boldsymbol U^{N}-\partial_t \boldsymbol U^{N}||^2_{L^2(\Omega^p)} 
\end{equation*}
\begin{equation*}
+C \Delta t \sum_{n=1}^{N-1}||d_{tt} \boldsymbol U^{n+1}-\partial_t (d_t \boldsymbol U^{n+1} )||^2_{L^2(\Omega^p)}+C_{\epsilon} \Delta t^2 ||d_t \boldsymbol V^{N}||^2_{L^2(\Omega^p)}  +C \Delta t^3 \sum_{n=1}^{N-1}||d_{tt} \boldsymbol V^{n+1}||^2_{L^2(\Omega^p)}
\end{equation*}
\begin{equation*}
+C_{\epsilon} ||d_t \boldsymbol U^{N}-\partial_t \boldsymbol U^{N}||^2_{L^2(\Gamma)}
 +C \Delta t \sum_{n=1}^{N-1}||d_{tt}\boldsymbol U^{n+1}-\partial_t (d_t \boldsymbol U^{n+1})||^2_{L^2(\Gamma)} +C_{\epsilon} \Delta t^2||d_t \boldsymbol V^{N}||^2_{L^2(\Gamma)}
\end{equation*}
\begin{equation*}
 +C \Delta t^3 \sum_{n=1}^{N-1}||d_{tt} \boldsymbol V^{n+1}||^2_{L^2(\Gamma)}  + \epsilon  ||\delta_v^N||^2_{L^2(\Omega^p)} + \epsilon  ||\delta_v^N||^2_{L^2(\Gamma)}+ C \Delta t \sum_{n=1}^{N-1} ||\delta_v^n||^2_{L^2(\Omega^p)}.
+  C\Delta t \sum_{n=1}^{N-1} ||\delta_v^n||^2_{L^2(\Gamma)}
\end{equation*}
\item
$\Delta t \displaystyle\sum_{n=0}^{N-1}\mathcal{R}^{n+1}_s(d_t \delta_s^{n+1}) = \rho_{p} \int_{\Omega^p} (d_t \boldsymbol V^N - \partial_t \boldsymbol V^N) \cdot \delta_s^N d \boldsymbol x
- \frac{1}{2} \Delta t a_e(d_t \boldsymbol U^N, \delta_s^N)$
\begin{equation*}
- \rho_{p} \Delta t \sum_{n=1}^{N-1} \int_{\Omega^p} (d_{tt} \boldsymbol V^{n+1} - \partial_t (d_t \boldsymbol V^{n+1})) \cdot \delta_s^n d \boldsymbol x + \frac{\Delta t^2}{2} \sum_{n=1}^{N-1} a_e(d_{tt} \boldsymbol U^{n+1}, \delta_s^n)
\end{equation*}
\begin{equation*}
 +\rho_{m}r_m \int_{\Gamma} (d_t \boldsymbol V^N - \partial_t \boldsymbol V^N) \cdot \delta_s^N d x- \rho_{m}r_m \Delta t \sum_{n=1}^{N-1} \int_{\Gamma} (d_{tt} \boldsymbol V^{n+1} - \partial_t (d_t \boldsymbol V^{n+1})) \cdot \delta_s^n d x
\end{equation*}
\begin{equation*}
 - \frac{\Delta t}{2} a_m(d_t \boldsymbol U|_{\Gamma}^N, \delta_s|_{\Gamma}^N)
 + \frac{\Delta t^2}{2}\sum_{n=1}^{N-1} a_m(d_{tt} \boldsymbol U|_{\Gamma}^{n+1}, \delta_s|_{\Gamma}^n) \leq 
 C_{\epsilon} ||d_t \boldsymbol V^{N}-\partial_t \boldsymbol V^{N}||^2_{L^2(\Omega^p)} 
\end{equation*}
\begin{equation*}
+C \Delta t \sum_{n=1}^{N-1}||d_{tt} \boldsymbol V^{n+1}-\partial_t (d_t \boldsymbol V^{n+1})||^2_{L^2(\Omega^p)}+C_{\epsilon} \Delta t^2||D(d_t \boldsymbol U^{N})||^2_{L^2(\Omega^p)}  +C_{\epsilon} \Delta t^2||\nabla \cdot d_t \boldsymbol U^{N}||^2_{L^2(\Omega^p)}
\end{equation*}
\begin{equation*}
+C \Delta t^3 \sum_{n=1}^{N-1} ||D(d_{tt} \boldsymbol U^{n+1})||^2_{L^2(\Omega^p)}    +C \Delta t^3 \sum_{n=1}^{N-1}||\nabla \cdot d_{tt}\boldsymbol U^{n+1}||^2_{L^2(\Omega^p)} + C_{\epsilon} ||d_t \boldsymbol V^{N}-\partial_t \boldsymbol V^{N}||^2_{L^2(\Gamma)}
\end{equation*}
\begin{equation*}
  +C \Delta t \sum_{n=1}^{N-1}||d_{tt} \boldsymbol V^{n+1} - \partial_t(d_t \boldsymbol V^{n+1})||^2_{L^2(\Gamma)}
+C_{\epsilon} \Delta t^2 \mathcal{M}(d_t \boldsymbol U|_{\Gamma}^N) +C \Delta t^3 \sum_{n=1}^{N-1} \mathcal{M}(d_{tt} \boldsymbol U|_{\Gamma}^{n+1}) 
\end{equation*}
\begin{equation*}
+ C \Delta t \sum_{n=1}^{N-1} ||\delta_s^n||^2_{L^2(\Omega^p)}+\epsilon ||D(\delta_s^N)||^2_{L^2(\Omega^p)}  +\epsilon||\nabla \cdot \delta_s^N||^2_{L^2(\Omega^p)}  + C \Delta t  \sum_{n=1}^{N-1}||D(\delta_s^n)||^2_{L^2(\Omega^p)}
\end{equation*}
\begin{equation*}
  +C \Delta t \sum_{n=1}^{N-1}||\nabla \cdot \delta_s^n||^2_{L^2(\Omega^p)} + C \Delta t \sum_{n=1}^{N-1} ||\delta_s^n||^2_{L^2(\Gamma)}
+ \epsilon \mathcal{M}(\delta_s|_{\Gamma}^N)+ C\Delta t \sum_{n=1}^{N-1}\mathcal{M}(\delta_s|_{\Gamma}^n).
\end{equation*}
\end{itemize}
The final consistency error estimate follows by applying Lemma~\ref{Lemma_consistency}.
\end{proof}


\begin{lemma} \label{lemma_interpolation}
(Interpolation errors) The following inequalities hold:
\begin{equation*}
\Delta t \sum_{n=0}^{N-1} ||d_t \theta_p^{n+1}||^2_{L^2(\Omega^p)} \le \Delta t \sum_{n=0}^{N-1}||\partial_t \theta_p^n||^2_{L^2(0,T;L^2(\Omega^p))} \le h^{2k+2} ||\partial_t p_p||^2_{l^2(0,T;H^{k+1}(\Omega^p))},
\end{equation*}
\begin{equation*}
\Delta t \sum_{n=0}^{N-1} ||\nabla d_t \theta_f^{n+1}||^2_{L^2(\Omega^f)} \le \Delta t \sum_{n=0}^{N-1}  ||\partial_t \theta_f||^2_{L^2(0,T;H^1(\Omega^f))} \le h^{2k} ||\partial_t \boldsymbol v||^2_{l^2(0,T;H^{k+1}(\Omega^f))},
\end{equation*}
\begin{equation*}
 \Delta t \sum_{n=0}^{N-1} ||d_{tt}\theta_s^{n+1}||^2_{L^2(\Omega^p)} \le \Delta t \sum_{n=0}^{N-1}  ||\partial_{tt} \theta_s||^2_{L^2(0,T;L^2(\Omega^p))}\le h^{2k+2} ||\partial_{tt} \boldsymbol U||^2_{l^2(0,T;H^{k+1}(\Omega^p))},
\end{equation*}
\begin{equation*}
 \Delta t \sum_{n=0}^{N-1} \bigg|\bigg|\frac{\theta_v^{n+1}-\theta_v^{n-1}}{2 \Delta t}\bigg| \bigg|^2_{L^2(\Gamma)} \le \Delta t \sum_{n=0}^{N-1}  ||\partial_t \theta_v||^2_{L^2(0,T;L^2(\Gamma))}\le h^{2k+2} ||\partial_{tt} \boldsymbol \xi||^2_{l^2(0,T;H^{k+1}(0,L)},
\end{equation*}
\begin{equation*}
 \Delta t \sum_{n=0}^{N-1} \mathcal{M}(\frac{\theta_s|_{\Gamma}^{n+1}-\theta_s|_{\Gamma}^{n-1}}{2 \Delta t})\le \Delta t \sum_{n=0}^{N-1}  ||\partial_t \theta_s|_{\Gamma}||^2_{L^2(0,T;H^1(0,L))}\le h^{2k} ||\partial_{tt} \boldsymbol \eta||^2_{l^2(0,T;H^{k+1}(0,L))},
\end{equation*}
\begin{equation*}
 \Delta t \sum_{n=0}^{N-1} \bigg(||D(\theta_f^{n+1})||^2_{L^2(\Omega^f)} +||\nabla \theta_p^{n+1}||^2_{L^2(\Omega^p)}+||\nabla \theta_p^{n}||^2_{L^2(\Omega^p)}\bigg) \leq \Delta t \sum_{n=0}^{N} h^{2k} \big(||\boldsymbol v^n||^2_{H^{k+1}(\Omega^f)}+||p_p^n||^2_{H^{k+1}(\Omega^p)} \big)
\end{equation*}
\begin{equation*}
  \le h^{2k} \big(||\boldsymbol v||^2_{l^2(0,T;H^{k+1}(\Omega^f))}+||p_p||^2_{l^2(0,T;H^{k+1}(\Omega^p))} \big).
\end{equation*}
\end{lemma}
\begin{proof}
 Inequalities in Lemma~\ref{lemma_interpolation} can be easily shown using manipulations similar to the ones in Lemma~\ref{consistency_error}, and approximation properties~\eqref{ap1}-\eqref{ap2}.
\end{proof}

\if 1=0
Taking into account the estimates above, and summing over $0 \le n \le N-1$, we get
\begin{equation*}
 \mathcal{E}_{\delta}^{N}  +\frac{7 \gamma \Delta t}{12}  \sum_{n=0}^{N-1} ||D(\delta_f^{n+1})||^2_{L^2(\Omega^f)}  + \frac{3 \Delta t}{5} \sum_{n=0}^{N-1}||\sqrt{\kappa} \nabla \delta_p^{n+1}||^2_{L^2(\Omega^p)}\le \sum_{n=0}^{N-1} \Delta t \bigg(\mathcal{R}^{n+1}_f(\delta_f^{n+1})
 \end{equation*}
\begin{equation*}
+ \mathcal{R}_s^{n+1}(d_t \delta_s^{n+1})+\mathcal{R}_v^{n+1}(d_t \delta_v^{n+1})+ \mathcal{R}_p^{n+1}(\delta_p^{n+1})\bigg)+ C \Delta t \sum_{n=0}^{N-1}  ||p_f^{n+1} - \lambda^{n+1}_{h}||^2_{L^2(\Omega^f)}
\end{equation*}
\begin{equation*}
 +C \Delta t \sum_{n=0}^{N-1} \bigg(||D(\theta_f^{n+1})||^2_{L^2(\Omega^f)} +||\nabla \theta_p^{n+1}||^2_{L^2(\Omega^p)}+||\nabla \theta_p^n||^2_{L^2(\Omega^p)} +||d_t \theta_f^{n+1}||^2_{L^2(\Omega^f)}+||d_t \theta_f|_{\Gamma}^{n+1}\cdot \boldsymbol \tau||^2_{L^2(0,L)}
\end{equation*}
 \begin{equation*}
+ ||\nabla d_t \theta_f^{n+1}||^2_{L^2(\Omega^f)} +||\nabla d_t\theta_s^{n+1}||^2_{L^2(\Omega^p)}\bigg)+\rho_{p} \Delta t \sum_{n=0}^{N-1}\int_{\Omega^p} d_t \theta_s^{n+1} \cdot d_t \delta_v^{n+1} d \boldsymbol x+\Delta t  \sum_{n=0}^{N-1}  b_{ep}(\theta^{n+1}_{p}, d_t \delta_s^{n+1})
\end{equation*}
\begin{equation*}
-\rho_{m} r_m \Delta t  \sum_{n=0}^{N-1} \int_0^L \theta_v|_{\Gamma}^{n+1/2} \cdot d_t \delta_v|^{n+1}_{\Gamma} dx 
 +\rho_{m}r_m \Delta t \sum_{n=0}^{N-1} \int_0^L d_t \theta_s|_{\Gamma}^{n+1} \cdot d_t \delta_v|^{n+1}_{\Gamma} d x
\end{equation*}
\begin{equation}
-\rho_{m} r_m\Delta t  \sum_{n=0}^{N-1}\int_0^L d_t \theta_v|_{\Gamma}^{n+1} \cdot d_t \delta_s|^{n+1}_{\Gamma} dx \Delta t - \sum_{n=0}^{N-1} a_m (\theta_s|_{\Gamma}^{n+1/2}, d_t \delta_s|_{\Gamma}^{n+1})+ \Delta t \sum_{n=0}^{N-1} c_{ep}(\theta^{n+1}_{p}, d_t \delta_s^{n+1}). 
\end{equation}
\fi



\bibliographystyle{plain}
\bibliography{shellBiot}

\begin{thebibliography}{10}

\bibitem{armentano1995arterial}
R.L. Armentano, J.G. Barra, J.~Levenson, A.~Simon, and R.H. Pichel.
\newblock Arterial wall mechanics in conscious dogs: assessment of viscous,
  inertial, and elastic moduli to characterize aortic wall behavior.
\newblock {\em Circ. Res.}, 76(3):468--478, 1995.

\bibitem{badia2008fluid}
S.~Badia, F.~Nobile, and C.~Vergara.
\newblock Fluid-structure partitioned procedures based on {R}obin transmission
  conditions.
\newblock {\em J. Comput. Phys.}, 227:7027--7051, 2008.

\bibitem{badia2009robin}
S.~Badia, F.. Nobile, and C.~Vergara.
\newblock Robin-{R}obin preconditioned {K}rylov methods for fluid-structure
  interaction problems.
\newblock {\em Comput. Methods Appl. Mech. Eng.}, 198(33-36):2768--2784, 2009.

\bibitem{badia2008splitting}
S.~Badia, A.~Quaini, and A.~Quarteroni.
\newblock Splitting methods based on algebraic factorization for
  fluid-structure interaction.
\newblock {\em SIAM J. Sci. Comput.}, 30(4):1778--1805, 2008.

\bibitem{badia2009coupling}
S.~Badia, A.~Quaini, and A.~Quarteroni.
\newblock Coupling {B}iot and {N}avier--{S}tokes equations for modelling
  fluid--poroelastic media interaction.
\newblock {\em Journal of Computational Physics}, 228(21):7986--8014, 2009.

\bibitem{barker2010scalable}
A.T. Barker and X.C. Cai.
\newblock Scalable parallel methods for monolithic coupling in fluid-structure
  interaction with application to blood flow modeling.
\newblock {\em Journal of Computational Physics}, 229(3):642--659, 2010.

\bibitem{bauer1979separate}
R.D. Bauer, R.~Busse, A.~Schabert, Y.~Summa, and E.~Wetterer.
\newblock Separate determination of the pulsatile elastic and viscous forces
  developed in the arterial wall {\em{in vivo}}.
\newblock {\em Pfl{\"u}gers Archiv}, 380(3):221--226, 1979.

\bibitem{botkin2007dispersion}
N.D. Botkin, K.-H. Hoffmann, O.A. Pykhteev, and V.L. Turova.
\newblock Dispersion relations for acoustic waves in heterogeneous
  multi-layered structures contacting with fluids.
\newblock {\em Journal of the Franklin Institute}, 344(5):520--534, 2007.

\bibitem{bukavc2012fluid}
M.~Buka{\v{c}}, S.~{\v{C}}ani{\'c}, R.~Glowinski, J.~Tamba{\v{c}}a, and
  A.~Quaini.
\newblock Fluid-structure interaction in blood flow capturing non-zero
  longitudinal structure displacement.
\newblock {\em Journal of Computational Physics}, 2012.

\bibitem{burman2009stabilization}
E.~Burman and M.~A. Fern{\'a}ndez.
\newblock Stabilization of explicit coupling in fluid-structure interaction
  involving fluid incompressibility.
\newblock {\em Comput. Methods Appl. Mech. Eng.}, 198:766--784, 2009.

\bibitem{SunBorMar}
S.~Canic, B.~Muha, and M.~Bukac.
\newblock Stability of the kinematically coupled $\beta$-scheme for
  fluid-structure interaction problems in hemodynamics.
\newblock {\em Submitted.}, 2012, arXiv:1205.6887.

\bibitem{canic2006modeling}
S.~Canic, J.~Tambaca, G.~Guidoboni, A.~Mikelic, C.J. Hartley, and
  D.~Rosenstrauch.
\newblock Modeling viscoelastic behavior of arterial walls and their
  interaction with pulsatile blood flow.
\newblock {\em SIAM Journal on Applied Mathematics}, 67(1):164--193, 2006.

\bibitem{causin2005added}
P.~Causin, J.F. Gerbeau, and F.~Nobile.
\newblock Added-mass effect in the design of partitioned algorithms for
  fluid-structure problems.
\newblock {\em Comput. Methods Appl. Mech. Eng.}, 194(42-44):4506--4527, 2005.

\bibitem{chung2012effect}
S.~Chung and K.~Vafai.
\newblock Effect of the fluid--structure interactions on low-density
  lipoprotein transport within a multi-layered arterial wall.
\newblock {\em Journal of biomechanics}, 45(2):371--381, 2012.

\bibitem{ciarlet1978finite}
P.~Ciarlet.
\newblock {\em The finite element method for elliptic problems}, volume~4.
\newblock North Holland, 1978.

\bibitem{cinthio2006longitudinal}
M.~Cinthio, {\AA}.R. Ahlgren, J.~Bergkvist, T.~Jansson, H.W. Persson, and
  K.~Lindstrom.
\newblock Longitudinal movements and resulting shear strain of the arterial
  wall.
\newblock {\em Am. J. Physiol. Heart Circ. Physiol.}, 291(1):H394--H402, 2006.

\bibitem{cinthio2005evaluation}
M.~Cinthio, A.R. Ahlgren, T.~Jansson, A.~Eriksson, H.W. Persson, and
  K.~Lindstrom.
\newblock Evaluation of an ultrasonic echo-tracking method for measurements of
  arterial wall movements in two dimensions.
\newblock {\em IEEE Trans. Ultrason. Ferroelectr. Freq. Control},
  52(8):1300--1311, 2005.

\bibitem{detournay1993fundamentals}
E.~Detournay and A.H.D. Cheng.
\newblock Fundamentals of poroelasticity, {C}hapter 5 in {C}omprehensive {R}ock
  {E}ngineering: {P}rinciples, {P}ractice and {P}rojects, {V}ol. {I}{I},
  {A}nalysis and {D}esign {M}ethod: 113--171, ed. {C}. {F}airhurst, 1993.

\bibitem{donea1983arbitrary}
J.~Donea.
\newblock {\em Arbitrary Lagrangian-Eulerian finite element methods, in:
  {C}omputational methods for transient analysis}.
\newblock North-{H}olland, {A}msterdam, 1983.

\bibitem{fernandez2013fully}
M.~A. Fern{\'a}ndez and M.~Landajuela.
\newblock A fully decoupled scheme for the interaction of a thin-walled
  structure with an incompressible fluid.
\newblock {\em Comptes Rendus Math{\'e}matique}, 2013.

\bibitem{fernandez2011incremental}
M.A. Fern{\'a}ndez.
\newblock Incremental displacement-correction schemes for the explicit coupling
  of a thin structure with an incompressible fluid.
\newblock {\em Comptes Rendus Math{\'e}matique}, 349(7):473--477, 2011.

\bibitem{formaggia2001coupling}
L.~Formaggia, J.F. Gerbeau, F.~Nobile, and A.~Quarteroni.
\newblock On the coupling of 3{D} and 1{D} {N}avier--{S}tokes equations for
  flow problems in compliant vessels.
\newblock {\em Comput. Methods Appl. Mech. Eng.}, 191(6-7):561--582, 2001.

\bibitem{fung1972bio}
Y.~C. Fung.
\newblock {\em Biomechanics: Its Foundation and Objectives, {Y}.{C}.{F}ung,
  {N}.{P}errone, {M}.Anliker ({E}ds)}.
\newblock Prentice-{H}all, {E}nglewood {C}liff, {N}{J}, 1972.

\bibitem{girault1979finite}
V.~Girault and P.-A. Raviart.
\newblock Finite element approximation of the {N}avier--{S}tokes equations.
\newblock {\em Lecture Notes in Mathematics}, 794, 1979.

\bibitem{glowinski2003finite}
R.~Glowinski.
\newblock {\em Finite element methods for incompressible viscous flow, in:
  {P}.{G}.{C}iarlet, {J}.-{L}.{L}ions ({E}ds), {H}andbook of numerical
  analysis}, volume~9.
\newblock North-{H}olland, {A}msterdam, 2003.

\bibitem{hansbo2005nitsche}
P.~Hansbo.
\newblock Nitsche’s method for interface problems in computational mechanics.
\newblock {\em GAMM-Mitt.}, 28(2):183--206, 2005.

\bibitem{holzapfel2009constitutive}
G.~A. Holzapfel and R.~W. Ogden.
\newblock Constitutive modelling of passive myocardium: a structurally based
  framework for material characterization.
\newblock {\em Philosophical Transactions of the Royal Society A: Mathematical,
  Physical and Engineering Sciences}, 367(1902):3445--3475, 2009.

\bibitem{holzapfel2004anisotropic}
G.~A. Holzapfel, G.~Sommer, and P.~Regitnig.
\newblock Anisotropic mechanical properties of tissue components in human
  atherosclerotic plaques.
\newblock {\em Transactions of the ASME-K-Journal of Biomechanical
  Engineering}, 126(5):657--665, 2004.

\bibitem{hughes1981lagrangian}
T.J.R. Hughes, W.K. Liu, and T.K. Zimmermann.
\newblock Lagrangian-{E}ulerian finite element formulation for incompressible
  viscous flows.
\newblock {\em Comput. Methods Appl. Mech. Eng.}, 29(3):329--349, 1981.

\bibitem{humphrey1995mechanics}
J.~D. Humphrey et~al.
\newblock Mechanics of the arterial wall: {R}eview and directions.
\newblock {\em Critical reviews in biomedical engineering}, 23(1-2):1, 1995.

\bibitem{kim2002fluid}
J.-M. Kim, S.-H. Chang, and C.-B. Yun.
\newblock Fluid-structure-soil interaction analysis of cylindrical liquid
  storage tanks subjected to horizontal earthquake loading.
\newblock {\em Structural Engineering and Mechanics}, 13(6):615--638, 2002.

\bibitem{layton2012long}
William Layton, Hoang Tran, and Xin Xiong.
\newblock Long time stability of four methods for splitting the evolutionary
  {S}tokes--{D}arcy problem into {S}tokes and {D}arcy subproblems.
\newblock {\em Journal of Computational and Applied Mathematics},
  236(13):3198--3217, 2012.

\bibitem{lee2008permeability}
K.~Lee, G.M. Saidel, and M.S. Penn.
\newblock Permeability change of arterial endothelium is an age-dependent
  function of lesion size in apolipoprotein e-null mice.
\newblock {\em American Journal of Physiology-Heart and Circulatory
  Physiology}, 295(6):H2273--H2279, 2008.

\bibitem{lesinigo2013lumped}
M.~Lesinigo.
\newblock {\em Lumped Mathematical Models for Intracranial Dynamics}.
\newblock PhD thesis, EPFL, Switzerland, 2013.

\bibitem{ma1992numerical}
X.~Ma, GC.. Lee, and S.G. Wu.
\newblock Numerical simulation for the propagation of nonlinear pulsatile waves
  in arteries.
\newblock {\em Journal of biomechanical engineering}, 114:490, 1992.

\bibitem{mikelic2012convergence}
A.~Mikeli{\'c} and M.~Wheeler.
\newblock Convergence of iterative coupling for coupled flow and geomechanics.
\newblock {\em Computational Geosciences}, pages 1--7, 2013.

\bibitem{miller2005computational}
L.A. Miller and C.S. Peskin.
\newblock A computational fluid dynamics of 'clap and fling' in the smallest
  insects.
\newblock {\em J. Exp. Biol.}, 208(2):195--212, 2005.

\bibitem{BorSunMulti}
B.~Muha and S.~\v{C}ani{\'c}.
\newblock Existence of a weak solution to a fluid-multi-layered-structure
  interaction problem.
\newblock {\em Submitted, arXiv:1305.5310}, 2013.

\bibitem{murad2001micromechanical}
M.~A. Murad, J.~N. Guerreiro, and A.~F.D. Loula.
\newblock Micromechanical computational modeling of secondary consolidation and
  hereditary creep in soils.
\newblock {\em Computer methods in applied mechanics and engineering},
  190(15):1985--2016, 2001.

\bibitem{nobile2001numerical}
F.~Nobile.
\newblock {\em Numerical approximation of ﬂuid–structure interaction
  problems with application to haemodynamics}.
\newblock PhD thesis, EPFL, Switzerland, 2001.

\bibitem{nobile2008effective}
F.~Nobile and C.~Vergara.
\newblock An effective fluid-structure interaction formulation for vascular
  dynamics by generalized {R}obin conditions.
\newblock {\em SIAM J. Sci. Comput.}, 30:731--763, 2008.

\bibitem{prosi2005mathematical}
M.~Prosi, P.~Zunino, K.~Perktold, and A.~Quarteroni.
\newblock Mathematical and numerical models for transfer of low-density
  lipoproteins through the arterial walls: a new methodology for the model set
  up with applications to the study of disturbed lumenal flow.
\newblock {\em Journal of biomechanics}, 38(4):903--917, 2005.

\bibitem{rausch2011material}
S.M.K. Rausch, C.~Martin, P.B. Bornemann, S.~Uhlig, and W.A. Wall.
\newblock Material model of lung parenchyma based on living precision-cut lung
  slice testing.
\newblock {\em Journal of the Mechanical Behavior of Biomedical Materials},
  4(4):583--592, 2011.

\bibitem{robertson2012structurally}
A.~Robertson, M.~Hill, and D.~Li.
\newblock Structurally motivated damage models for arterial walls. {T}heory and
  application.
\newblock In {\em Modeling of Physiological Flows}, pages 143--185. Springer,
  2012.

\bibitem{shan2012decoupling}
L.~Shan, H.~Zheng, and W.~Layton.
\newblock A decoupling method with different subdomain time steps for the
  nonstationary {S}tokes--{D}arcy model.
\newblock {\em Numerical Methods for Partial Differential Equations}, 2012.

\bibitem{showalter2010poroelastic}
{R}. Showalter.
\newblock Poroelastic filtration coupled to {S}tokes flow.
\newblock {\em Lecture Notes in Pure and Applied Mathematics}, 242:229--241,
  2010.

\bibitem{tang2011multi}
D.~Tang, C.~Yang, T.~Geva, G.~Gaudette, and P.~Del~Nido.
\newblock Multi-physics {MRI}-based two-layer fluid--structure interaction
  anisotropic models of human right and left ventricles with different patch
  materials: {C}ardiac function assessment and mechanical stress analysis.
\newblock {\em Computers \& structures}, 89(11):1059--1068, 2011.

\bibitem{temam2001navier}
R.~Temam.
\newblock {\em Navier-{S}tokes equations: theory and numerical analysis},
  volume 343.
\newblock Oxford University Press, 2001.

\bibitem{tully2009coupling}
B.~Tully and Y.~Ventikos.
\newblock Coupling poroelasticity and {CFD} for cerebrospinal fluid
  hydrodynamics.
\newblock {\em Biomedical Engineering, IEEE Transactions on}, 56(6):1644--1651,
  2009.

\bibitem{Vignon-Clementel20063776}
I.E. Vignon-Clementel, C.~Alberto~Figueroa, K.E. Jansen, and C.A. Taylor.
\newblock Outflow boundary conditions for three-dimensional finite element
  modeling of blood flow and pressure in arteries.
\newblock {\em Computer Methods in Applied Mechanics and Engineering},
  195(29-32):3776--3796, 2006.

\bibitem{vito2003blood}
R.~Vito and S.~Dixon.
\newblock Blood vessel constitutive models-1995-2002.
\newblock {\em Annual review of biomedical engineering}, 5(1):413--439, 2003.

\bibitem{zhou1997degree}
J.~Zhou and Y.C. Fung.
\newblock The degree of nonlinearity and anisotropy of blood vessel elasticity.
\newblock {\em Proceedings of the National Academy of Sciences},
  94(26):14255--14260, 1997.

\end{thebibliography}

\end{document}